\documentclass[final,onefignum,onetabnum]{siamart171218}
\usepackage{braket,amsfonts}
\usepackage{amsmath}
\usepackage{amssymb}
\usepackage{graphicx}
\usepackage{placeins}
\usepackage{longtable}
\usepackage{alltt}
\usepackage{algorithm}
\usepackage[utf8]{inputenc}
\usepackage{babel}
\usepackage[font=small,labelfont=bf]{caption}
\usepackage{array}
\usepackage{cancel}
\usepackage{longtable}

\usepackage[caption=false]{subfig}
\usepackage{pgfplots}

\allowdisplaybreaks

\newcommand{\tnorm}[1]{{\left\vert\kern-0.25ex\left\vert\kern-0.25ex\left\vert #1 
		\right\vert\kern-0.25ex\right\vert\kern-0.25ex\right\vert}}

\newsiamthm{claim}{Claim}
\newsiamremark{remark}{Remark}
\newsiamremark{hypothesis}{Hypothesis}
\crefname{hypothesis}{Hypothesis}{Hypotheses}

\usepackage{algorithmic}

\usepackage{graphicx,epstopdf}

\crefname{ALC@unique}{Line}{Lines}

\usepackage{amsopn}

\usepackage{xspace}
\usepackage{bold-extra}
\usepackage[most]{tcolorbox}

\colorlet{texcscolor}{blue!50!black}
\colorlet{texemcolor}{red!70!black}
\colorlet{texpreamble}{red!70!black}
\colorlet{codebackground}{black!25!white!25}

\allowdisplaybreaks

\newcommand{\R}{{\mathbb R}}

\def\epsilon{\varepsilon}

\def\s{\mathbf{s}}
\def\t{\mathbf{t}}
\renewcommand{\v}[1]{\underline{#1}}
\newcounter{ind}

\def\undertilde#1{{\baselineskip=0pt\vtop
{\hbox{$#1$}\hbox{$\scriptscriptstyle\sim$}}}{}}

\def\epsilon{\varepsilon}

\def\undertilde#1{{\baselineskip=0pt\vtop
{\hbox{$#1$}\hbox{$\scriptscriptstyle\sim$}}}{}}

\def\nabr1{\undertilde{\nabla}_{r_1}\,}
\def\nabr2{\undertilde{\nabla}_{r_2}\,}

\def\w{\mathbf{w}}
\def\u{\mathbf{u}}

\def\W{\mathrm{W}}

\def\f{\mathbf{f}}
\def\u{\mathbf{u}}
\def\v{\mathbf{v}}
\def\w{\mathbf{w}}
\def\z{\mathbf{z}}
\def\y{\mathbf{y}}
\def\n{\mathbf{n}}

\def\x{x}

\lstdefinestyle{siamlatex}{%
	style=tcblatex,
	texcsstyle=*\color{texcscolor},
	texcsstyle=[2]\color{texemcolor},
	keywordstyle=[2]\color{texemcolor},
	moretexcs={cref,cref,maketitle,mathcal,text,headers,email,url},
}

\tcbset{%
	colframe=black!75!white!75,
	coltitle=white,
	colback=codebackground, 
	colbacklower=white, 
	fonttitle=\bfseries,
	arc=0pt,outer arc=0pt,
	top=1pt,bottom=1pt,left=1mm,right=1mm,middle=1mm,boxsep=1mm,
	leftrule=0.3mm,rightrule=0.3mm,toprule=0.3mm,bottomrule=0.3mm,
	listing options={style=siamlatex}
}

\newtcblisting[use counter=example]{example}[2][]{%
	title={Example~\thetcbcounter: #2},#1}

\newtcbinputlisting[use counter=example]{\examplefile}[3][]{%
	title={Example~\thetcbcounter: #2},listing file={#3},#1}

\DeclareTotalTCBox{\code}{ v O{} }
{ 
	fontupper=\ttfamily\color{black},
	nobeforeafter,
	tcbox raise base,
	colback=codebackground,colframe=white,
	top=0pt,bottom=0pt,left=0mm,right=0mm,
	leftrule=0pt,rightrule=0pt,toprule=0mm,bottomrule=0mm,
	boxsep=0.5mm,
	#2}{#1}

\makeatletter
\newcommand*{\bdiv}{%
	\nonscript\mskip-\medmuskip\mkern5mu%
	\mathbin{\operator@font div}\penalty900\mkern5mu%
	\nonscript\mskip-\medmuskip
}
\makeatother

\patchcmd\newpage{\vfil}{}{}{}
\begin{tcbverbatimwrite}{tmp_\jobname_header.tex}
	\title{ Divergence conforming DG method for the optimal control of the Oseen equation with variable viscosity}
\end{tcbverbatimwrite}
	\title{ Divergence conforming DG method for the optimal control of the Oseen equation with variable viscosity}

\author{ Harpal Singh
\thanks{Department of Mathematics, Indian Institute of Technology Roorkee, Roorkee 247667, India. E-mail: {\tt harpal\_s@ma.iitr.ac.in}. HS is supported by Council of Scientific and Industrial Research (CSIR), India (Award no. 09/0143(16048)/2022-EMR-I).}
\and Arbaz Khan 
\thanks{ Department of Mathematics, Indian Institute of Technology Roorkee, Roorkee 247667, India. E-mail: {\tt arbaz@ma.iitr.ac.in}}
}
\ifpdf
\hypersetup{pdftitle={Guide to Using  SIAM'S \LaTeX\ Style} }
\fi

\begin{document}
\maketitle
\date{\today}
\begin{abstract}
This study introduces the divergence-conforming discontinuous Galerkin finite element method (DGFEM) for numerically approximating optimal control problems with distributed constraints, specifically those governed by stationary generalized Oseen equations. We provide optimal a priori error estimates in energy norms for such problems using the divergence-conforming DGFEM approach. Moreover, we thoroughly analyze $L^2$ error estimates for scenarios dominated by diffusion and convection. Additionally, we establish the new reliable and efficient a posteriori error estimators for the optimal control of the Oseen equation with variable viscosity. Theoretical findings are validated through numerical experiments conducted in both two and three dimensions.\end{abstract}


\section{Introduction} \label{Introduction.} \setcounter{equation}{0}
In the last few decades, the manipulation of fluid flow and the computation of physical properties through the identification of optimal control values have become a central focus in computational fluid dynamics (CFD) and various scientific and engineering disciplines. {Moving beyond the mere approximation of fluid variables like velocity, temperature, and viscosity, the focus now lies on directing the fluid towards a desired state or elucidating its properties under specific conditions.} These challenges in fluid flow are formulated as optimization problems, where constraint equations originate from the principles of fluid dynamics. Such problems are termed partial differential equation (PDE)-constrained optimal control problems, involving the simultaneous solution of governing PDEs and an optimization problem.
Exploring the regulation of flows encompasses a wide spectrum, addressing vital aspects such as weather control and practical applications like managing blood flow \cite{GTJ}, aeronautics \cite{GMD}, optimizing vehicle aerodynamics \cite{HWG}, wastewater treatment \cite{ACM}, modeling groundwater reserves and reservoir simulation \cite{JAI}, and mitigating air pollution through the regulation of industrial pollutant emissions to maintain acceptable pollution levels \cite{ZJQ}, as well as drag reduction.  
Finite element methods (FEM) emerge as precise and efficient approaches, employing weighted residuals to integrate partial differential equations over an element by multiplying weight functions, commonly known as shape functions. A notable advantage lies in their adaptability for use on complex geometries without incurring additional expenses.

The examination of optimal control problems governed by specific PDEs poses a significant challenge due to the interdependence of control, state, and co-state equations. To address this challenge, various numerical strategies have been devised and continually refined. Initially,  Falk \cite{FRO} and Geveci \cite{GTO} explored conforming  
FEM discretizations for optimal control problems. {They utilized piecewise linear polynomials for discretizing state and co-state variables, while control variable was discretized using piecewise constant polynomial.} In pursuit of a more accurate approximation of the control, Casas et al. \cite{CFU} introduced piecewise linear approximation for control and provided a comprehensive convergence analysis. For control variables, convergence rates of $O(h)$ and $O(h^{\frac{3}{2}})$ were established for piecewise constant and linear discretizations, respectively. Hinze introduced a variational discretization approach to enhance the convergence rate to $O(h^{2})$ in \cite{HVD}.
\subsection{Optimal control problem}
An illustration of fluid flow challenges involves the general-ized Oseen equations, a commonly employed linearization of the Navier Stokes equations describing the low-Reynolds-number flow of viscous, incompressible fluids. Numerous well-known problems, such as minimizing the drag coefficient of a body in motion relative to a fluid and designing airfoils in aerodynamics, constitute applications of optimal control problems that incorporate the Oseen equations.
Consider $\Omega$, an open, bounded Lipschitz polygon in $\R^d$, with $\Gamma$ representing its boundary, defined as $\partial\Omega$. 
For a given regularization (or control cost) parameter $\lambda>0$ and a desired velocity field  $\mathbf{y}_d \in [L^{2}(\Omega)]^d$, define
\begin{align}\label{1.1}
  J(\mathbf{y,u}) = \frac{1}{2} \|\mathbf{y}-\mathbf{y}_d\|_{[L^{2}(\Omega)]^d}^2
	 +\frac{\lambda }{2}  \|\mathbf{u}\|_{[L^{2}(\Omega)]^d}^2.
\end{align}
We aim to address an optimization challenge, specifically, the minimization of the functional $J$ mentioned above, while adhering 
to the steady-state generalized Oseen equations with control. Within the domain $\Omega \subset \R^d$, our objective is to determine a velocity 
field denoted as $\y$ and a pressure denoted as $p$, satisfying the following equations:
\begin{align}\label{1.2}
\begin{cases}
-\nabla\cdot( \nu(x)\nabla \mathbf{y}) + (\boldsymbol{\beta}\cdot\nabla)\mathbf{y} + \sigma \mathbf{y}  + \nabla p &= \mathbf{f} + \mathbf{u} \ \ \text{in} \  \Omega,\\
\hspace{4cm}\nabla \cdot\mathbf{y} &= 0 \ \ \ \  \ \ \ \ \text{in} \  \Omega,\\
\hspace{4.4cm} \mathbf{y} &= 0 \ \ \ \  \ \ \ \ \text{on} \ \Gamma,\\ 
\end{cases} 
\end{align}
with the control constraints 
\begin{align}\label{1.3}
\mathbf{u}_a(x) \le \mathbf{u}(x) \le \mathbf{u}_b(x) \ \text{for} \ \text{a.e.} \  x \in \Omega,
\end{align}
where $\mathbf{u}_a, \mathbf{u}_b \in \R^d$ with $\mathbf{u}_a < \mathbf{u}_b$ (considered componentwise), $\nu\in L^\infty(\Omega)$ is the viscosity coefficient, $\boldsymbol{\beta}: \Omega \rightarrow \R^d$ is a vector field (belonging to $W^{1,\infty}(\Omega)$), $\sigma$ is a positive scalar function, $\mathbf{\f} \in [L^{2}(\Omega)]^d$ is given external body force and $\mathbf{\u} \in [L^{2}(\Omega)]^d$ is the control. 
Additionally, the parameters $\nu$, $\sigma$, and $\boldsymbol{\beta}$ adhere to the following assumptions:
\begin{align}
                          0 < \nu_0 < \nu(x) < \nu_1, \quad 
	\label{A1}	\sigma(\x) -\frac{1}{2} \nabla \cdot \boldsymbol{\beta}(x) &\ge \kappa, \quad \forall \,\x \in \Omega, 
	\quad \|\sigma - \nabla \cdot \boldsymbol{\beta}\|_{L^{\infty}(\Omega)} \le \kappa \xi,
\end{align}
where $\nu_0$, $\nu_1$, $\kappa$ and $\xi$ are positive constants. 
\subsection{Literature Review}
The equations described above are relevant in several situations~\cite{VANAYA, JVK, PLE, RJS}, such as:
\begin{itemize}
\item The linearization of non-Newtonian flow problems, where the viscosity is not constant but depends on the flow rate.
 \item Applications in various fields where the viscosity of a fluid can vary depending on factors like temperature, concentration, or the presence of different materials within the fluid.
\item Scenarios where the patterns of fluid flow are significantly influenced by the spatial distribution of the viscosity within the fluid.
\end{itemize}
In the context of optimal control of Oseen equation, there are very few result avialable. Moreover, these results only focus for Oseen equation with constant viscosity. Braack et al. \cite{braack2009optimal} introduced stabilized FEM for linear quadratic optimal control problems involving Oseen equations and derived a priori estimates. Subsequently, Allendes et al. \cite{AAE} established globally reliable a posteriori error estimates for the same. As far as we are aware there is no result for optimal control with variable viscosity. 
However, discontinuous approximations are often preferred to preserve physically meaningful properties, especially when dealing with models characterized by uneven coefficients and anticipating sharp solutions. Upon comprehensive review of the existing literature, it is noteworthy that the utilization of Discontinuous Galerkin Finite Element Methods (DGFEM) appears to be underexplored in the context of optimal control problems.  
The utilization of adaptive DGFEM for addressing optimal control problems governed by convection-diffusion equations is explored in detail in \cite{yucel2015adaptive}. Dond et al. \cite{ATR} investigated discontinuous finite element methods for optimal control problems governed by the Stokes equation and provided a comprehensive convergence analysis, considering minimal regularity assumptions.
To exploit the conservative properties of Finite Volume Methods (FVMs) and the advantages of DG methods, authors in \cite{LXC, SRK} employed DGFVM for approximating elliptic optimal control problems with a primal-dual active set strategy established in \cite{PDS}. Later, this analysis was extended to the optimal control problem governed by Brinkman equations \cite{SRR}.

The divergence-conforming discontinuous Galerkin method, pioneered by Cockburn,  Kanschat, and Sch$\ddot{o}$tzau \cite{cockburn2007note}, has emerged as a powerful tool for tackling challenging fluid dynamics problems. It's core strength lies in ensuring the crucial property of divergence-free velocity fields, a critical attribute for incompressible flows. This leads to superior accuracy and stability compared to traditional DG approaches. However, the classical divergence-conforming methods often rely on penalty terms that can introduce stiffness and limit their applicability to complex geometries. Subsequently, Khan et al. \cite{AGO} proposed the new robust a posteriori error estimator for Oseen equation. 

\subsection{Main contribution:} 
This study pioneers the application of the divergence-conforming DG method to optimal control problems governed by the Oseen equation with variable viscosity, marking uncharted territory in this field. We are the first to explore this approach in both two and three dimensions. 

\begin{itemize}
\item{\textbf{A priori analysis:}} 
The primary contributions of this work include establishing the well-posedness of the discrete problem and deriving optimal a priori error estimates for the control, state, and adjoint variables within their respective inherent norms. Furthermore, we conduct a comprehensive examination of $L^2$ error estimates for scenarios dominated by both diffusion and convection.
\item{\textbf{A posteriori analysis:}} 
Our second major contribution lies in establishing a foundational framework for this innovative direction by deriving the new a posteriori error estimates. We demonstrate the effectiveness of our estimator through rigorous theoretical analysis and practical error estimation techniques. These findings open up avenues for further advancements and applications in optimal control problems involving variable-viscosity fluid flow. Specifically, in scenarios where the convective velocity $\beta$, coefficient $\sigma$, and domain $\Omega$ are all of order one, and the viscosity $\nu$ is constant, the inverse of viscosity $\nu^{-1}$ serves as the Reynolds number. Furthermore, our estimator showcases robustness concerning the Reynolds number $\nu^{-1}$.
\item{\textbf{Unified analysis:}} 
We offer a comprehensive analysis of optimal control problems governed by Stokes ($\boldsymbol{\beta}=0$, $\sigma=0$), Brinkman ($\boldsymbol{\beta}=0$), Oseen ($\sigma = 0$), and generalized Oseen equations with variable viscosity by employing divergence-conforming Discontinuous Galerkin method.
\end{itemize}
\subsection{Outline of paper} The remainder of the document is structured as follows: In Section \ref{Preliminaries and continuous formulation.}, we present notations and examine the existence and uniqueness of the continuous formulation and derive the necessary and sufficient optimality conditions  for the model problem. Section \ref{Discrete Formulation.} introduces the divergence-conforming DG discrete formulation for the control problem. Sections \ref{PRIORI ERROR ESTIMATES.} and \ref{A POSTERIORI ERROR ESTIMATES.} are dedicated to deriving a priori and a posteriori error estimates, respectively. Lastly, Section \ref{Numerical Experiments.} provides numerical examples to validate the theoretical results.
\section{Preliminaries and continuous formulation} \label{Preliminaries and continuous formulation.} \setcounter{equation}{0}
\subsection{Function spaces.} 
The notation $W^{k,p}(\Omega)$ represents the standard Sobolev spaces applicable to scalar-valued functions, characterized by norms
$\|\cdot\|_{W^{k,p}(\Omega)}$, where $k$ is non-negative integer and  $p$ lies in the range $1$ to $\infty$.
When $k=0$, we denote  $W^{0,p}(\Omega)=L^p(\Omega)$ with norm $\|\cdot\|_{p}$. Additionally, when $p$ is equal to $2$, we express  
$W^{k,2}(\Omega)=H^{k}(\Omega)$. 
Bold letters are employed to represent vector-valued counterparts of the aforementioned spaces. 
Finally, we present the spaces required for our analysis in the following manner:
\begin{align*}
\nonumber L_{0}^{2}(\Omega) &:= \biggl\{u \in L^2(\Omega) \ | \int_{\Omega} u \ dx = 0 \biggl\},\;
\nonumber &&\boldsymbol{H}_{0}^{1}(\Omega) := \{\mathbf{u} \in \boldsymbol{H}^{1}(\Omega) \ | \  \mathbf{u}|_{\Gamma} = 0 \},\\
\nonumber \boldsymbol{H}^{\text{div}}(\Omega) &:= \{\mathbf{u} \in \boldsymbol{L}^2(\Omega) \ | \ \nabla\cdot\mathbf{u} \in L^2(\Omega) \},\;
\nonumber &&\boldsymbol{H}_{0}^{\text{div}}(\Omega) := \{\mathbf{u} \in \boldsymbol{H}^{\text{div}}(\Omega) \ | \ \mathbf{u}\cdot\mathbf{n} = 0 \ \text{on} \ \partial \Omega \}.
\end{align*} 
	 The notation $m \precsim n$ means that there exists a positive constant $C$ independent of $m$ and $n$ such that $m \le C n.$
 For simplicity, we write 
$\boldsymbol{V} = \boldsymbol{H}_{0}^{1}(\Omega)$ and $Q = L_{0}^{2}(\Omega)$. 

\subsection{Continuous formulation}
Let's delve into the continuous representation of the optimal control problem expressed in equations (\ref{1.1}-\ref{1.3}) and establish an optimality system. Our goal is to find $(\mathbf{y}, p) \in \boldsymbol{V} \times Q$ that meets the weak formulation of the state problem (\ref{1.2}), where:
\begin{align}\label{2.1}
	&\begin{cases} 
		a(\mathbf{y},\mathbf{v}) + b(\mathbf{v},p) &= (\mathbf{f} + \mathbf{u}, \mathbf{v}) \hspace{3.86cm}  \forall \ \mathbf{v} \in \boldsymbol{V}, \\
		\hspace{1.47cm}b(\mathbf{y},\phi) &= 0 \hspace{5.15cm}   \forall \ \phi \in Q,
	\end{cases}
\end{align}
where the bilinear forms  $a : \boldsymbol{V} \times \boldsymbol{V} \rightarrow \R$ and $ b : \boldsymbol{V} \times Q \rightarrow \R$ are defined by
\begin{align}
\label{2.2} a(\mathbf{y},\mathbf{v}) &:= (\nu \nabla \mathbf{y},\nabla \mathbf{v}) + \frac{1}{2}((\boldsymbol{\beta} \cdot \nabla)\y,\v) - \frac{1}{2}((\boldsymbol{\beta} \cdot \nabla)\v,\y) + \bigg(\Big(\sigma - \frac{1}{2} \nabla \cdot \boldsymbol{\beta}\Big) \y,\v \bigg), \\
\label{2.3} b(\mathbf{y},\phi) &:= -(\phi, \nabla \cdot \mathbf{v}).
\end{align}
The norms on $\boldsymbol{V}$ and $Q$ are represented as $|\!|\!|\cdot|\!|\!|_{\boldsymbol{V}}$ and $|\!|\!|\cdot|\!|\!|_{Q}$, respectively. For any $\mathbf{v} \in \boldsymbol{V}$ and $\phi \in Q$, we express this as follows:
$$|\!|\!|\mathbf{v}|\!|\!|^{2} :=  \|\nu^{1/2}\nabla \mathbf{v}\|_{ {\boldsymbol{L}}^{2}(\Omega)}^2 + \kappa \|\mathbf{v}\|_{\boldsymbol{L}^{2}(\Omega)}^2  \ \ \ \ \text{and} \ \ \  \ |\!|\!|\phi|\!|\!|_{Q}^2 := \|\nu^{-1/2}\phi\|_{L^2(\Omega)}^2.$$
The well-posedness of the state problem (\ref{2.1}) has been established in previous works \cite{TPF, GFB}. By applying de Rham's Theorem \cite[Section 4.1.3 and Theorem ~ B73]{TPF}, the equivalent formulation of (\ref{2.1}) is stated as follows: find $\mathbf{y} \in \boldsymbol{V}_{0}$ such that
\begin{align}\label{2.7} 
	a(\mathbf{y},\mathbf{v}) &= (\mathbf{f} + \mathbf{u}, \mathbf{v}) \hspace{1cm}  \forall  \ \mathbf{v} \in \boldsymbol{V}_{0},
\end{align}
where \ $\boldsymbol{V}_{0} := \{\mathbf{v} \in \boldsymbol{V}: \nabla \cdot \mathbf{v} = 0\}.$ 
	The \textbf{admissible set} of controls is defined as 
	\begin{align*}
			\boldsymbol{U}_{ad} := \{\mathbf{u} \in \boldsymbol{L}^{2}(\Omega) \ | \ \mathbf{u}_a(\x) \le \mathbf{u}(\x) \le \mathbf{u}_b(\x)  \ \text{for} \ a.e. \  \x \in \Omega \},
		\end{align*}
		where $\mathbf{u}_a, \mathbf{u}_b \in \R^d$, with $\mathbf{u}_a < \mathbf{u}_b$ considered componentwise.
   The set $\boldsymbol{U}_{ad}$ is a nonempty, bounded, convex, closed subset of the reflexive Banach space $\boldsymbol{L}^{2}(\Omega)$, and it is weakly sequentially compact as proven in \cite[Theorem~ 2.11]{FTO}. To analyze the optimal control problem, we adopt the notations and results presented in \cite{FTO}. The control-to-state map $S : \boldsymbol{L}^{2}(\Omega) \rightarrow \boldsymbol{L}^{2}(\Omega)$ associates the state $\mathbf{y}$ with a given control $\mathbf{u} \in \boldsymbol{U}_{ad}$.
We introduce the reduced functional $F:\boldsymbol{L}^{2}(\Omega) \rightarrow \R$ as follows:
\begin{align}\label{2.8}
	F(\mathbf{u}) := \frac{1}{2} \|S(\mathbf{u})-\mathbf{y}_d\|_{\boldsymbol{L}^{2}(\Omega)}^2 +\frac{\lambda }{2}  \|\mathbf{u}\|_{\boldsymbol{L}^{2}(\Omega)}^2.
\end{align}
The functional $F$ defined over a Banach space satisfies the conditions of continuity and convexity. These characteristics lead to the weakly lower semicontinuity of $F$ as established in \cite[Theorem~ 2.12]{FTO}. Furthermore, $F$ is strictly convex for $\lambda >0$.
As a consequence of the above properties, the problem 
\begin{align}\label{2.9}
	\underset{\mathbf{u} \in \boldsymbol{U}_{ad}}{\min} F(\u)\quad \mbox{subject to (\ref{2.7})} ,
\end{align} has an unique optimal solution $\widetilde{\mathbf{u}}$ \cite[Theorem~2.14]{FTO} and corresponding optimal state $\widetilde{\mathbf{y}}$ that satisfies (\ref{2.7}), or equivalently (\ref{2.1}). 
The existence of a pressure state $\widetilde{p}$ that satisfies equation (\ref{2.1}) is guaranteed by de Rham's Theorem. Leveraging the Gateaux differentiability of $F$, the optimal solution $\widetilde{\mathbf{u}}$ of problem (\ref{2.9}) meets the first-order necessary optimality condition, which is also known as a variational inequality:
\begin{align}\label{2.10}
	F'(\widetilde{\mathbf{u}})(\mathbf{u}-\widetilde{\mathbf{u}}) \ge 0 \quad \forall \ \mathbf{u} \in \boldsymbol{U}_{ad},
\end{align}
 proved in \cite[Lemma~2.21]{FTO}.  By finding Fréchet derivative of the reduced functional, the variational inequality becomes
 	\begin{align}\label{2.10n}
 		\big(\boldsymbol{S}^{*}(\boldsymbol{S} \widetilde{\u}-\y_d)+ \lambda \widetilde{\u},\mathbf{u}-\tilde{\mathbf{u}}\big) \ge 0 \quad \forall \ \mathbf{u} \ \in \boldsymbol{U}_{ad},
 	\end{align}
  by \cite[Theorem~2.22]{FTO}, where $\boldsymbol{S}^{*}$ denotes the adjoint operator. The variational inequality (\ref{2.10}) is further reduced to 
 \begin{align}\label{2.11}
 	(\widetilde{\mathbf{w}} + \lambda \widetilde{\mathbf{u}},\mathbf{u}-\widetilde{\mathbf{u}}) \ge 0 \quad \forall \ \mathbf{u} \ \in \boldsymbol{U}_{ad},
 \end{align}
\cite[Theorem~2.25]{FTO}, where $\widetilde{\mathbf{w}}$ is the unique solution to the adjoint problem: find $(\mathbf{w},r) \in \boldsymbol{V} \times Q$ such that
\begin{align}\label{2.12}
	&\begin{cases} 
		a(\mathbf{z},\mathbf{w}) - b(\mathbf{z},r) &= (\mathbf{y} - \mathbf{y}_d , \mathbf{z})_{\boldsymbol{L}^{2}(\Omega)} \hspace{2.89cm}  \quad \forall  \ \mathbf{z} \in \boldsymbol{V}, \\
		\hspace{1.45cm}b(\mathbf{w},\psi) &= 0 \hspace{5.15cm}  \quad \forall \ \psi \in Q.
	\end{cases}
\end{align}
From the equations (\ref{2.1}), (\ref{2.11}-\ref{2.12}), the optimality system can be expressed as follows: 

An optimal solution for the optimal control problem (\ref{1.1}-\ref{1.3}) is characterized by $(\mathbf{y}, p, \mathbf{u}) \in \boldsymbol{V} \times Q \times \boldsymbol{U}_{ad}$ if and only if $(\mathbf{y}, p, \w, r, \mathbf{u}) \in \boldsymbol{V} \times Q \times \boldsymbol{V} \times Q \times \boldsymbol{U}_{ad}$ satisfies the system.
\begin{align} 
	\label{2.13a} a(\mathbf{y},\mathbf{v}) + b(\mathbf{v},p) &= ( \f + \u, \mathbf{v} ) \hspace{4.6cm}  &&\forall \ \mathbf{v} \in  \boldsymbol{V}, \\
\label{2.13b} 	\hspace{1.57cm}b(\y,\phi) &= 0 \hspace{5.89cm} \  &&\forall \ \phi \in Q,\\
\label{2.13c} 	a(\z,\w) - b(\mathbf{z},r) &= ( \y - \mathbf{y}_d , \mathbf{z}) \    \ \hspace{4.2cm} &&\forall \ \mathbf{z} \in  \boldsymbol{V}, \\
\label{2.13d}	\hspace{1.56cm}b(\w,\psi) &= 0 \   \hspace{5.95cm}  &&\forall \ \psi \in Q,\\
\label{2.13e} 	(\w  + \lambda \u,\overline{\u}-\u) &\ge 0 \hspace{6.1cm} &&\forall \ \overline{\u} \ \in \boldsymbol{U}_{ad}.
\end{align} 
Utilizing the optimal control variable's projection formula \cite[Theorem 2.28]{FTO}, the variational inequality (\ref{2.13e}) in the optimality system can be expressed equivalently as:
\begin{align}\label{2.14} 
	\u = \Pi_{[\mathbf{u}_a,\mathbf{u}_b]} \bigg(-\frac{\w}{\lambda}\bigg) \ \ \ a.e. \ \text{in} \ \Omega,
\end{align}
where 
$\Pi$ represents a projection defined as:
\[\Pi_{[\mathbf{u}_a,\mathbf{u}_b]}(\mathbf{v})(\x) := \min \bigl\{\mathbf{u}_b(\x), \max \{\mathbf{u}_a(\x),\mathbf{v}(\x)\} \bigr\}.\]
\subsection{Second-order sufficient optimality conditions}
In conducting a numerical analysis of the problem and evaluating optimization algorithms, we establish the second-order sufficient optimality conditions (\textbf{SSC}) by referencing \cite{FTSS}. 
Prior to delving into the primary outcome, we lay the groundwork with several key definitions.
\begin{definition}
       A control $\tilde{\u} \in \boldsymbol{U}_{ad}$ is deemed \textbf{locally optimal} in $\boldsymbol{L}^{2}(\Omega)$ if there exists a positive constant $\epsilon$ such that the inequality
	\begin{align*}
		J(\tilde{\y},\tilde{\u}) \le J(\y,\u) 
	\end{align*}
holds for all $\u \in \boldsymbol{U}_{ad}$ with $|\tilde{\u}-\u|_{0}^{2} \le \epsilon$. In this context, $\tilde{\y}$ and $\y$ represent the corresponding states associated with the controls $\tilde{\u}$ and $\u$, respectively.
\end{definition}
\begin{definition}
	A pair $(\tilde{\y},\tilde{\u}) \in \boldsymbol{V} \times \boldsymbol{U}_{ad}$ is considered a \textbf{globally optimal solution} to the optimal control problem (\ref{1.1}-\ref{1.3}) if
	\begin{align*}
		J(\tilde{\y},\tilde{\u}) = \underset{(\y,\u) \in \boldsymbol{V} \times \boldsymbol{U}_{ad}}{\min} J(\y,\u).
	\end{align*}
\end{definition}
{
\noindent For $\boldsymbol{X} := \boldsymbol{V} \times Q$, we define $\boldsymbol{\mathcal{A}} : \boldsymbol{X} \times \boldsymbol{X} \rightarrow \R$ as 
\begin{align}
\nonumber	\boldsymbol{\mathcal{A}}(\y,\z) &= (\nu \nabla \mathbf{y}^{v},\nabla \z^{v}) + \frac{1}{2}((\boldsymbol{\beta} \cdot \nabla)\y^{v},\z^{v}) - \frac{1}{2}((\boldsymbol{\beta} \cdot \nabla)\z^{v},\y^{v})+ \bigg(\Big(\sigma - \frac{1}{2} \nabla \cdot \boldsymbol{\beta}\Big) \y^{v},\z^{v} \bigg)\\
\label{2.18} & \ \ \ \ - (\y^{p}, \nabla \cdot \z^{v}) + (\z^{p},\nabla \cdot \y^{v}),
\end{align}
\noindent where $\z = (\z^v,\z^p), \ \y = (\y^{v},\y^{p}) \in \boldsymbol{X}$.} 
We define the Lagrange function denoted by $\mathcal{L}: \boldsymbol{X} \times \boldsymbol{L}^{2}(\Omega) \times \boldsymbol{X} \rightarrow \R$ , for the optimal control problem as 
\begin{align}\label{2.16} 
	\mathcal{L}(\y,\u,\z) = \ &J(\y,\u) - \mathcal{A}(\y,\z) + (\f+\u,\z^{v}).
\end{align}
This function is twice Fr\'echet-differentiable, as demonstrated in Lemma-\ref{lem: 2.4.}, with respect to both $\y$ and $\u$. 
The expression of the first-order necessary conditions can also be articulated as:
\begin{alignat*}{2}
	\mathcal{L}_{\y}(\tilde{\y},\tilde{\u},\tilde{\z}) \s &= 0 \quad && \forall \ \s \in \boldsymbol{V},\\
	\mathcal{L}_{\u}(\tilde{\y},\tilde{\u},\tilde{\z}) (\u-\tilde{\u}) &\ge 0 \quad && \forall \ \u \in \boldsymbol{U}_{ad},
\end{alignat*}
where $\mathcal{L}_{\y}, \mathcal{L}_{\u}$ denote the partial Fr\'echet-derivative of $\mathcal{L}$ w.r.t. $\y$, $\u$, respectively.
\begin{lemma}\label{lem: 2.4.}
        The Lagrangian $\mathcal{L}$, as defined in (\ref{2.16}),  is twice Fr\'echet-differentiable with respect to $\v=(\y,\u)$.         
        The second-order derivative at $\v = (\tilde{\y},\tilde{\u})$ in conjunction with the associated adjoint state $\tilde{\z}$ satisfies the following condition:
	\begin{align}
	\label{2.17a} 	\mathcal{L}_{\v \v} (\tilde{\v},\tilde{\z})[(\s_1,\t_1),(\s_2,\t_2)] &= \mathcal{L}_{\u \u} (\tilde{\v},\tilde{\z})[\t_1,\t_2] + \mathcal{L}_{\y \y} (\tilde{\v},\tilde{\z})[\s_1,\s_2],\\
	\label{2.17b} 	|\mathcal{L}_{\y \y} (\tilde{\v},\tilde{\z})[\s_1,\s_2]| &\le C_{\mathcal{L}} |\s_1| |\s_2|
	\end{align}
for all $(\s_i,\t_i) \in \boldsymbol{V} \times \boldsymbol{L}^{2}(\Omega)$ where $i=1,2$ , with some positive constant $C_{\mathcal{L}}$ independent of $\tilde{\v},\s_1$ and $\s_2.$
\end{lemma}
\begin{proof}
	The first order derivatives of $\mathcal{L}$ w.r.t. $\y$ and $\u$ are
	\begin{align*}
		\mathcal{L}_{\y} (\tilde{\v},\tilde{\z}) \s &= (\s,\tilde{\y}-\y_{d}) - \mathcal{A}(\s,\tilde{\z}),\\
		\mathcal{L}_{\u} (\tilde{\v},\tilde{\w}) \t &= \lambda (\t,\tilde{\u}) + (\u,\tilde{\z}).
	\end{align*}
	The mappings $\tilde{\y} \mapsto \mathcal{L}_{\y}(\tilde{\v}, \tilde{\w})$ and $\tilde{\u} \mapsto \mathcal{L}_{\u}(\tilde{\w}, \tilde{\z})$ exhibit an affine linear structure with bounded linear components, ensuring continuity. Consequently, both mappings are Fréchet-differentiable. This observation establishes that $\mathcal{L}$ is twice Fréchet-differentiable. The second-order derivative of $\mathcal{L}$ with respect to $\v$ is then expressed as:
\begin{align*}
	\mathcal{L}_{\v \v}(\tilde{\v}, \tilde{\w})[(\s_1,\t_1),(\s_2,\t_2)]  &= \mathcal{L}_{\u \u}(\tilde{\v}, \tilde{\w})[\t_1,\t_2]  + \mathcal{L}_{\y \y}(\tilde{\v}, \tilde{\w})[\s_1,\s_2] \\
	&= \lambda (\t_1,\t_2) + (\s_1,\s_2).
\end{align*}
Given the absence of mixed derivatives, the second estimate (\ref{2.17b}) is derived through the application of the Cauchy-Schwarz inequality.
\end{proof}
To shorten the notations, we abbreviate $[\v,\v]$ by $[\v]^{2}$, $i.e.$
\begin{align*}
	\mathcal{L}_{\v \v}({\v}, {\w})[\s,\t]^{2} = \mathcal{L}_{\v \v}({\v}, {\w})[(\s,\t),(\s,\t)],
\end{align*}
\begin{definition}
	For fixed $\epsilon>0$ and all $i=1,\ldots,d$, we define the \textbf{strongly active sets} $\Omega_{\epsilon,i}$ as 
	\begin{align*}
		\Omega_{\epsilon,i} = \{ x \in \Omega: |\lambda \tilde{u}_{i}(x) + \tilde{w}_{i}(x)| > \epsilon\}.
	\end{align*}
\end{definition}
\noindent 
Let's consider the optimal pair $\tilde{\v} = (\tilde{\y}, \tilde{\u})$ and the corresponding co-state $\tilde{\w}$. Assume that the following condition on $\mathcal{L}_{\v\v}(\tilde{\v}, \tilde{\w})$ holds, referred to as the \textbf{second-order sufficient optimality condition (SSC)}:
There exists $\epsilon>0$ and $\delta>0$ such that
{\begin{align}\label{2.18n} 
	\mathcal{L}_{\v \v}(\tilde{\v},\tilde{\w})[(\s,\mathbf{t})]^2 \geq \delta \|\mathbf{t}\|_{0}^2
\end{align}}
holds for all pairs $(\s, \mathbf{t}) \in \boldsymbol{V} \times \boldsymbol{L}^{2}(\Omega)$ with
$$\mathbf{t} = \u - \tilde{\u}, \quad \u \in \boldsymbol{U}_{ad}, \quad t_i = 0 \text{ on } \Omega_{\epsilon,i} \text{ for } i = 1, \ldots, d,$$
$\s \in \boldsymbol{V}$ is the weak solution of the equation 
\begin{align*}
	a(\z,\s) = (\mathbf{t},\z) \quad \forall \ \z \in \boldsymbol{V}.
\end{align*}
\begin{remark}
The definition of $\t$ indicates that $\t(x) \ge 0$ when $\tilde{\u} = \u_a$ and $\t(x) \le 0$ when $\tilde{\u} = \u_b$. It is important to note that the condition $\epsilon > 0$ cannot be eased to $\epsilon = 0$, as demonstrated in the counterexample provided in \cite{DND}. Additionally, the presence of a penalty term is crucial for the expectation of second-order sufficient conditions; this aligns with the scenario where $\lambda = 0$.
\end{remark}
\noindent 
The conclusion that the combination of the \textbf{second-order sufficient optimality condition (SSC)} and the \textbf{first-order necessary optimality conditions} is adequate for establishing the local optimality of $(\tilde{\y}, \tilde{\u})$ is derived in \cite[Theorem~ 3.17]{FTSS}.
\section{Discrete Formulation} \label{Discrete Formulation.} \setcounter{equation}{0}
\subsection{Discretization, Finite element spaces and traces} 
Firstly, we introduce the notations related to the discretization of the domain $\Omega$. Let $\mathcal{T}_h = {K}$ be a shape-regular partition of $\overline{\Omega}$ into closed triangles (or tetrahedra if $d=3$) following the criteria in \cite{MLF}, such that $\bigcup_{K \in \mathcal{T}_h} K = \overline{\Omega}.$ The global mesh size is denoted as $h = \max \{h_K : K \in \mathcal{T}_h \}$, where $h_K$ represents the diameter of an element $K$. The sets $\mathcal{E}^{i}(\mathcal{T}_h)$, $\mathcal{E}^{b}(\mathcal{T}_h)$, and $\mathcal{E}(\mathcal{T}_h)$ consist of interior edges, boundary edges, and all edges of $\mathcal{T}_h$, respectively. Additionally, let $h_e$ and $\boldsymbol{n}_e$ denote the length and unit exterior normal vector of an edge $e$. The integrals over meshes and faces are defined as:
\begin{align*}
	(f,g)_{\mathcal{T}_h} &:= \sum_{K \in \mathcal{T}_h}(f,g)_K, \ \text{and} \ \   \langle f,g \rangle_{\mathcal{E}(\mathcal{T}_h)} := \sum_{E \in \mathcal{E}(\mathcal{T}_h)} \langle  f,g \rangle_E.
\end{align*}
The norms on $\mathcal{T}_h$ and $\mathcal{E}(\mathcal{T}_h)$ are defined as:
$$\|f\|_{\mathcal{T}_h} := \sqrt{(f,f)_{\mathcal{T}_h}}, \ \ \ \ \text{and} \ \ \ \ \|f\|_{\mathcal{E}(\mathcal{T}_h)} := \sqrt{\langle f,f \rangle _{\mathcal{E}(\mathcal{T}_h)}}.$$
For $s \ge 0$, we define the broken Sobolev spaces as
$$H^{s}(\mathcal{T}_h) = \{v \in L^2(\Omega) : v|_K \in H^{s}(K) \ \text{for all} \ K \in \mathcal{T}_h\}.$$
Consider two triangles $K^{+}$ and $K^{-}$ in the mesh that share an edge $E \in \mathcal{E}(\mathcal{T}_h)$. Suppose $u \in H^{1}(\mathcal{T}_h)$ has traces $u^{+}$ and $u^{-}$ on $E$ from the two elements $K^{+}$ and $K^{-}$, respectively. In this context, we define the average operator $\{\!\!\{\cdot\}\!\!\}$ as 
$	\{\!\!\{u\}\!\!\} := \frac{u^{+}+u^{-}}{2}$.
Let $\mathbf{n}^{+}$ and $\mathbf{n}^{+}$ be be the outward unit normal vectors to $K^{+}$ and $K^{-}$, respectively. 
The jumps $[\![ \cdot ]\!]$ across the edge $E$ are defined as:
\begin{align*}
[\![u \mathbf{n}]\!] := u^{+} \mathbf{n}^{+} + u^{-} \mathbf{n}^{-} := (u^{+}-u^{-})\mathbf{n^{+}}, \ \ \ 
[\![u \otimes \mathbf{n}]\!] = u^{+} \otimes \mathbf{n}^{+} + u^{-} \otimes \mathbf{n}^{-}, \ \ \
[\![u]\!] = (u^{+}-u^{-}).
\end{align*}
For an edge $e$ lying on the boundary $\Gamma$, we have $[\![\mathbf{u}]\!]_{e} = \{\!\{\mathbf{u}\}\!\}_{e} = \mathbf{u}.$
We represent the inflow and outflow parts of the boundary $\Gamma$ as $\Gamma_{\text{in}}$ and $\Gamma_{\text{out}}$, respectively, where
$$\Gamma_{\text{in}} = \{\mathbf{x} \in \Gamma : \boldsymbol{\beta} \cdot \n < 0\}, \quad \Gamma_{\text{out}} = \{\mathbf{x} \in \Gamma : \boldsymbol{\beta} \cdot \n \ge 0\}.$$
Likewise, the inflow and outflow parts of the boundary of an element $K$ are referred to as $\partial K_{\text{in}}$ and $\partial K_{\text{out}}$, respectively, where
$$\partial K_{\text{in}} = \{\mathbf{x} \in \partial K : \boldsymbol{\beta} \cdot \n < 0\}, \quad \partial K_{\text{out}} = \{\mathbf{x} \in \partial K : \boldsymbol{\beta} \cdot \n \ge 0\}.$$
Let $\boldsymbol{V}_h$ be a discrete subspace of $\boldsymbol{H}_{0}^{\text{div}}(\Omega)$ (used to approximate state and adjoint velocity) such that 
$$\boldsymbol{V}_h = \{\v \in \boldsymbol{H}_{0}^{\text{div}}(\Omega) \ | \ \v|_K \in \ \text{BDM}_k \ \text{for all} \ K \in  \mathcal{T}_h \ \text{and} \ k \ge 1\},$$ and 
$$\boldsymbol{V}_h^{0} = \{\mathbf{v} \in \boldsymbol{V}_{h} \ | \ \nabla\cdot\mathbf{v} = 0\},$$ 
where $\text{BDM}_k(K) = [\mathcal{P}_k(K)]^2$ represents the Brezzi-Douglas-Marini spaces of order $k$. In this context, $\mathcal{P}_{k}(K)$ signifies the space of polynomials with a maximum degree of $k$ defined on $K \in \mathcal{T}_h$, and  $\boldsymbol{P}_{k}(K)$ refers to its vector-valued equivalent.\\
The analogous discrete subspace $Q_h$ of $L_{0}^2(\Omega)$ (utilized for approximating state and adjoint pressure) is
$$Q_h = \{v \in L_{0}^2(\Omega) \ | \ v|_K \in \mathcal{P}_k(K) \ \text{for all} \ K \in  \mathcal{T}_h \ \text{and} \ k \ge 0\}.$$
A crucial characteristic of the selected discrete spaces is that
$$\nabla \cdot \boldsymbol{V}_{h}\subset Q_h.$$
Regarding the control variable, we consider $\boldsymbol{U}_{ad,h}$ to be a nonempty, closed, and convex discrete subspace of $ \boldsymbol{U}_{ad}$. Employing a piecewise constant discretization, we define the discrete control space as:
\begin{align*}
	\boldsymbol{U}_{ad,h}=\{\u_h \in \boldsymbol{L}^2(\Omega) \ | \ \u_h|_{K} \in \boldsymbol{P}_{0}(K) \ \forall \ K \in \mathcal{T}_h \}.
\end{align*}
\subsection{$\boldsymbol{H}^{\text{div}}$-DG formulation for the Oseen problem with control} 

The discrete divergence-conforming DGFEM weak formulation for the optimality system (2.12) is : find $(\y_h, p_h, \w_h, r_h, \mathbf{u}_h) \in  \boldsymbol{V}_h \times Q_h \times  \boldsymbol{V}_h \times Q_h \times \boldsymbol{U}_{ad,h}$ such that
\begin{subequations}
\begin{align}
	\label{3.1a} a_h(\mathbf{y}_h,\mathbf{v}_h) + O_h(\boldsymbol{\beta};\y_h,\mathbf{v}_h) + b_h(\mathbf{v}_h,p_h) &= (\mathbf{f} + \u_h, \mathbf{v}_h), \\
 	\label{3.1b} \hspace{1.57cm}b_h(\y_h,\phi_h) &= 0,\\
 	\label{3.1c} a_h(\mathbf{z}_h,\w_h) + O_h(\boldsymbol{\beta};\z_h,\mathbf{w}_h) - b_h(\mathbf{z}_h,r_h) &= (\y_h - \mathbf{y}_d , \mathbf{z}_h), \\
	\label{3.1d} \hspace{1.65cm}b_h(\w_h,\psi_h) &= 0,\\
	\label{3.1e} (\w_h + \lambda \u_h,\overline{\mathbf{u}}_h-\u_h)_{\boldsymbol{L}^{2}(\Omega)} &\ge 0,
\end{align}
\end{subequations}
for all $(\mathbf{v}_h, \phi_h, \mathbf{z}_h, \psi_h, \overline{\u}_h) \in \boldsymbol{V}_h \times Q_h \times \boldsymbol{V}_h \times Q_h \times \boldsymbol{U}_{ad,h}$, where the bilinear forms $a_h(\cdot,\cdot), \ O_h(\cdot,\cdot)$ and  $b_h(\cdot,\cdot)$ are defined as follows:
\begin{align}
\nonumber a_h(\y_h,\mathbf{v}_h) = \ & (\nu\nabla \y_h,\nabla \mathbf{v}_h)_{\mathcal{T}_h}-  \langle \{\!\!\{\nu \nabla\y_h\}\!\!\},[\![\mathbf{v}_h \otimes \mathbf{n}]\!] \rangle_{\mathcal{E}^{i}(\mathcal{T}_h)} -  \langle \{\!\!\{\nu \nabla \mathbf{v}_h\}\!\!\},[\![\y_h \otimes \mathbf{n}]\!] \rangle_{\mathcal{E}^{i}(\mathcal{T}_h)} \\
\nonumber &+  \langle \gamma_{h}^{2} [\![\y_h \otimes \mathbf{n}]\!],[\![\mathbf{v}_h \otimes \mathbf{n}]\!] \rangle_{\mathcal{E}^{i}(\mathcal{T}_h)}  - \langle \nu \nabla \y_h,\mathbf{v}_h \otimes \mathbf{n} \rangle_{\mathcal{E}^{b}(\mathcal{T}_h)} -  \langle \nu \nabla \mathbf{v}_h,\y_h \otimes \mathbf{n} \rangle_{\mathcal{E}^{b}(\mathcal{T}_h)} \\
&+ 2\langle \gamma_{h}^{2}{\mathbf{y}}_h \otimes \mathbf{n},\mathbf{v}_h \otimes \mathbf{n} \rangle_{\mathcal{E}^{b}(\mathcal{T}_h)}\label{3.2},  \\
\nonumber 
O_h(\boldsymbol{\beta};\y_h,\mathbf{v}_h) &= \sum_{K \in \mathcal{T}_h} \int_{K} ((\sigma -\nabla\cdot \boldsymbol{\beta})\y_h\cdot \mathbf{v}_h -  \y_h \boldsymbol{\beta}^{T}: \nabla \mathbf{v}_h) \ dx + \sum_{K \in \mathcal{T}_h} \int_{\partial K_{\text{out}} \cap \Gamma_{\text{out}}} (\boldsymbol{\beta} \cdot \boldsymbol{n}_{K}) \y_h \cdot \mathbf{v}_h \ ds \\
&+  \sum_{K \in \mathcal{T}_h} \int_{\partial K_{\text{out}} \setminus \Gamma} (\boldsymbol{\beta} \cdot \boldsymbol{n}_{K}) \y_h \cdot (\mathbf{v}_h-\mathbf{v}^{e}) \ ds,\label{3.3}\\
 b_h(\y_h,\phi) &=  -(\phi_h, \nabla \cdot \y_h), \label{3.4}
\end{align}
and $\gamma_{h}^{2} = \frac{\gamma \nu}{h_E}.$ Here, $\gamma >0$ is a penalty parameter that is chosen to be sufficiently large, independent of mesh size and viscosity coefficient $\nu$, to ensure the stability of the DG formulation, and $\mathbf{v}^{e}$ represents the exterior trace of $\mathbf{v}$ evaluated over the face under consideration and is set to zero on the boundary.

The norms on the discrete spaces $\boldsymbol{V}_h$ and $Q_h$ are delineated as follows:
\begin{align}
\label{3.8}	 |\!|\!|(\mathbf{v},\phi)|\!|\!|^{2} &:= |\!|\!|\mathbf{v}|\!|\!|_h^{2} + \|\nu^{-1/2}\phi\|_{\mathcal{T}_h}^{2},
\end{align} where 
\begin{align}
\label{3.9}	|\!|\!|\mathbf{v}|\!|\!|_{h}^{2} &:=  \|\nu^{\frac{1}{2}} \nabla \mathbf{v}\|_{\mathcal{T}_h}^2 
+ \|[\![\gamma_{h} \mathbf{v} \otimes \mathbf{n}]\!]\!] \|_{\mathcal{E}^{i}(\mathcal{T}_h)}^2 
+  2 \| \gamma_{h}\mathbf{v} \otimes \mathbf{n}\|_{\mathcal{E}^{b}(\mathcal{T}_h)}^2 + \kappa \|\mathbf{v}\|_{\mathcal{T}_h}^2.
\end{align}
\begin{lemma}\label{lem_31} 
The bilinear forms outlined in (\ref{3.2}-\ref{3.4}) exhibit the following properties:
		\begin{align*}
			a_h(\mathbf{y},\mathbf{v}) &\lesssim  |\!|\!|\mathbf{y}|\!|\!|_{h} |\!|\!|\mathbf{v}|\!|\!|_{h}, &&\forall \mathbf{y}, \mathbf{v} \in \boldsymbol{V}_h, 
			&&& a_h(\mathbf{v},\mathbf{v}) \gtrsim |\!|\!|\mathbf{v}|\!|\!|_{h}^{2},\quad \forall \mathbf{v} \in \boldsymbol{V}_h,\\
			|b_h(\mathbf{v},\phi) &\lesssim  |\!|\!|\mathbf{v}|\!|\!|_{h} \|\phi\|_{0} \vspace{-9mm} &&\forall (\mathbf{v},\phi) \in \boldsymbol{V}_h \times Q_h,
			&&& \underset{\phi \in Q_h}{\inf} \ \underset{\mathbf{v} \in \boldsymbol{V}_h}{\sup} \frac{b_h(\mathbf{v},\phi)}{|\!|\!|\mathbf{v}|\!|\!|_{h} \|\phi\|_{0}} \gtrsim \alpha.
		\end{align*}
\end{lemma}
\begin{proof}
       The continuity of $a_h(\cdot,\cdot)$ is established through the application of the Cauchy-Schwarz inequality in (\ref{3.2}), while coercivity is demonstrated by referencing the analysis in \cite[Section~4]{UDG}. This is accomplished by considering $\boldsymbol{V}_h$ as a subspace of the vector-valued DG spaces based on $\mathcal{P}_{k+1}$. Given that $\boldsymbol{V}_h \subset \boldsymbol{H}^{\text{div}}(\Omega)$, the proof directly follows from \cite[Propositions ~4.2, 4.3]{CBK}.
Furthermore, the continuity of $b_h(\cdot,\cdot)$ is established through the utilization of the Cauchy-Schwarz inequality in (\ref{3.4}), and the inf-sup condition is proven in \cite[Theorem 6.12]{MHD}.
\end{proof}
Applying Lemma \ref{lem_31} in conjunction with the Babu{s}ka-Brezzi theory for saddle point problems ensures the unique solvability of the discrete optimality system.
\section{A priori error estimates}\label{PRIORI ERROR ESTIMATES.} \setcounter{equation}{0} 
This section is primarily dedicated to establishing a priori error estimates for the state, co-state, and control variables. 
Throughout this section, we assume that the variable viscosity $\nu(x) \in \W^{1,\infty}(\Omega)$, and the pair $(\text{BDM}_{1}-P_{0})$ is used to approximate velocity and pressure variables, respectively. To derive these estimates, we initially introduce certain auxiliary variables. For a specified control $\u$ and source field $\f$, let $(\y_h(\u), p_h(\u)) \in \boldsymbol{V}_h \times Q_h$ represent the solution to the following problem:
\begin{subequations}\label{4.1a}
\begin{align}
	\label{4.1a1}  a_h(\mathbf{y}_h(\u),\mathbf{v}_h)+ O_h(\boldsymbol{\beta};\y_h(\u),\v_h) + b_h(\mathbf{v}_h,p_h(\u)) &= (\f +\u, \mathbf{v}_h ) \hspace{1.12cm} \forall \ \mathbf{v}_h \in  \boldsymbol{V}_h, \\
	\label{4.1b}	\hspace{0.46cm} b_h(\mathbf{y}_h(\u),\phi_h) &= 0 \hspace{2.6cm} \forall \ \phi_h \in Q_h.
\end{align} 
\end{subequations}
Likewise, for the state velocity $\y$, let $(\w_h(\y), r_h(\y)) \in \boldsymbol{V}_h \times Q_h$ denote the solution to the following problem:
\begin{subequations}\label{4.1c}
\begin{align}
	\label{4.1c1}	a_h(\z_h,\w_h(\y)) + O_h(\boldsymbol{\beta};\z_h,\w_h(\y)) - b_h(\mathbf{z_h},r_h(\y)) &= (\mathbf{y} - \mathbf{y}_d , \mathbf{z}_h) \hspace{1cm} \forall \ \mathbf{z}_h \in  \boldsymbol{V}_h, \\
	\label{4.1d}	 \hspace{0.28cm} b_h(\w_h(\y),\psi_h) &= 0 \hspace{2.66cm} \forall \ \psi_h \in Q_h.
\end{align}
\end{subequations}
\begin{lemma}\label{lem:lemma 4.1.}
	Let $(\y_h,p_h), (\w_h,r_h)$ be a solution of the discrete systems (\ref{3.1a}-\ref{3.1b}) and (\ref{3.1c}-\ref{3.1d}), respectively. Additionally, let $(\y_h(\u), p_h(\u))$ and $(\w_h(\y), r_h(\y))$ represent the auxiliary variables defined earlier. Then, the following estimates hold:
	\begin{align}
		\label{4.2}		|\!|\!|(\y_h(\u)-\y_h, p_h(\u)-p_h)|\!|\!|  &\lesssim  \|\u-\u_h\|_{0}, \\
		\label{4.3}		|\!|\!|(\w_h(\y)-\w_h, r_h(\y)-r_h)|\!|\!|  &\lesssim  \|\y-\y_h\|_{0}.
	\end{align}
\end{lemma}
\begin{proof} 
By subtracting (\ref{3.1a}-\ref{3.1b}) from (\ref{4.1a}) and (\ref{3.1c}-\ref{3.1d}) from (\ref{4.1c}), and utilizing Lemma-\ref{lem_31}, we infer the indicated outcome.
\end{proof}	
\begin{lemma}\label{lem: Lemma 4.2.}
	Suppose that $(\y,p), (\w,r) \in \boldsymbol{H}^{2}(\Omega) \times H^{1}(\Omega)$ be the solutions of the systems (\ref{2.13a}-\ref{2.13b})  and (\ref{2.13c}-\ref{2.13d}), repectively, for \ $\u,\ \f,\ \y_d \in \boldsymbol{H}^{1}(\Omega)$.  Let $(\y_h(\u),p_h(\u)), (\w_h(\y),r_h(\y))$ be the discrete auxiliary variables. Then, there exists a positive constant $C$ independent of mesh size $h$, but dependent on $\nu_{0}^{-1}$, $\kappa$ and $\boldsymbol{\beta}$ such that :
	\begin{itemize}
		\item[(i)] \textbf{Diffusion-dominated regime:} If $R_{K}(\x) = \min \Big \{\frac{|\boldsymbol{\beta} \cdot n| h_K}{C_{a}^{c} \nu\alpha}, 1 \Big\} <1, \ \forall \x \in \partial K, \ \forall K \in \mathcal{T}_h$, we have
	\begin{align}
		\label{4.11} 	\|\y-\y_h(\u)\|_{0} &\le Ch^{2} \|\y\|_{2},\quad  \|\w-\w_h(\y)\|_{0} \le Ch^{2} \|\w\|_{2}.
	\end{align}
		\item[(ii)] {\textbf{Convection-dominated regime:} If $R_{K}(\x) = \min \Big \{\frac{|\boldsymbol{\beta} \cdot n| h_K}{C_{a}^{c}\nu \alpha}, 1 \Big \} =1, \ \forall \x \in \partial K, \ \forall K \in \mathcal{T}_h$, we have
	\begin{align}
		\label{4.13} 	\|\y-\y_h(\u)\|_{0} &\le Ch^{3/2} \|\y\|_{2},\quad  \|\w-\w_h(\y)\|_{0} \le Ch^{3/2} \|\w\|_{2}.
	\end{align}}
\end{itemize}
\end{lemma}
	To prove the result for the diffusion-dominated regime, we firstly extend \cite[Lemma 4.1]{HH} for variable viscosity parameter $\nu$ and a convective velocity field $\boldsymbol{\beta}$ which is not necessarily divergence free. Let $(\mathbf{s}_h,t_h) \in \boldsymbol{V}_h \times Q_h$ be the approximations satisfying 
	\begin{align}
		\label{4.12a}	a_h(\mathbf{s}_h,\v_h) + b_h(\v_h,t_h) &= -(\nabla\cdot(\nu \nabla \y), \v_h) \hspace{1cm} \forall \ \v_h \in V_h,\\
		\label{4.12b}	b_h(\mathbf{s}_h,\psi_h) &= 0 \hspace{2.63cm} \forall \ \psi_h \in Q_h.
	\end{align}
	Then, the following bound holds \cite{RSW}:
	\begin{align}
	\label{4.13n} \|\y-\mathbf{s}_h\|_{0} + h |\!|\!|\y-\mathbf{s}_h|\!|\!| \le Ch^{j} \|u\|_{j}, \ \ \ 1 \le j \le k+1.
	\end{align}
	Now we introduce the approximation and discretization errors for velocity and pressure variables as:
	\begin{align}
		\label{4.14n}	\eta_{y} &= \y - \mathbf{s}_h, \ \xi_{y} = \y_h(\u) - \mathbf{s}_h, \ \eta_p = p - \Pi_{Q}p,\	\xi_p = p_h - \Pi_{Q}p,
	\end{align}
where $\Pi_{Q}$ denotes the standard $L^{2}$-projection onto $Q_h$. 
		By taking the difference between the two formulations, we obtain:
		\begin{align*}
			a_h(\y-\y_h(\u),\v_h) + b_h(p-p_h(\u),\v_h) - b_h(\psi_h,\y-\y_h(\u)) + O_h(\boldsymbol{\beta};\y,\v_h)   - O_h(\boldsymbol{\beta};\y_h(\u),\v_h) = 0. 
		\end{align*}
		By employing (\ref{4.14n}) and substituting $(\v_h, \psi_h) = (\xi_{y}, \xi_p)$ with $a_h(\eta_{y}, \xi_{y}) = b_h(\eta_{p},\eta_{y}) = 0$, into the preceding equation, we get:
		\begin{align}
		\label{4.16}	a_h(\xi_{y},\xi_{y}) + O_h(\boldsymbol{\beta};\xi_{y},\xi_{y}) = b_h(\eta_p,\xi_y) + O_h(\boldsymbol{\beta};\eta_{y},\xi_{y}),
		\end{align}
		Applying the coercivity of $a_h(\cdot,\cdot)$ and employing estimates for $O_h$, we obtain:
		\begin{align}
		\label{4.17}	
		C_{a}^{c} |\!|\!|{\xi}_{y}|\!|\!|_{h}^{2} 
		\le   |b_h(\eta_{p},\xi_{y})| + |O_h(\boldsymbol{\beta};\eta_{y},\xi_{y})|.
		\end{align}
		The initial term on the right-hand side becomes zero in our approach due to the pointwise divergence-free and divergence-conforming nature of $\xi_y$. Using Holder's inequality, and Young's inequality, we arrive at:
		\begin{align}
			\nonumber	|O_h(\boldsymbol{\beta};\eta_{u},\xi_{y})| \le \ &\bigg(C \kappa \xi + \frac{C \|\boldsymbol{\beta}\|_{L^{\infty}}^{2}}{\nu_{0}} \bigg) \|\eta_{y}\|_{0}^{2} +  \frac{C_{a}^{c}}{2} |\!|\!|\xi_{y}|\!|\!|_{h}^{2} 
				+ 3\sum_{K \in \mathcal{T}_h}  \int_{\partial K} R_{K} |\boldsymbol{\beta} \cdot n| |[\![\eta_{y}]\!]|^{2} \ ds \\
				&+ 4 \sum_{K \in \mathcal{T}_h}  \int_{\partial K} R_{K} |\boldsymbol{\beta} \cdot n| |\eta_{y}|^{2} \ ds 
			\label{4.18} + {\frac{1}{4} \sum_{K \in \mathcal{T}_h}  \int_{\partial K} R_{K}^{-1} |\boldsymbol{\beta} \cdot n| |[\![\xi_{y}]\!]|^{2} \ ds}
		\end{align}
		Choosing $R_{K} = \frac{|\boldsymbol{\beta} \cdot n| h_K}{C_{a}^{c} \nu \alpha}$, we obtain:
		\begin{align}
			 \frac{1}{4} \sum_{K \in \mathcal{T}_h}  \int_{\partial K} \frac{C_{a}^{c} \nu \alpha}{|\boldsymbol{\beta} \cdot n| h_K} |\boldsymbol{\beta} \cdot n| |[\![\xi_{y}]\!]|^{2} \ ds 
			= \frac{C_{a}^{c}}{4} \sum_{K \in \mathcal{T}_h}  \int_{\partial K} \gamma_{h}^{2}  |[\![\xi_{y}]\!]|^{2} \ ds 
			\label{4.19}	 &\le \frac{C_{a}^{c}}{2} \|\gamma_{h} [\![\xi_{y}]\!]\|^{2}.
		\end{align}
		Through the application of the triangle inequality:
		\begin{align*}
			\|\y-\y_h(\u)\|_{0} = \|\y - \mathbf{s}_h + \mathbf{s}_h - \y_h(\u)\|_{0} \le \|\y - \mathbf{s}_h\|_{0} + \|\mathbf{s}_h - \y_h(\u)\|_{0} = \|\eta_y\|_{0} + \|\xi_y\|_{0}.
		\end{align*}
		Utilizing equations (\ref{4.18}-\ref{4.19}) in (\ref{4.17}) 
	and 	applying (\ref{4.13n}) in conjunction with the preceding results, we obtain:
		\begin{align}
			\nonumber	\|\y-\y_h\|_{0} 
			\nonumber	&\le Ch^{2} \|\y\|_{2} \ + \frac{1}{\sqrt{\kappa}} \bigg( \bigg(C \kappa \xi + \frac{C \|\boldsymbol{\beta}\|_{L^{\infty}}^{2}}{\nu_{0}}\bigg) h^{4} \|\y\|_{2}^{2} +  \frac{12C}{C_{a}^{c} \nu_0 \alpha} \sum_{K \in \mathcal{T}_h} \int_{\partial K} h_K |[\![\eta_{y}]\!]|^{2} \ ds \\
			\label{4.20} 	&+  \frac{16C}{C_{a}^{c} \nu_0 \alpha} \sum_{K \in \mathcal{T}_h} \int_{\partial K} h_K |\eta_{y}|^{2} \ ds \bigg)^{1/2}. 
		\end{align}
		The utilization of the trace inequality leads to the stated estimates. 
		Similarly, the second estimate (\ref{4.11}) for the co-state can be derived. In the same manner, we can establish the estimate for the convection-dominated regime by employing the approach outlined in \cite[Lemma 4.5]{HH} and setting parameter $R_{K}=1$.
	
\begin{lemma}\label{lem: Lemma 4.4.}
	Given solutions $(\y,p)$ and $(\w,r)$ in $\boldsymbol{H}^{2}(\Omega) \times H^{1}(\Omega)$ for the systems (\ref{2.13a}-\ref{2.13b}) and (\ref{2.13c}-\ref{2.13d}), respectively, with prescribed inputs $\u,\ \f,\ \y_d \in \boldsymbol{H}^{1}(\Omega)$, let $(\y_h(\u),p_h(\u))$ and $(\w_h(\y),r_h(\y))$ be the discrete auxiliary variables.Then, the following relationships hold:
	\begin{align}
		\label{4.21a} 	|\!|\!|\y-\y_h(\u)|\!|\!|_{h} + \|\nu^{-1/2}(p-p_h(\u))\|_{0} &\le C_s h \ (\|\y\|_{2}+ \|p\|_{1}), \\
		\label{4.21b} 	|\!|\!|\w-\w_h(\y)|\!|\!|_{h} + \|\nu^{-1/2}((r-r_h(\y))\|_{0} &\le C_a h \ (\|\w\|_{2}+ \|r\|_{1}),
	\end{align}
	for positive constants $C_{s}$ and ${C_{a}}$ independent of $h$ but dependent on $\nu$.
\end{lemma}
\begin{proof}
	Proof of this lemma follows from \cite[Theorem 4.8]{CBK}.
\end{proof}
\begin{lemma}\label{lem: Lemma 4.5.}
	Let $(\y,p,\w,r,\u)$ be the solution of the continuous system and $(\y_h,p_h,\w_h,r_h,\u_h)$ be their $H^{div}$ - DGFEM approximation under piecewise constant discretization of control. For the regularity assumptions same as previous result and a positive constant C, the following holds
	\begin{align*}
		\|\u-\u_h\|_{0} \le Ch \|\u\|_{1}.
	\end{align*}
\end{lemma}
\begin{proof}
	For $\u_h \in \boldsymbol{U}_{ad}$, 
	let $(\y(\u_h),p(\u_h)) \in \boldsymbol{V} \times Q$ represent the solution to the problem,
	\begin{align}
		\label{4.22a} a(\mathbf{y}(\u_h),\mathbf{v})+ b(\mathbf{v},p(\u_h)) &= (\f +\u_h, \mathbf{v} ) \hspace{1cm} \forall \ \mathbf{v} \in  \boldsymbol{V}, \\
		\label{4.22b}	\hspace{0.46cm} b(\mathbf{y}(\u_h),\phi) &= 0 \hspace{2.5cm} \forall \ \phi \in Q,
	\end{align} 
	and let $(\w(\u_h),r(\u_h))  \in  \boldsymbol{V} \times Q$ be the solution of the problem
	\begin{align}
		\label{4.22c}	a(\z,\w(\u_h)) - b(\mathbf{z},r(\u_h)) &= ( \mathbf{y}(\u_h) - \mathbf{y}_d , \mathbf{z}) \hspace{0.45cm} \forall \ \mathbf{z} \in \boldsymbol{V}, \\
		\label{4.22d}	 \hspace{0.28cm} b(\w(\u_h),\psi) &= 0 \hspace{2.6cm} \forall \ \psi \in Q.
	\end{align}
Considering the solution $\w$ of (\ref{2.13c}) and $\w(\u_h)=\w(\y(\u_h)),$ the reduced functional $F$ exhibits the following property:
	\begin{align*}
		F'(\u) (\theta) &= - \lambda (\u,\theta) + (\theta,\w) \hspace{1.8cm} &&\forall \ \theta \in \boldsymbol{U}_{ad}, \\ 
		F'(\u_{h}) (\theta_{h}) &= - \lambda (\u_{h},\theta_{h}) + (\theta_{h},\w(\u_h)) \hspace{0.5cm} &&\forall \ \theta_{h} \in \boldsymbol{U}_{ad,h}.
	\end{align*}
By employing second-order sufficient optimality conditions, inspired by the approach in \cite{MBO}, we establish the following identities:
\begin{align*}
	-(\lambda \u,\u - \u_{h}) + (\u - \u_{h},\w) &= 0 = -(\lambda \u,\u - I_{h} \u) + (\u - I_{h} \u,\w), \\
	-(\lambda \u_{h},\u_{h} - I_{h} \u) + (\u_{h} - I_{h} \u,\w_{h}) &= 0,
\end{align*}
where $I_h$ denotes the $L^2-$projection.  Moreover, we have 
\begin{align}
\nonumber	\lambda \|\u-\u_{h}\|_{0}^{2} &\le F'(\u)(\u-\u_h) - F'(\u_h)(\u-\u_h) \\
\nonumber	&=  {\lambda(\u-\u_h,I_h \u- \u)} + {(I_h \u - \u,\w_h-\w(\u_h))} + {(I_h \u - \u,\w(\u_h)-\w)} \\
\label{4.23n} 	& \quad + {(\u_{h}-\u,\w(\u_h)-\w_h)}= T_1+T_2+T_3+T_4.
\end{align}
By employing Cauchy-Schwarz and Young's inequality, for a positive scalar $\epsilon$, we obtain:
\begin{align*}
	T_1 \  &\le \lambda \|\u-\u_h\|_{0} \|I_h \u- \u\|_{0} \le \lambda \epsilon \|\u-\u_h\|_{0}^{2} + \frac{\lambda}{4\epsilon} \|I_h \u- \u\|_{0}^{2},\\
	T_2 \  &\le \|I_h \u- \u\|_{0} \|\w_h-\w(\u_h)\|_{0} \le \frac{\lambda}{2\epsilon}  \|I_h \u- \u\|_{0}^{2} + \frac{\epsilon}{2\lambda}  \|\w_h-\w(\u_h)\|_{0}^{2},\\
	T_3 \  &\le \|I_h \u- \u\|_{0} \|\w(\u_h)-\w\|_{0} \le \frac{\epsilon}{\lambda}  \|I_h \u- \u\|_{0}^{2} + \frac{\lambda}{4\epsilon}  \|\w(\u_h)-\w\|_{0}^{2},\\
	T_4 \  &\le \|\u_{h}-\u\|_{0} \|\w(\u_h)-\w_h\|_{0} \le \lambda \epsilon \|\u_{h}-\u\|_{0}^{2} + \frac{1}{4\lambda \epsilon}   \|\w(\u_h)-\w_h\|_{0}^{2}.
\end{align*}
Incorporating these bounds into (\ref{4.23n}), we get:
\begin{align}
\label{4.24n} 	\lambda \|\u-\u_{h}\|_{0}^{2} &\le \frac{\lambda}{\epsilon} \|\u-I_h \u\|_{0}^{2} + 2\lambda \epsilon \|\u-\u_h\|_{0}^{2} + \frac{\epsilon}{\lambda}  \|\w-\w(\u_h)\|_{0}^{2} + \bigg( \frac{\epsilon}{2\lambda}+\frac{1}{4\lambda \epsilon}\bigg)\|\w(\u_h)-\w_h\|_{0}^{2}.
\end{align}
Subtracting (\ref{4.22c}-\ref{4.22d}) from (\ref{2.13c}-\ref{2.13d}), we obtain:
	\begin{align*}
	a(\z,\w-\w(\u_h)) - b(\mathbf{z},r-r(\u_h)) &= (\y - \mathbf{y}(\u_h), \mathbf{z}) \hspace{0.45cm} \forall \ \mathbf{z} \in  \boldsymbol{V}, \\
	\hspace{0.28cm} b(\w-\w(\u_h),\psi) &= 0 \hspace{2.63cm} \forall \ \psi \in Q.
\end{align*}
Consider $\z = \w-\w(\u_h)$ and $\psi = r-r(\u_h)$; then, we have:
\begin{align*}
	a(\w-\w(\u_h),\w-\w(\u_h)) = (\y - \mathbf{y}(\u_h), \w-\w(\u_h)).
\end{align*}
Now using
\begin{align*}
	|\!|\!|\v|\!|\!|^{2} 
	&\ge \bigg(\frac{\nu_{0}}{C_P^{2}}+ \kappa \bigg) \|\mathbf{v}\|_{0}^2 = \eta \|\mathbf{v}\|_{0}^2,
\end{align*}
in the above equation, we have
\begin{align}
\label{4.25} 	\|\w-\w(\u_h)\|_{0} \le \eta^{1/2} 	\|\y-\y(\u_h)\|_{0} \le \eta \|\u-\u_h\|_{0}.
\end{align}
Applying (\ref{4.25}) to (\ref{4.24n}), we obtain:
\begin{align*}
	\lambda \|\u-\u_{h}\|_{0}^{2} &\le \frac{\lambda}{\epsilon} \|\u-I_h \u\|_{0}^{2} + \epsilon \bigg(2\lambda + \frac{\eta^{2}}{\lambda}\bigg) \|\u-\u_h\|_{0}^{2} + \bigg( \frac{\epsilon}{2\lambda}+\frac{1}{4\lambda \epsilon}\bigg)\|\w(\u_h)-\w_h\|_{0}^{2}.
\end{align*}
By selecting $\epsilon = \frac{\lambda}{2} \left(2\lambda + \eta^{2} \lambda^{-1}\right)^{-1}$, we have:
\begin{align*}
	\frac{\lambda}{2} \|\u-\u_{h}\|_{0}^{2} &\le \frac{\lambda}{\epsilon} \|\u-I_h \u\|_{0}^{2} + \bigg( \frac{\epsilon}{2\lambda}+\frac{1}{4\lambda \epsilon}\bigg)\|\w(\u_h)-\w_h\|_{0}^{2} \\
	\|\u-\u_{h}\|_{0}^{2} &\le \frac{2}{\epsilon} \|\u-I_h \u\|_{0}^{2} + \bigg( \frac{\epsilon}{\lambda^{2}}+\frac{1}{2\lambda^{2} \epsilon}\bigg)\|\w(\u_h)-\w_h\|_{0}^{2}.
\end{align*}
Applying the standard $L^2-$projection estimate and considering estimates for the second term, we obtain:
\begin{align*}
		\|\u-\u_{h}\|_{0} &\le C_1 \bigg(\sum_{K \in \mathcal{T}_h} h_{K}^{2} \|\u\|_{1,K}\bigg)^{1/2} \le Ch \|\u\|_{1}.
\end{align*}
\end{proof}
	\begin{theorem}\label{thm: Theorem 4.6.}
		Let $(\y,p,\w,r)$ be the solution to the system (\ref{2.13a}-\ref{2.13d}), and let  $(\y_h,p_h, \w_h,r_h)$ denote its $H^{div}$ - DGFEM approximation achieved through a piecewise constant discretization of the control. For a positive constant $C$, the theorem asserts the following inequality:
		\begin{itemize}
			\item[(i)] 	For the \textbf{diffusion-dominated regime}, we have the following estimates:
			\begin{align}
				\label{4.39} \| \mathbf{y}- \mathbf{y}_h\|_{0}  &\le Ch^{2}\big[ \|\y\|_{2}+\|\u\|_{1}\big], \\
				\label{4.40}  \| \mathbf{w}- \mathbf{w}_h\|_{0}  &\le  Ch^{2} \big[\|\y\|_{2}+\|\u\|_{1}+ \|\w\|_{2}\big].
			\end{align} 
					\item[(ii)] {For the \textbf{convection-dominated regime}, we have the following estimates:
		\begin{align}
			\label{4.41} \| \mathbf{y}- \mathbf{y}_h\|_{0}  &\le Ch^{3/2}\big[ \|\y\|_{2}+\|\u\|_{1}\big], \\
			\label{4.42}  \| \mathbf{w}- \mathbf{w}_h\|_{0}  &\le  Ch^{3/2} \big[\|\y\|_{2}+\|\u\|_{1}+ \|\w\|_{2}\big].
		\end{align}}
		\end{itemize}	\end{theorem}
	\begin{proof}
	        To initiate, we decompose the total error and apply the triangle inequality to yield:
		\begin{align}
			\label{4.27} 	\| \mathbf{y}- \mathbf{y}_h\|_{0} \le \| \mathbf{y}- \mathbf{y}_h(\u)\|_{0} + \| \mathbf{y}_h(\u)- \mathbf{y}_h(\Pi_h \u)\|_{0} + \| \mathbf{y}_h(\Pi_h \u)- \mathbf{y}_h\|_{0},
		\end{align} 
		where $\Pi_h$ denotes $\boldsymbol{L}^{2}$-projection operator onto the discrete control space $\boldsymbol{U}_{ad,h}.$
		 Let $(\tilde{\w}_h,\tilde{r}_h) \in \boldsymbol{V}_h \times Q_h$ be the unique solution of the following auxiliary discrete dual Oseen problem 
		\begin{align}
			\label{4.28a}	a_h(\tilde{\z}_h,\tilde{\w}_h) + O_h(\boldsymbol{\beta};\tilde{\z}_h,\tilde{\w}_h)- b_h(\tilde{\z}_h,\tilde{r}_h) &= (\y_h(\u) - \y_h(\Pi_h \u) , \tilde{\z}_h) \    \ \hspace{1.47cm} \forall \ \tilde{\z}_h \in  \boldsymbol{V}_h, \\
			\label{4.28b}	\hspace{1.65cm}b_h(\tilde{\w}_h,\tilde{\psi}_h) &= 0   \hspace{5cm}  \forall \ \tilde{\psi}_h \in Q_h.
		\end{align}
		By substituting $\tilde{\z}_h = \tilde{\w}_h$ into (\ref{4.28a}) and $\tilde{\psi}_h = \tilde{r}_h$ into (\ref{4.28b}), and subsequently combining the resultant equations, we derive:
		\begin{align*}
			a_h(\tilde{\w}_h,\tilde{\w}_h) + O_h(\boldsymbol{\beta};\tilde{\w}_h,\tilde{\w}_h) = (\y_h(\u) - \y_h(\Pi_h \u) , \tilde{\w}_h).
		\end{align*} 
		Using coercivity of $a_h(\cdot,\cdot)$ and positiveness property of $O_h(\boldsymbol{\beta};\cdot,\cdot)$, we get
		\begin{align}
			\label{4.29n} |\!|\!|\tilde{\w}_h|\!|\!|_{h} \le C \|\y_h(\u) - \y_h(\Pi_h \u)\|_{0}.
		\end{align}
		By substituting $\tilde{\z}_h = \y_h(\u) - \y_h(\Pi_h \u)$ and $\tilde{\psi}_h = p_h(\u) - p_h(\Pi_h \u)$ into (\ref{4.28a}) and (\ref{4.28b}) respectively, and summing both equations, we obtain:
		\begin{align}
			\nonumber	a_h(\y_h(\u) - &\y_h(\Pi_h \u),\tilde{\w}_h) + O_h(\boldsymbol{\beta};\y_h(\u) - \y_h(\Pi_h \u),\tilde{\w}_h) - b_h(\y_h(\u) - \y_h(\Pi_h \u),\tilde{r}_h) \\
			\label{4.30n} &+ b_h(\tilde{\w}_h,p_h(\u) - p_h(\Pi_h \u)) = (\y_h(\u) - \y_h(\Pi_h \u) , \y_h(\u) - \y_h(\Pi_h \u)).
		\end{align}
		Now, by utilizing the discrete state equations for $\y_h(\u)$ and $\y_h(\Pi_h \u)$, we acquire:
		\begin{align}
			\nonumber	a_h(\y_h(\u) - &\y_h(\Pi_h \u),\tilde{\w}_h) + O_h(\boldsymbol{\beta};\y_h(\u) - \y_h(\Pi_h \u),\tilde{\w}_h) + b_h(\tilde{\w}_h,p_h(\u) - p_h(\Pi_h \u)) \\
			\label{4.31n}&- b_h(\y_h(\u) - \y_h(\Pi_h \u),\tilde{r}_h) = (\u - \Pi_h \u , \tilde{\w}_h).
		\end{align}
		By subtracting (\ref{4.31n}) from (\ref{4.30n}) and using $(\u - \Pi_h \u , \Pi_h \tilde{\w}_h)=0$, we find that
		\begin{align}
				\|\y_h(\u) - \y_h(\Pi_h \u)\|_{0}^{2} &= (\u - \Pi_h \u , \tilde{\w}_h) 
			\label{4.32n}= (\u - \Pi_h \u , \tilde{\w}_h - \Pi_h \tilde{\w}_h) 
		\end{align}
		Employing the Cauchy-Schwarz inequality and the approximation properties of $\boldsymbol{L}^2-$projection, we obtain:
		\begin{align}
			 (\u - \Pi_h \u , \tilde{\w}_h - \Pi_h \tilde{\w}_h) &\le Ch \|\u - \Pi_h \u\|_{0} \|\tilde{\w}_h\|_{1}
			\label{4.33}\le Ch \|\u - \Pi_h \u\|_{0} |\!|\!|\tilde{\w}_h|\!|\!|_{h}
		\end{align}
		By applying (\ref{4.29n}) and (\ref{4.33}) to (\ref{4.32n}), we obtain:
		\begin{align}
			\label{4.34}\|\y_h(\u) - \y_h(\Pi_h \u)\|_{0} \le Ch \|\u - \Pi_h \u\|_{0}.
		\end{align}
		Following the approach employed in the proof of Lemma-\ref{lem:lemma 4.1.}, the third term in the error decomposition is constrained as:
	\begin{align}
		\label{4.40h}	\| \mathbf{y}_h(\Pi_h \u)- \mathbf{y}_h\|_{0} &\le Ch \|\u - \Pi_h \u\|_{0} |\!|\!|\tilde{\w}_h|\!|\!|_{h} + Ch^{2} \|\u\|_{1} |\!|\!|\tilde{\w}_h|\!|\!|_{h}.
	\end{align}
			By substituting (\ref{4.33}) and (\ref{4.40h}) into (\ref{4.27}), we derive the following estimation
			\begin{align*}
				 \| \mathbf{y}- \mathbf{y}_h\|_{0} &\le Ch^{2} \ [\|\y\|_{2} +\|\u\|_{1}].
			\end{align*}
			We decompose the adjoint velocity error as $\w-\w_h= \w-\w_h(\y)+\w_h(\y)-\w_h$. By employing the triangle inequality and leveraging Lemmas-\ref{lem:lemma 4.1.} and \ref{lem: Lemma 4.2.} along with the previously established result, we obtain the following:
			\begin{align}
					\| \mathbf{w}- \mathbf{w}_h\|_{0} 
					& \le \|\w-\w_h(\y)\|_{0} + \|\y-\y_h\|_{0} 
				\label{4.30n1}\le Ch^{2} \ \big[\|\y\|_{2}+\|\u\|_{1}+ \|\w\|_{2}\big].
			\end{align}
(ii){Applying the triangle inequality and employing similar steps as in the first part, we utilize Lemma-\ref{lem: Lemma 4.2.} to obtain the estimates (\ref{4.41}) and (\ref{4.42})}.
\end{proof}
\begin{remark}
	{As evident from the findings in Lemma-\ref{lem: Lemma 4.2.} and Theorem-\ref{thm: Theorem 4.6.}, the velocity variable exhibits quasi-optimal convergence rates in convection-dominated regimes. 
	Nevertheless, when considering the assumptions delineated by Cockburn et al. in \cite[Assumptions - A1, A2]{CDGA}, optimal convergence rates are reinstated.}
\end{remark}
Next, we articulate and demonstrate an energy norm error estimate, derived from the previously obtained estimates for the error decomposition terms.
\begin{theorem} 
		Let $(\y, p, \w, r)$ be the state and adjoint velocities, and pressures, solutions of the system (\ref{2.13a}-\ref{2.13d}) with corresponding discrete approximate solutions $(\y_h, p_h, \w_h, r_h)$. Then, there exists a positive constant $C$ independent of the mesh size h such that 
		\begin{align}
			\label{4.59}	|\!|\!|\y-\y_h|\!|\!|_{h} + \|\nu^{-1/2}(p-p_h)\|_{0} &\le C h \ (\|\y\|_{2}+ \|p\|_{1}+\|\u\|_{1}), \\
	\label{4.60}	|\!|\!|\w-\w_h|\!|\!|_{h} + \|\nu^{-1/2}(r-r_h)\|_{0} &\le C h \ (\|\w\|_{2}+ \|r\|_{1}).
		\end{align}
		\end{theorem}
		\begin{proof}
			Utilizing the triangle inequality and subsequently employing equations(\ref{4.2}, \ref{4.3}, \ref{4.21a}, and \ref{4.21b}), we can express:
			\begin{align*}
				|\!|\!|\y-\y_h|\!|\!|_{h} + \|\nu^{-1/2}(p-p_h)\|_{0} %
	&\le  C h \ (\|\y\|_{2}+ \|p\|_{1}) + C\|\u-\u_h\|_{0} \le  C h \ (\|\y\|_{2}+ \|p\|_{1}+\|\u\|_{1}),
			\end{align*}
			and
			\begin{align*}
				|\!|\!|\w-\w_h|\!|\!|_{h} + \|\nu^{-1/2}(r-r_h)\|_{0} %
	&\le C h \ (\|\w\|_{2}+ \|r\|_{1}) + C\|\y-\y_h\|_{0} \le C h \ (\|\w\|_{2}+ \|r\|_{1}).
			\end{align*}
\end{proof}
\section{A posteriori error estimates}\label{A POSTERIORI ERROR ESTIMATES.} \setcounter{equation}{0}
In this section, we develop and analyze a residual-based a posteriori error estimator for the solution of the discretized optimality system. The error in the velocity of the state and co-state is quantified using the norm $|\!|\!|\cdot|\!|\!|_h$ defined in (\ref{3.8}) and the semi-norm $|\cdot|_{A}$ defined by:
\begin{align}
\label{5.1}	|\y|_{A}^{2} &:= |\y \boldsymbol{\beta}^{T}|_{*}^{2} + \sum_{E \in \mathcal{E}(\mathcal{T}_h)} \bigg(\sigma h_{E}\|[\![\y]\!]\|_{0,E}^{2} + h_{E} \|[\![\nu^{-1/2}\y]\!]\|_{0,E}^{2} \bigg),
 \end{align}
where
	$|\mathbf{y}  \boldsymbol{\beta}^{T}|_{*} := \underset{\v \in \boldsymbol{H}_{0}^{1}(\Omega) \setminus \{0\}}{\sup} \frac{\int_{\Omega} \mathbf{y} \boldsymbol{\beta}^{T} : \nabla \v \ dx }{	|\!|\!|\mathbf{v}|\!|\!|_h}.$
 The expressions involving $|\y \boldsymbol{\beta}^{T}|^2_*$ and $h_E |[\nu^{-1/2} \y]|^2_{0,E}$ within the semi-norm $|\cdot|_A$ are used to approximate the convective derivative. Meanwhile, the term $\sigma h_E \|[\![\y]\!]\|_{0,E}^{2} $ is associated with the reaction term. We then introduce scaling factors denoted by $\rho_K$ and $\rho_E$ as follows.
\[
\rho_{K} = \begin{cases}
	\min \Big\{h_{K} \nu^{-1/2}, \kappa^{-1/2}\Big\} & \text{if } \kappa \neq 0 \\
	h_{K} \nu^{-\frac{1}{2}} & \text{if } \kappa = 0
\end{cases},
\quad
\rho_{E} = \begin{cases}
	\min \Big\{h_{E} \nu^{-1/2}, \kappa^{-1/2}\Big\} & \text{if } \kappa \neq 0 \\
	h_{E} \nu^{-\frac{1}{2}} & \text{if } \kappa = 0
\end{cases}.
\]
 Let  $\nu_h,\ \boldsymbol{\beta}_{h},\ \sigma_{h},\ \f_{h}$ and $\y_{d,h}$ represent the piecewise polynomial approximations of the viscosity coefficient $\nu$, the convective velocity field $\boldsymbol{\beta}$, the reaction term coefficient $\sigma$, souce function $\f$ and the desired velocity field $\y_d$, respectively.
 These approximations may exhibit discontinuities across elemental edges.
For a specific element $K\in\mathcal{T}_h$, and given values of $(\y_h,p_h,\w_h,r_h, \u_h)$ from the corresponding finite element spaces $\boldsymbol{V}_h \times Q_h \times \boldsymbol{V}_h \times Q_h \times \boldsymbol{U}_{ad,h}$, we define local error indicators named $\eta_{K}^{\y}$, $\eta_{K}^{\w}$, and $\eta_{K}^{\u}$ as
\begin{align*}
	(\eta_{K}^{\y})^{2} := (\eta_{R_K}^{\y})^{2}   + (\eta_{E_K}^{\y})^{2}   +  (\eta_{J_K}^{\y})^{2} , \quad \quad (\eta_{K}^{\w})^{2} := (\eta_{R_K}^{\w})^{2}   + (\eta_{E_K}^{\w})^{2}   +  (\eta_{J_K}^{\w})^{2}, \quad \quad
	(\eta_{K}^{\u})^{2} := (\eta_{R_K}^{\u})^{2} 
\end{align*} 
with the interior residual terms defined as
\begin{align*}
	\begin{cases}
		\big(\eta_{R_{K}}
		^{\y}\big)^{2} &:= \|\rho_{K}(\f_{h} + \u_h + \nabla \cdot (\nu_h \nabla \y_h) - (\boldsymbol{\beta}_{h} \cdot \nabla) \y_h - \nabla p_h - \sigma_h \y_h)\|_{0,K}^{2}, \\
	\big(\eta_{R_{K}}^{\w}\big)^{2} &:= \|\rho_{K}(\y_{h} - \y_{d,h}  + \nabla \cdot (\nu_h \nabla \y_h) + (\boldsymbol{\beta}_{h} \cdot \nabla) \w_h + \nabla r_h - (\sigma_h-\nabla \cdot \boldsymbol{\beta}) 		\w_h)\|_{0,K}^{2},\\
	(\eta_{R_K}^{\u})^{2}  &:= h_{K}^{2} \|\w_h + \lambda \u_h\|_{0,K}^{2},
	\end{cases} 
\end{align*}
edge residuals defined as
 	\begin{align*}
	\begin{cases}
		\big(\eta_{E_{K}}^{\y}\big)^{2} &:= \frac{1}{2} \sum \limits_{E \in \partial K \setminus \Gamma} \|[\![\rho_E^{1/2} \nu_h^{-1/4}(p_h I - \nu \nabla \y_h)\cdot \n]\!]\|_{0,E}^{2},\\
	\big(\eta_{E_{K}}^{\w}\big)^{2} &:= \frac{1}{2} \sum \limits_{E \in \partial K \setminus \Gamma}  \|[\![\rho_E^{1/2} \nu_h^{-1/4}(r_h I - \nu \nabla \w_h)\cdot \n]\!]\|_{0,E}^{2},
	\end{cases} 
\end{align*}
and the trace residuals defined as
 \begin{align*}
	\big(\eta_{J_{K}}^{\y}\big)^{2} &:= \frac{1}{2} \sum \limits_{E \in \partial K \setminus \Gamma} \Big(\frac{\gamma}{h_E} \|[\![\nu_h^{1/2} \y_h \otimes \n ]\!]\|_{0,E}^{2} + \kappa h_{E} \|[\![\y_h \otimes \n ]\!]\|_{0,E}^{2} + h_E \|[\![\nu_h^{-1/2} \y_h \otimes \n ]\!]\|_{0,E}^{2} \Big) \vspace{1mm}\\
	& \ \ \  \ + \sum \limits_{E \in \partial K \cap \Gamma} \Big(\frac{\gamma}{h_E} \|\nu_h^{1/2} \y_h\|_{0,E}^{2} + \kappa h_{E} \|\y_h\|_{0,E}^{2} + h_E \|\nu_h^{-1/2} \y_h\|_{0,E}^{2} \Big),\vspace{1mm} 
\end{align*}
 where $I$ is the $d \times d$ identity matrix with $d=2,3$, and $\big(\eta_{J_{K}}^{\w}\big)^{2} =\big(\eta_{J_{K}}^{\y}\big)^{2}|_{\y=\w}$. 
For $(\y_h, p_h, \w_h, r_h, \u_h) \in \boldsymbol{V}_h \times Q_h \times \boldsymbol{V}_h \times Q_h \times \boldsymbol{U}_{ad,h}$ and $K \in \mathcal{T}_h$, the \textbf{local data oscillation terms} $\Theta_{K}^{\y}$ and $\Theta_{K}^{\w}$ are defined as:
\begin{align*}
	(\Theta_{K}^{\y})^{2} &:= \big(\|\rho_{K}(\f-\f_h)\|_{0,K}^{2} + \|\nu^{-1/2}  (\nu-\nu_h)  \nabla \y_h\|_{0,K}^{2} + \| \rho_{K}((\boldsymbol{\beta}-\boldsymbol{\beta}_h) \cdot \nabla) \y_h \|_{0,K}^{2} + \| \rho_{K}(\sigma-\sigma_h)\y_h\|_{0,K}^{2}\big),\\
	(\Theta_{K}^{\w})^{2} &:= \big(\|\rho_{K}(\y_{d,h}-\y_d)\|_{0,K}^{2}  + \|\nu^{-1/2} (\nu-\nu_h)  \nabla \w_h\|_{0,K}^{2} + \|\rho_{K}((\boldsymbol{\beta}-\boldsymbol{\beta}_h) \cdot \nabla) \w_h \|_{0,K}^{2}\\
	& \ \ \ \  + \| \rho_{K}((\sigma - \nabla \cdot \boldsymbol{\beta})-(\sigma_h - \nabla \cdot \boldsymbol{\beta}_h))\w_h\|_{0,K}^{2}\big).
\end{align*} 
Finally, we define the \textbf{global error estimators} $\eta^{\y}, \eta^{\w}, \eta^{\u}$ and \textbf{data oscillation errors} $\Theta^{\y}$, $\Theta^{\w}$  as
\begin{align}
\label{5.3}	(\eta^{\y})^2 &:= \sum_{K \in \mathcal{T}_h}(\eta_{K}^{\y})^{2} , &&  (\eta^{\w})^2 := \sum_{K \in \mathcal{T}_h}(\eta_{K}^{\w})^{2}, \quad \quad \quad \quad \eta^{\u} := \sum_{K \in \mathcal{T}_h}(\eta_{K}^{\u})^{2} ,\\
\label{5.4}	 (\Theta^{\y})^2 &:= \sum_{K \in \mathcal{T}_h}(\Theta_{K}^{\y})^{2}, && (\Theta^{\w})^2 := \sum_{K \in \mathcal{T}_h}(\Theta_{K}^{\w})^{2}.
\end{align}
\subsection{Auxiliary forms and their properties.}
The discrete form, $a_h(\y,\v)$, is only well-defined for functions $\y$ and $\v$ that belong to both the Sobolev space $\boldsymbol{H}_{0}^{1}(\Omega)$ and the finite element space $\boldsymbol{V}_h$. This limitation arises when using functions solely from the Sobolev space $\boldsymbol{H}^1_0(\Omega)$. To overcome this, we employ a lifting operator as described in  \cite{EKGD}. Our method involves splitting the discrete form into multiple components, eliminating the requirement for continuity estimates in the consistency and symmetrization terms. To begin, we introduce the following forms:
\begin{align*}
	D_1(\y,\v) &= \sum_{K \in \mathcal{T}_h} \int_{K} (\nu \nabla \y : \nabla \v + (\sigma - \nabla \cdot \boldsymbol{\beta})\y \cdot \v) \ dx,\\
	D_2(\y,\v) &= - \sum_{K \in \mathcal{T}_h} \Bigg( \int \limits_{K}  \y \boldsymbol{\beta}^{T}: \nabla \mathbf{v} \ dx - \int \limits_{\partial K_{\text{out}} \cap \Gamma_{\text{out}}} (\boldsymbol{\beta} \cdot \boldsymbol{n}_{K}) \y \cdot \mathbf{v} ds - \int \limits_{\partial K_{\text{out}} \setminus \Gamma} (\boldsymbol{\beta} \cdot \boldsymbol{n}_{K}) \y \cdot (\mathbf{v}-\mathbf{v}^{e}) ds \Bigg),\\
	D_3(\y,\v) &= - \sum_{E \in \mathcal{E}(\mathcal{T}_h)} \int_{E}  \Big( \{\!\!\{\nu \nabla \mathbf{y}\}\!\!\} : [\![\v \otimes \mathbf{n}]\!] + \{\!\!\{\nu \nabla \mathbf{v}\}\!\!\} : [\![\y \otimes \mathbf{n}]\!] \Big) ds,\\
	D_4(\y,\v) &=  \sum_{E \in \mathcal{E}(\mathcal{T}_h)} \int_{E} \frac{\gamma}{h_E} [\![\nu^{1/2} \y \otimes \mathbf{n}]\!] : [\![\nu^{1/2} \v \otimes \mathbf{n}]\!] ds.
\end{align*}
The bilinear form $\tilde{\mathcal{A}}_{h}$ is specified as:
\begin{align*}
	\tilde{\mathcal{A}}_{h}(\y,\v) = D_1(\y,\v) + D_2(\y,\v) + D_4(\y,\v).
\end{align*}
This is well-defined for all $\y,\v \in \boldsymbol{V}(h) = \boldsymbol{V}_h +  \boldsymbol{H}_{0}^{1}(\Omega)$, and it satisfies:
\begin{align}
	\label{5.5}	\mathcal{A}_h(\y,p;\v,\phi) := \tilde{\mathcal{A}}_{h}(\y,\v) + D_3(\y,\v) + b_h(\v,p) + b_h(\y,\phi) = (\f+\u,\v).
\end{align}
\begin{lemma}\label{lem: Lemma 5.4.}
	The following estimates hold:
	\begin{align*}
		|D_1(\y,\v)| &\precsim |\!|\!|\y|\!|\!| |\!|\!|\v|\!|\!|,
		\quad & \y,\v \in \boldsymbol{V}(h), \\
		|D_4(\y,\v)| &\precsim |\!|\!|\y|\!|\!| |\!|\!|\v|\!|\!|, \quad & \y,\v \in \boldsymbol{V}(h), \\
		|D_2(\y,\v)| &\precsim |\y \boldsymbol{\beta}^{T}|_{*} |\!|\!|\v|\!|\!|, \quad & \y \in \boldsymbol{V}(h), \v \in  \boldsymbol{H}_{0}^{1}(\Omega),\\
		|\tilde{\boldsymbol{A}}_{h}(\y,\v) | &\precsim |\!|\!|\y|\!|\!| |\!|\!|\v|\!|\!|, \quad & \y,\v \in \boldsymbol{V}(h),\\
		|D_3(\y,\v)| &\precsim \gamma^{-1/2}  \bigg(\sum_{E \in \mathcal{E}(\mathcal{T}_h)} \int_{E} \frac{\gamma}{h_E} \|[\![\nu^{1/2} \y \otimes \mathbf{n}]\!]\|_{0,E}^{2} ds \bigg)^{1/2}|\!|\!|\v|\!|\!|, \; &\y \in \boldsymbol{V}(h), \  \ \v \in \boldsymbol{H}_{0}^{1}(\Omega) \cap \boldsymbol{V}_h.
	\end{align*}
\end{lemma}
\begin{proof}
The first and second bounds directly result from the Cauchy-Schwarz inequality. To establish the third bound, we utilize the definition of $|\y \boldsymbol{\beta}^{T}|_{*}$. The fourth bound is derived by combining the first three bounds. The final bound is obtained by employing the Cauchy-Schwarz inequality and the inverse trace estimate.
\end{proof}
\subsection{Reliability}
Initially, we establish a reliability estimate for the a posteriori error estimator of the optimal control problem. Reliability estimates of a posteriori error estimators play a crucial role in improving the accuracy, efficiency, and credibility of numerical simulations. First, we extend the stability result \cite[Lemma~ 4.5.]{AGO} to the case where the viscosity coefficient is variable.
\begin{lemma}\label{lem: Lemma 5.5.}
        For any $(\y, p) \in \boldsymbol{H}_{0}^{1}(\Omega) \times L_{0}^{2}(\Omega)$, there exists $(\v, \phi) \in \boldsymbol{H}_{0}^{1}(\Omega) \times L_{0}^{2}(\Omega)$ such that $|\!|\!|(\mathbf{v},\phi)|\!|\!| \le 1$ and:
 	\begin{align*}
 		\mathcal{A}_h(\y,p;\v,\phi) \succsim |\!|\!|(\mathbf{y},p)|\!|\!| + |\y \boldsymbol{\beta}^{T}|_{*}.
 	\end{align*}
\end{lemma}
\begin{proof}
By applying the inf-sup condition along with continuity and coercivity, and considering the definition of norms, we derive the stated result.
 \end{proof}
To establish the reliability estimate (for the state and co-state problem), we adopt the approach outlined in \cite[Section~4]{AGO} and uniquely decompose the DG velocity approximation into:
\begin{align*}
	\y_h &= \y_h^{c} + \y_h^{r},\quad
	\w_h = \w_h^{c} + \w_h^{r},
\end{align*}
where $\y_h^{c},\w_h^{c} \in \boldsymbol{V}_{h}^{c}$ and $\y_h^{r},\w_h^{r} \in \boldsymbol{V}_{h}^{\perp}$. As $\y_h, \w_h \in \boldsymbol{V}_h$ and $\y_h^{c}, \w_h^{c} \in \boldsymbol{V}_{h}^{c} \subset \boldsymbol{V}_{h}$, so $\y_h^{r} = \y_h - \y_h^{c} \in \boldsymbol{V}_{h}$ and $\w_h^{r} = \w_h - \w_h^{c} \in \boldsymbol{V}_{h}$.\\
By employing the triangle inequality, we can express:
\begin{align}
	\label{5.6a} 	|\!|\!|\y-\y_h|\!|\!|_h + |\y-\y_h|_{A} &\le |\!|\!|\y-\y_h^{c}|\!|\!|_h +  |\y-\y_h^{c}|_{A} + |\!|\!|\y_h^{r}|\!|\!|_h +  |\y_h^{r}|_{A},\\
	\label{5.6b} 	|\!|\!|\w-\w_h|\!|\!|_h + |\w-\w_h|_{A} &\le |\!|\!|\w-\w_h^{c}|\!|\!|_h +  |\w-\w_h^{c}|_{A} + |\!|\!|\w_h^{r}|\!|\!|_h +  |\w_h^{r}|_{A}.
\end{align}
In the upcoming results, we illustrate that the continuous errors, the remaining terms, and the pressure errors can be reliably constrained by the global error estimators and the data oscillation errors associated with both the state and co-state. Firstly, we establish an upper bound for the remaining terms.
\begin{lemma}\label{lem: Lemma 5.6.}
	For the remaining terms $\y_h^{r}$ and $\w_h^{r}$, the following estimates hold:
	\begin{align}
		\label{5.7} 	|\!|\!|\y_h^{r}|\!|\!|_h +  |\y_h^{r}|_{A} \le \eta^{\y},\quad
			|\!|\!|\w_h^{r}|\!|\!|_h +  |\w_h^{r}|_{A} \le \eta^{\w}.
	\end{align}
\end{lemma}
\begin{proof}
Employing \cite[Lemma~4.6]{AGO} and the definition of the trace residuals $\eta_{J_{K}}^{\y}$ and $\eta_{J_{K}}^{\w}$ implies the stated result.
\end{proof}
We will now establish an upper bound for both the continuous error terms and the pressure error terms.
\begin{lemma}\label{lem: Lemma 5.7.}
The following estimates
\begin{align}
	\label{5.13} 	\int_{\Omega} (\f+\u_h) \cdot (\v - \boldsymbol{I}_h \v) dx - \tilde{\mathcal{A}}_h (\y_h,\v - \boldsymbol{I}_h \v) - b_h(\v - \boldsymbol{I}_h \v,p_h) \precsim (\eta^{\y} + \Theta^{\y})|\!|\!|\mathbf{v}|\!|\!|, \\
	\label{5.14} 	\int_{\Omega} (\y_{d}-\y_h) \cdot (\z - \boldsymbol{I}_h \z) dx - \tilde{\mathcal{A}}_h (\w_h,\z - \boldsymbol{I}_h \z) + b_h(\z - \boldsymbol{I}_h \z,r_h) \precsim (\eta^{\w} + \Theta^{\w})|\!|\!|\mathbf{z}|\!|\!|,
\end{align}
hold for all $\v, \z \ \text{in} \ \boldsymbol{V}$, where $\boldsymbol{I}_h$ represents the standard Scott-Zhang interpolant discussed in \cite{GFB}, which adheres to the following estimates:
\begin{align}
	\label{5.15} 	\bigg(\sum_{K \in \mathcal{T}_h} \|\rho_{K}^{-1}(\v - \boldsymbol{I}_h 		\v)\|_{0,K}^{2}\bigg)^{1/2} &\precsim |\!|\!|\mathbf{v}|\!|\!|, \quad
	\bigg(\sum_{E \in \mathcal{E}^{i}(\mathcal{T}_h)} \|[\![\rho_{E}^{-1/2} \nu^{1/4} (\v - \boldsymbol{I}_h \v)]\!]\|_{0,E}^{2} \bigg)^{1/2} \precsim |\!|\!|\mathbf{v}|\!|\!|. 
\end{align}
\end{lemma}
\begin{proof}
We commence with the definition of
\begin{align*}
	T = \int_{\Omega} (\f+\u_h) \cdot (\v - \boldsymbol{I}_h \v) dx - \tilde{\mathcal{A}}_h (\y_h,\v - \boldsymbol{I}_h \v) - b_h(\v - \boldsymbol{I}_h \v,p_h).
\end{align*}
Employing integration by parts, we obtain
\begin{align*}
	T &= \underbrace{\sum_{K \in \mathcal{T}_h} \int_{\Omega} (\f + \u_h + \nabla \cdot (\nu \nabla \y_h) - (\boldsymbol{\beta} \cdot \nabla) \y_h - \sigma \y_h - \nabla p_h)\cdot (\v - \boldsymbol{I}_h \v) dx}_{T_1} \\
	& \ \ \ \ +\underbrace{ {\sum_{K \in \mathcal{T}_h} \int_{\partial K} ((p_h I- \nu \nabla \y_h)\cdot \mathbf{n}_K)\cdot (\v - \boldsymbol{I}_h \v) ds}}_{T_2} + \underbrace{{\sum_{K \in \mathcal{T}_h} \int_{\partial K_{\text{in}} \setminus \Gamma} \boldsymbol{\beta}\cdot \mathbf{n}_K (\y_h-\y_h^{e})\cdot (\v - \boldsymbol{I}_h \v) ds}}_{T_3}.
\end{align*}
Initially, we incorporate both addition and subtraction of data approximation terms into $T_1$.  
Subsequently, by applying the Cauchy-Schwarz inequality and utilizing (\ref{5.15}), we obtain:
\begin{align*}
	T_1 
	&\precsim \bigg(\sum_{K \in \mathcal{T}_h} ((\eta_{R_{K}}^{\y})^{2} + (\Theta_{K}^{\y})^{2})\bigg)^{1/2} |\!|\!|\mathbf{v}|\!|\!|.
\end{align*}
The second term $T_2$ can be expressed as follows:
\begin{align*}
	T_2 = \sum_{E \in \mathcal{E}^{i}(\mathcal{T}_h)} \int_{E} [\![(p_h I- \nu \nabla \y_h)\cdot \n]\!] \cdot (\v - \boldsymbol{I}_h \v) ds.
\end{align*}
Applying a similar rationale, we obtain:
\begin{align}
	T_2 
\label{5.18} 	&\precsim \bigg(\sum_{E \in \mathcal{E}^{i}(\mathcal{T}_h)} (\eta_{E_{K}}^{\y})^{2} \bigg)^{1/2} |\!|\!|\mathbf{v}|\!|\!|.
\end{align}
Applying a comparable reasoning to $T_3$ yields:
\begin{align}
	T_3 
\label{5.19} 	&\precsim \bigg(\sum_{E \in \mathcal{E}^{i}(\mathcal{T}_h)} (\eta_{J_{K}}^{\y})^{2} \bigg)^{1/2} |\!|\!|\mathbf{v}|\!|\!|.
\end{align}
Combining the estimates for $T_1, T_2$, and $T_3$, we achieve the desired result. The second estimate for the co-state follows a similar line of reasoning.
\end{proof}
\begin{lemma}\label{lem: Lemma 5.8.}
	The following estimates holds
	\begin{align}
	\label{5.20} 	|\!|\!|\y-\y_h^{c}|\!|\!|_h +  |\y-\y_h^{c}|_{A} + \|\nu^{-1/2} (p-p_h)\|_{\mathcal{T}_h} &\precsim \eta^{\y} + \Theta^{\y},\\
	\label{5.21} 	|\!|\!|\w-\w_h^{c}|\!|\!|_h +  |\w-\w_h^{c}|_{A} + \|\nu^{-1/2} (r-r_h)\|_{\mathcal{T}_h} &\precsim \eta^{\w} + \Theta^{\w}.
	\end{align}
\end{lemma}
\begin{proof}
As $|\y-\y_h^{c}|_{A} = |\y-\y_h^{c}|_{*}$ and $(\y-\y_h^{c},p-p_h) \in \boldsymbol{H}_{0}^{1}(\Omega) \times L_{0}^{2}(\Omega)$, employing Lemma-\ref{lem: Lemma 5.5.}, we obtain
\begin{align}
	\label{5.22}  	|\!|\!|\y-\y_h^{c}|\!|\!|_h +  |\y-\y_h^{c}|_{*} + \|\nu^{-1/2} (p-p_h)\|_{\mathcal{T}_h} \precsim	\mathcal{A}_{h}(\y-\y_h^{c},p-p_h;\v,\phi)
\end{align}
for some $(\v,\phi)\in \boldsymbol{H}_{0}^{1}(\Omega) \times L_{0}^{2}(\Omega)$ with 	$|\!|\!|(\v,\phi)|\!|\!| \le 1$.\\
Using 
\begin{align}
\label{5.23n}	0 &= - \int_{\Omega} (\f+\u) \cdot \boldsymbol{I}_h \v dx + \mathcal{A}_h(\y_h,p_h;\boldsymbol{I}_h \v,0) ,
\end{align}
we find that:
\begin{align}
	\mathcal{A}_{h}(\y-\y_h^{c},p-p_h;\v,\phi) &= \underbrace{{\int_{\Omega} (\f+\u) \cdot (\v-\boldsymbol{I}_h \v) dx - \tilde{\boldsymbol{A}}_h(\y_h,\v-\boldsymbol{I}_h \v) - b_h(\v-\boldsymbol{I}_h \v,p_h)}}_{T_1}\nonumber\\
	& \label{5.25}\ \ \ + \underbrace{{D_1(\y_h^{r},\v) + D_4(\y_h^{r},\v) + D_2(\y_h^{r},\v) + b_h(\y_h^{r},\phi)}}_{T_2}+ \underbrace{D_3(\y_h,\boldsymbol{I}_h \v)}_{T_3}.
\end{align}
By utilizing the estimate (\ref{5.13}), we get:
\begin{align*}
	T_1 \precsim (\eta^{\y} + \Theta^{\y})|\!|\!|\mathbf{v}|\!|\!|.
\end{align*}
By employing Lemmas-\ref{lem: Lemma 5.4.} and \ref{lem: Lemma 5.6.}, in addition to applying the Cauchy-Schwarz inequality for the second term, we obtain:
\begin{align*}
	T_2 \precsim \eta^{\y} |\!|\!|(\mathbf{v},\phi)|\!|\!|.
\end{align*}
The final estimate in Lemma-\ref{lem: Lemma 5.4.} yields:
\begin{align*}
	T_3 \precsim \gamma^{-1/2} \bigg(\sum_{K \in \mathcal{T}_h} (\eta_{J_{K}})^{2}\bigg)^{1/2} |\!|\!|\mathbf{v}|\!|\!|.
\end{align*}
By incorporating the three bounds for $T_1, T_2$, and $T_3$, along with (\ref{5.25}), into (\ref{5.22}), we achieve the desired estimate (\ref{5.20}). Similar reasoning allows for the derivation of the second estimate.
\end{proof}
\begin{lemma}\label{lem: Lemma 5.9.}
Suppose $(\y,p)$, $(\w,r)$ be solutions of (\ref{2.13a}-\ref{2.13b}), (\ref{2.13c}-\ref{2.13d}), respectively, and $(\y_h,p_h)$, $(\w_h,r_h)$ be solutions of (\ref{3.1a}-\ref{3.1b}), (\ref{3.1c}-\ref{3.1d}), respectively. Then,
\begin{align}
\label{5.26} 	|\!|\!|\y-\y_h|\!|\!|_h + |\y-\y_h|_{A} + \|\nu^{-1/2}(p-p_h)\|_{0} &\precsim \eta^{y} + \Theta^{y},\\
\label{5.27} 	|\!|\!|\w-\w_h|\!|\!|_h + |\w-\w_h|_{A} + \|\nu^{-1/2}(r-r_h)\|_{0} &\precsim \eta^{w} + \Theta^{w}.
\end{align}
\end{lemma}
\begin{proof}
	These two estimates can be derived by employing Lemmas-\ref{lem: Lemma 5.6.}, \ref{lem: Lemma 5.7.}, and \ref{lem: Lemma 5.8.} in the equations (\ref{5.6a}-\ref{5.6b}).
\end{proof}
\begin{theorem}\label{thm: Theorem 5.11.}
Suppose that $(\y,p,\w,r,\mathbf{u}) \in  \boldsymbol{V} \times Q \times  \boldsymbol{V} \times Q \times U_{ad}$ solves the system (\ref{2.13a}-\ref{2.13e}) and $(\y_h,p_h,\w_h,r_h,\mathbf{u}_h) \in  \boldsymbol{V}_h \times Q_h \times  \boldsymbol{V}_h \times Q_h \times U_{ad,h}$ solves the system (\ref{3.1a}-\ref{3.1e}). If the error estimators and the data approximation errors are as defined in (\ref{5.3}-\ref{5.4}), then we have the following posteriori error bound 
\begin{align}
\nonumber	\|\u-\u_h\|_{0} + &|\!|\!|\y-\y_h|\!|\!|_h + |\y-\y_h|_{A} + \|\nu^{-1/2}(p-p_h)\|_{0} \\
\label{5.29} 	+  &|\!|\!|\w-\w_h|\!|\!|_h + |\w-\w_h|_{A} + \|\nu^{-1/2}(r-r_h)\|_{0} \precsim \eta^{\y} + \eta^{\w} + \Theta^{\y} + \Theta^{\w} + \eta^{\u}.
\end{align}
\end{theorem}
\begin{proof}
\textbf{Step -1:} For given $\u_h \in \boldsymbol{L}^{2}(\Omega)$, let $(\y(\u_h), p(\u_h))$ be a solution of the system (\ref{4.22a}-\ref{4.22b}) and $(\w(\u_h), r(\u_h))$ be solution of the system (\ref{4.22c}-\ref{4.22d}).
Subtracting these set of equations from (\ref{2.13a}-\ref{2.13b}) and (\ref{2.13c}-\ref{2.13d}), respectively, we have
\begin{align*}
	a(\y-\mathbf{y}(\u_h),\mathbf{v}) + b(\mathbf{v},p-p(\u_h)) &= ( \u-\u_h , \mathbf{v})  && \forall \ \mathbf{v} \in  \boldsymbol{V}, \\
	\hspace{1.57cm}b(\y-\mathbf{y}(\u_h),\phi) &= 0 &&  \forall \ \phi \in Q,\\
	a(\z,\w-\w(\u_h)) - b(\mathbf{z},r-r(\u_h)) &= ( \mathbf{y}- \mathbf{y}(\u_h) , \mathbf{z})  \    \ && \forall \ \mathbf{z} \in  \boldsymbol{V}, \\
	\hspace{1.65cm}b(\w-\w(\u_h),\psi) &= 0 \   &&  \forall \ \psi \in Q.
\end{align*} 
Upon substituting $\v = \w-\w(\u_h),\ \phi = r - r(\u_h),\ \z = \y-\y(\u_h)$,\ and $\psi = p - p(\u_h)$ into the aforementioned set of equations, we obtain:
\begin{align}
	\label{5.30}  a(\y-\y(\u_h),\w-\w(\u_h)) &= (\u-\u_h,\w-\w(\u_h)) = \|\y-\y(\u_h)\|_{0}^{2} \ge 0.
\end{align}
\textbf{Step -2:} 
We now exploit the convexity of the linear quadratic optimal control problem to demonstrate a relationship between the control and the co-state variables. Let
\begin{align}
	\label{5.31} 	 (F'(\u),\v) &= (\lambda \u+ \w, \v) \ \ && \forall \ \v \in \boldsymbol{V},\\
	\label{5.32} 	 (F'(\u_h),\v) &= (\lambda \u_h + \w(\u_h), \v) \ \ && \forall \ \v \in \boldsymbol{V},
\end{align} where $\w$ is a solution of the co-state problem. 
After performing the subtraction of (\ref{5.32}) from (\ref{5.31}) and employing the substitution $\v = \u - \u_h$, we arrive at the expression:
\begin{align}
	\label{5.33}	 (F'(\u)-F'(\u_h),\u - \u_h) &= \lambda(\u - \u_h, \u - \u_h) + (\w-\w(\u_h),\u - \u_h).
\end{align}
By incorporating (\ref{5.30}) and the variational inequality (\ref{2.13e}) into (\ref{5.33}), we obtain the following:
\begin{align}
  \label{5.34d}  \lambda \|\u - \u_h\|_{0}^{2}  
		&\le (\w_h - \w(\u_h),\u-\u_h) - (\w_h + \lambda \u_h, \u-\v_h).
\end{align}
We consider $\v_h = I_h \u \in \boldsymbol{U}_{ad,h}$ and leverage Young's inequality in (\ref{5.34d}) to obtain:
\begin{align*}
	   \lambda \|\u - \u_h\|_{0}^{2} 
		  &\le C |\!|\!|\mathbf{w}_h-\w(\u_h)|\!|\!|_{h}^{2} +  \frac{\lambda}{2}\|\u-\u_h\|_{0}^{2} + C(\eta^{\u})^{2}.
\end{align*}
As a consequence of the aforementioned equation, we arrive at 
\begin{align}
	\label{5.35} 	  \|\u - \u_h\|_{0}  \precsim \ & \eta^{\u} + |\!|\!|\mathbf{w}_h-\w(\u_h)|\!|\!|_{h}.
\end{align}
This result elucidates the relationship between the control and the co-state velocity.\\
\textbf{Step -3:} Now we construct the connection between state and co-state. Let $(\tilde{\w}, \tilde{r})$ solves (\ref{2.13c}-\ref{2.13d}) with $\y=\y_h$. Then $(\w(\u_h)-\tilde{\w}, r(\u_h)-\tilde{r})$ solves 
\begin{align*}
	a(\z,\w(\u_h)-\tilde{\w}) - b(\mathbf{z},r(\u_h)-\tilde{r}) &= (\mathbf{y}(\u_h) - \mathbf{y}_h , \mathbf{z}) && \forall \ \mathbf{z} \in  \boldsymbol{V},\\
	b(\w(\u_h)-\tilde{\w},\psi) &= 0 && \forall \ \phi \in Q.
\end{align*}
Leveraging Lemma-\ref{lem: Lemma 5.5.} on the continuous formulation, we arrive at
\begin{align}
\label{5.36} 	|\!|\!|\w(\u_h)-\tilde{\w}|\!|\!|_{h} + |\w(\u_h)-\tilde{\w}|_{A} + \|r(\u_h)-\tilde{r}\|_{0} \precsim \|\mathbf{y}(\u_h) - \mathbf{y}_h\|_{0}.
	\end{align}
Using Lemma-\ref{lem: Lemma 5.9.}, we end up with
\begin{align}
\label{5.37} 	|\!|\!|\tilde{\w}-\w_h|\!|\!|_h + |\tilde{\w}-\w_h|_{A} + \|\nu^{-1/2}(\tilde{r}-r_h)\|_{0} \precsim \eta^{w} + \Theta^{w}.
	\end{align}
By exploiting the triangle Inequality and the results presented in (\ref{5.36}) and (\ref{5.37}), we obtain
\begin{align}
	 |\!|\!|\w(\u_h)-\w_h|\!|\!|_h + |\w(\u_h)-\w_h|_{A} + \|\nu^{-1/2}(r(\u_h)-r_h)\|_{0} 
	&\precsim  \label{5.38} 	 \ \|\mathbf{y}(\u_h) - \mathbf{y}_h\|_{0}+\eta^{w} + \Theta^{w} .
\end{align}
In the context of the state equation, Lemma-\ref{lem: Lemma 5.9.} yields
\begin{align}
	\label{5.39} 	|\!|\!|\y(\u_h)-\y_h|\!|\!|_h + |\y(\u_h)-\y_h|_{A} + \|\nu^{-1/2}(p(\u_h)-p_h)\|_{0} \precsim \eta^{y} + \Theta^{y}.
\end{align}
Employing Lemma-\ref{lem: Lemma 5.5.} and incorporating (\ref{5.38}) and (\ref{5.39}) into (\ref{5.35}) implies the desired reliability result.
\end{proof}
\subsection{Efficiency estimates of a posteriori error estimators}
This subsection establishes efficiency estimates for discretization errors in the optimal control problem. The efficiency of the estimator implies a lower bound for these errors (up to data oscillations).
We initiate the process by proving efficiency estimates for both the state and co-state variables. 
\begin{lemma}\label{lem: Lemma 5.16.}
	Suppose $(\y,p)$, $(\w,r)$ be solutions of (\ref{2.13a}-\ref{2.13b}), (\ref{2.13c}-\ref{2.13d}), respectively, and $(\y_h,p_h)$, $(\w_h,r_h)$ be solutions of (\ref{3.1a}-\ref{3.1b}), (\ref{3.1c}-\ref{3.1d}), respectively. Then, we have 
	\begin{align}
	 \label{5.58} 	\eta^{y} &\precsim	|\!|\!|\y-\y_h|\!|\!|_h + |\y-\y_h|_{A} + \|\nu^{-1/2}(p-p_h)\|_{0}  + \Theta^{y},\\
	\label{5.59}     \eta^{w} 	&\precsim |\!|\!|\w-\w_h|\!|\!|_h + |\w-\w_h|_{A} + \|\nu^{-1/2}(r-r_h)\|_{0} + \Theta^{w}.
	\end{align}
\end{lemma}
	\begin{proof}
		Proof directly follows by deriving efficiency bounds for $\eta_{R_K}$, $\eta_{J_K}$, and $\eta_{E_K}$ through the application of the standard bubble function technique as presented in \cite{AGO,VERFP, VERFPO}.
	\end{proof}
\begin{theorem}
	Suppose that $(\y,p,\w,r,\mathbf{u}) \in  \boldsymbol{V} \times Q \times  \boldsymbol{V} \times Q \times U_{ad}$ solves the system (\ref{2.13a}-\ref{2.13e}) and $(\y_h,p_h,\w_h,r_h,\mathbf{u}_h) \in  \boldsymbol{V}_h \times Q_h \times  \boldsymbol{V}_h \times Q_h \times U_{ad,h}$ solves the system (\ref{3.1a}-\ref{3.1e}).  If the error estimators and the data approximation errors are as defined in (\ref{5.1}-\ref{5.3}), then we have the following lower bound
		\begin{align}
	\nonumber \eta^{\y} + \eta^{\w} + \eta^{\u} \precsim & \ \|\u-\u_h\|_{0} + |\!|\!|\y-\y_h|\!|\!|_h + |\y-\y_h|_{A} + \|\nu^{-1/2}(p-p_h)\|_{0}+ |\!|\!|\w-\w_h|\!|\!|_h \\
	\label{5.60} 	& + |\w-\w_h|_{A} + \|\nu^{-1/2}(r-r_h)\|_{0}   + \Theta^{\y} + \Theta^{\w}.
	\end{align}
\end{theorem}
\begin{proof}
	\textbf{Step -1:} Consider a solution $(\y(\u_h), p(\u_h))$ to  (\ref{4.22a}) - (\ref{4.22b}), and a solution $(\y_h, p_h)$ to equations (\ref{3.1a}) - (\ref{3.1b}). Applying Lemma-\ref{lem: Lemma 5.16.}, we arrive at the following:
	\begin{align}
	\label{5.61} 	 \eta^{\y} \precsim	& \ |\!|\!|\y(\u_h)-\y_h|\!|\!|_h + |\y(\u_h)-\y_h|_{A} + \|p(\u_h)-p_h\|_{0} + \Theta^{\y}.
	\end{align}
	Likewise, suppose $(\w(\u_h),r(\u_h))$ is a solution of (\ref{4.22c}-\ref{4.22d}), and $(\w_h,r_h)$ is a solution of (\ref{3.1c}-\ref{3.1d}). Then, Lemma-\ref{lem: Lemma 5.16.} yields:
	\begin{align}
		\label{5.62}  \eta^{\w} \precsim	& \ |\!|\!|\w(\u_h)-\w_h|\!|\!|_h + |\w(\u_h)-\w_h|_{A} + \|r(\u_h)-r_h\|_0 + \Theta^{\w}.
	\end{align}
	\textbf{Step -2:} Applying the triangle Inequality and Lemma-\ref{lem: Lemma 5.5.}, we obtain: 
			\begin{align}
	  \eta^{\w} 
		\label{5.63}   &\precsim	 |\!|\!|\w-\w_h|\!|\!|_h + |\w-\w_h|_{A}  + \|r-r_h\|_0 +\|\u - \u_h\|_{0}+ \Theta^{\w}.
	\end{align}
		Likewise, using (\ref{5.61}), we obtain:
	\begin{align}
	\label{5.64}  \eta^{\y}	  &\precsim	 |\!|\!|\y-\y_h|\!|\!|_{h} + |\y-\y_h|_{A}  + \|p-p_h\| +\|\u - \u_h\|_{0} + \Theta^{\y}.
	\end{align}
		By combining  (\ref{5.63}) and (\ref{5.64}), we obtain the desired bound.
\end{proof}
\section{Numerical Experiments}\label{Numerical Experiments.}
In this section, we showcase a series of numerical experiments on different polygonal domains carried out using the proposed scheme in Section \ref{Discrete Formulation.} to validate the theoretical findings outlined in sections \ref{PRIORI ERROR ESTIMATES.} and \ref{A POSTERIORI ERROR ESTIMATES.}. We use the pair $BDM_{1}-P_{0}$ to approximate the velocity and pressure, respectively. The solution of all linear systems is carried out with the multifrontal massively parallel sparse direct solver MUMPS in Fenics \cite{AMB}. We use the primal-dual active set strategy (PDASS) detailed extensively in \cite[Section~2.12.4]{FTO}. Additionally, we adopt an adaptive mesh refinement procedure inspired by \cite{VANAYA,VERFP} for Examples-\ref{Example 6.2.} and \ref{Example 6.3.}, with initial meshes depicted in Figure-\ref{FIGURE 01}. The control cost parameter is fixed at $\lambda = 1$ for all examples. The global estimator $\Upsilon$ and total error $|\!|\!|\mathbf{e}|\!|\!|_{\Omega}$ are defined as:
\begin{figure}
	\begin{center}
			\includegraphics[scale=0.38]{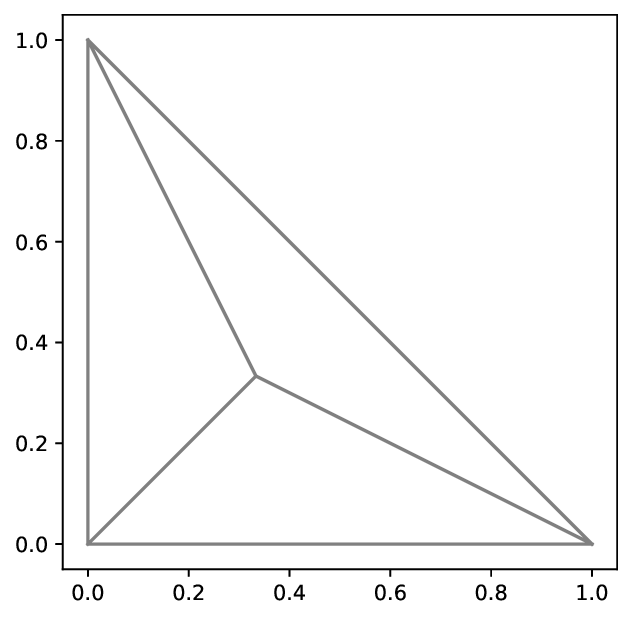} 
			\includegraphics[scale=0.38]{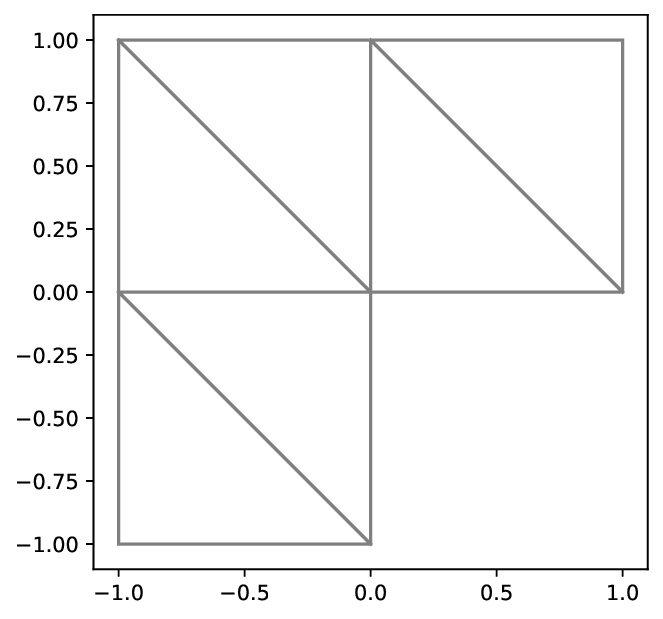}
			\includegraphics[scale=0.38]{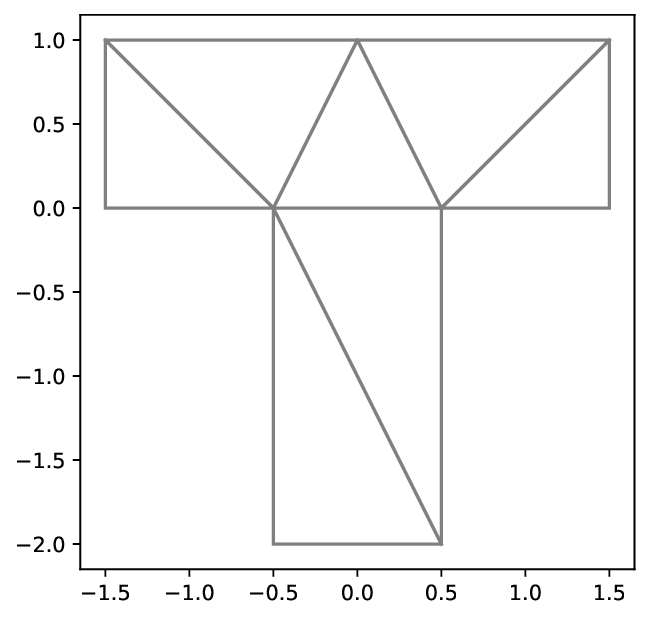}
		\end{center} \vspace{-1mm}
	\caption{ Initial uniform meshes used for Examples-\ref{Example 6.2.}, and \ref{Example 6.3.}.} \label{FIGURE 01}
\end{figure}
\begin{align*}
	\Upsilon = \bigg(\sum_{K \in \mathcal{T}_h} \Upsilon_{K}^{2}\bigg)^{1/2} &= \bigg( \sum_{K \in \mathcal{T}_h} \Big((\eta_{K}^{\y})^{2}+(\eta_{K}^{\w})^{2}+(\eta_{K}^{\u})^{2}\Big) \bigg)^{1/2}, \\
	|\!|\!|\mathbf{e}|\!|\!|_{\Omega} &= |\!|\!|(\mathbf{e}_{\y},\mathbf{e}_{p},\mathbf{e}_{\w},\mathbf{e}_{r},\mathbf{e}_{\u})|\!|\!|_{\Omega}.
\end{align*}
	  \subsection{Accuracy verification test}\label{Example 6.1.}
	  Consider the square domain $\Omega = (0,1)^{2}$, with coefficients $\nu =1+0.01 \xi_1^{2},\ \boldsymbol{\beta} = (\xi_1^{2},\xi_2^{2}),\ \sigma = 1,$ and control bounds  $\mathbf{u_a}=(-0.5,-0.5),$ and $ \mathbf{u_b}=(0.5,0.5)$. Choose $\mathbf{f}$ and $\mathbf{y}_d$ s.t.
	  \begin{align*}
	  	\y(\xi_1,\xi_2) &= \textbf{curl}\Big((\xi_1 (1-\xi_1)\xi_2 (1-\xi_2))^{2}\Big), \ p(\xi_1,\xi_2) = \cos(2\pi \xi_1)\cos(2\pi \xi_2),\\
	  	\w(\xi_1,\xi_2) &= \textbf{curl}\Big((\sin(2\pi\xi_1) \sin(2\pi\xi_2))^{2}\Big), \ \ r(\xi_1,\xi_2) = \cos(2\pi \xi_1)\cos(2\pi \xi_2),
	  \end{align*}
	  are the optimal solutions.
	   We achieve optimal convergence rates for state, co-state, and control variables through uniform refinement, as depicted in Figure-\ref{FIGURE 3}. Notably, both the global estimator $\Upsilon$ and the total error $|\!|\!|\mathbf{e}|\!|\!|_{\Omega}$ exhibit a decrease at the optimal rate, as illustrated in Figure-\ref{FIGURE 4}. The graphical representation reveals a parallel alignment between the graph of the estimator $\Upsilon$ and the error $|\!|\!|\mathbf{e}|\!|\!|_{\Omega}$, with their ratio, known as the efficiency index, consistently demonstrating a uniform behavior. This observation affirms the numerical reliability and efficiency of the proposed error estimator. Figure-\ref{FIGURE 2} presents plots for the numerical solutions of the state and co-state variables.
\begin{figure}
	\begin{center}
		\includegraphics[scale=0.52]{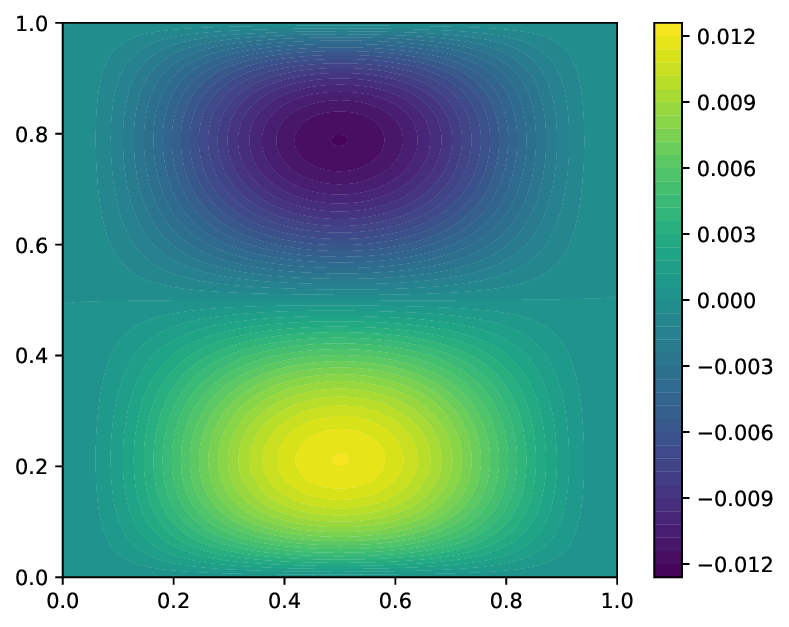}
		\includegraphics[scale=0.52]{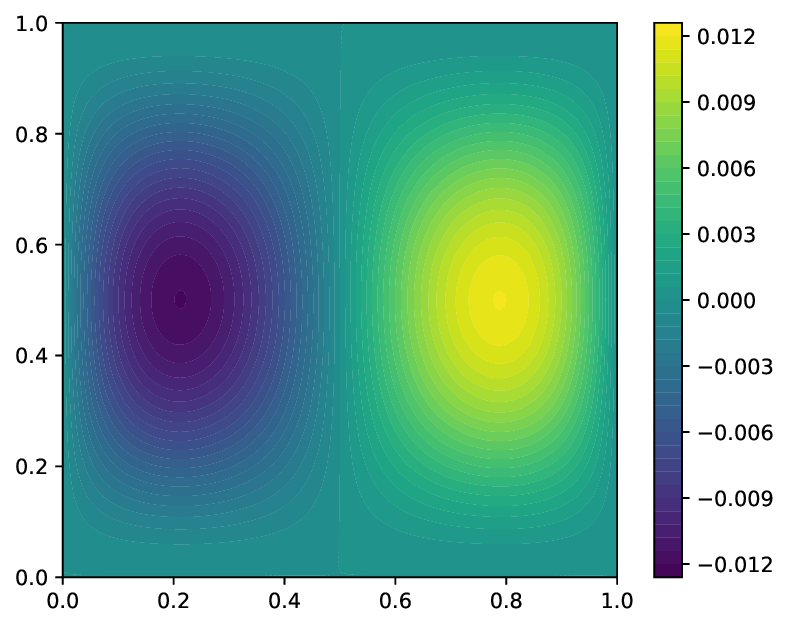}\\
		\includegraphics[scale=0.52]{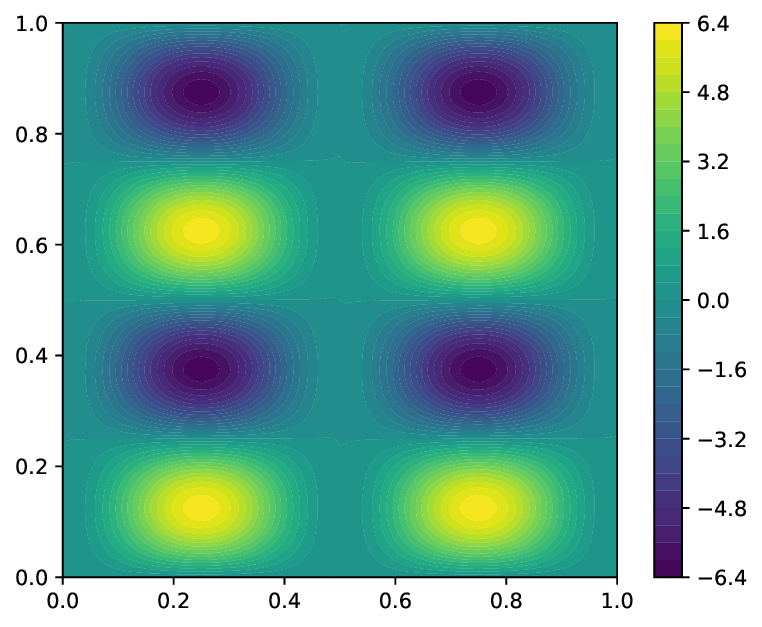}
		\includegraphics[scale=0.52]{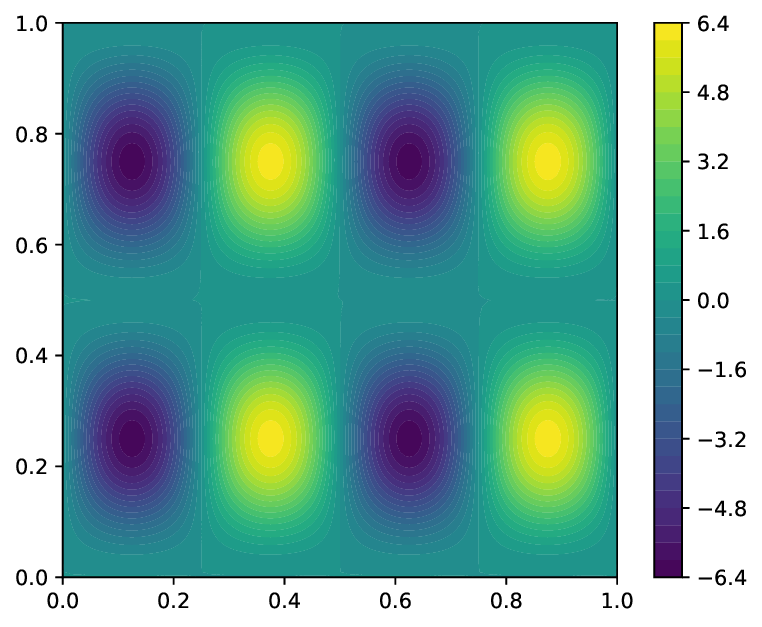}\\
		\includegraphics[scale=0.52]{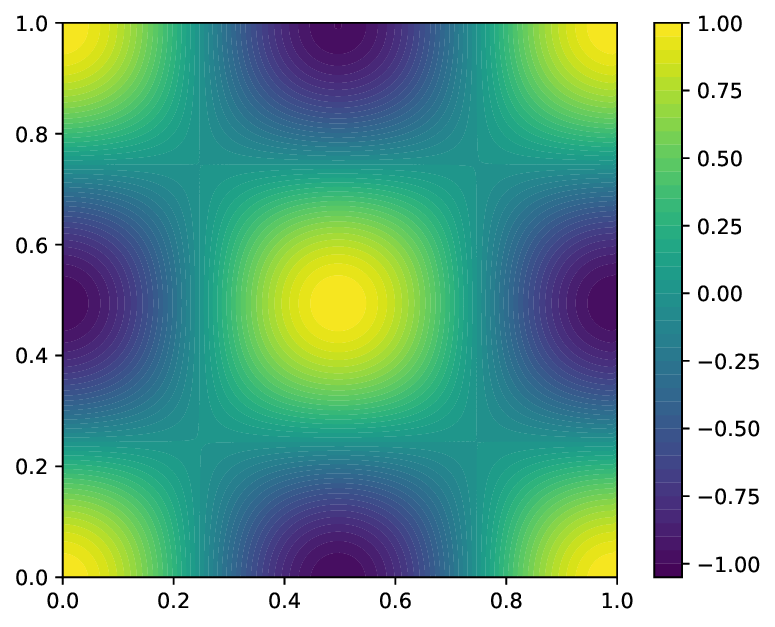} 
		\includegraphics[scale=0.52]{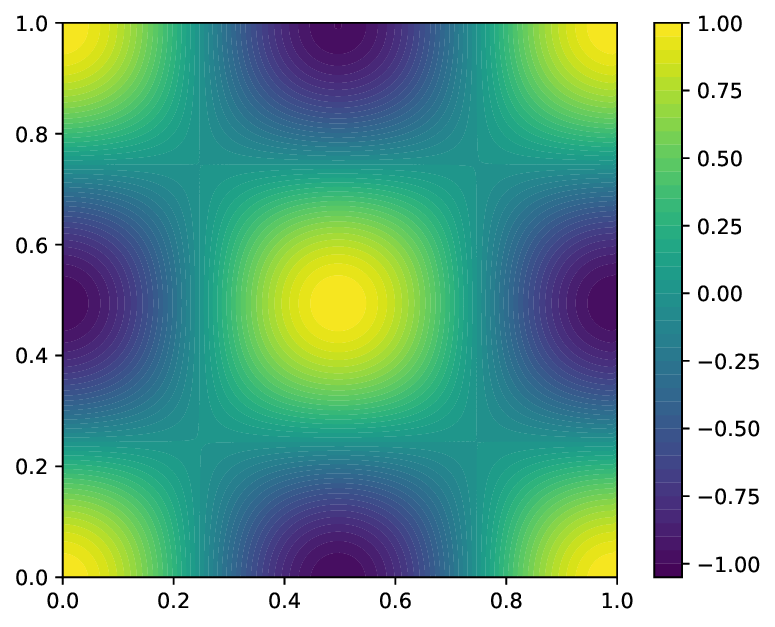} \\ 
		\includegraphics[scale=0.52]{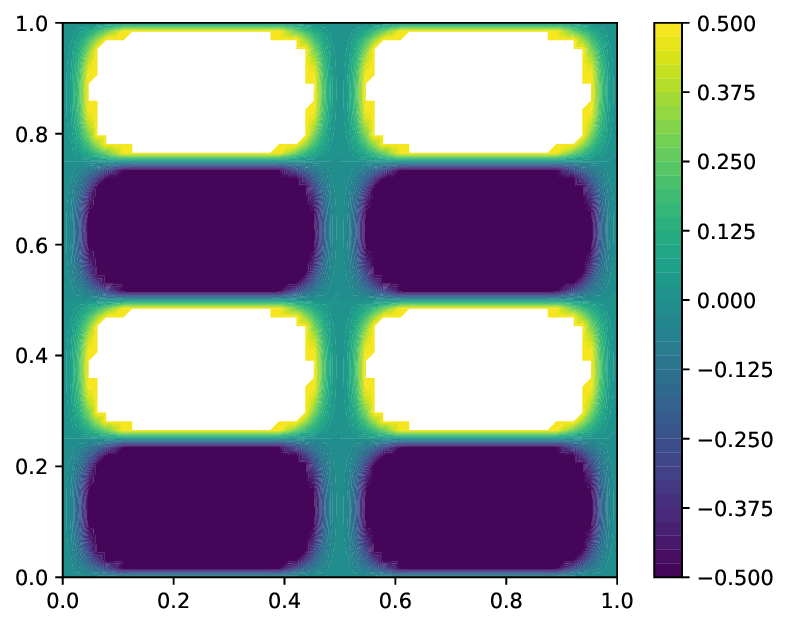}
		\includegraphics[scale=0.52]{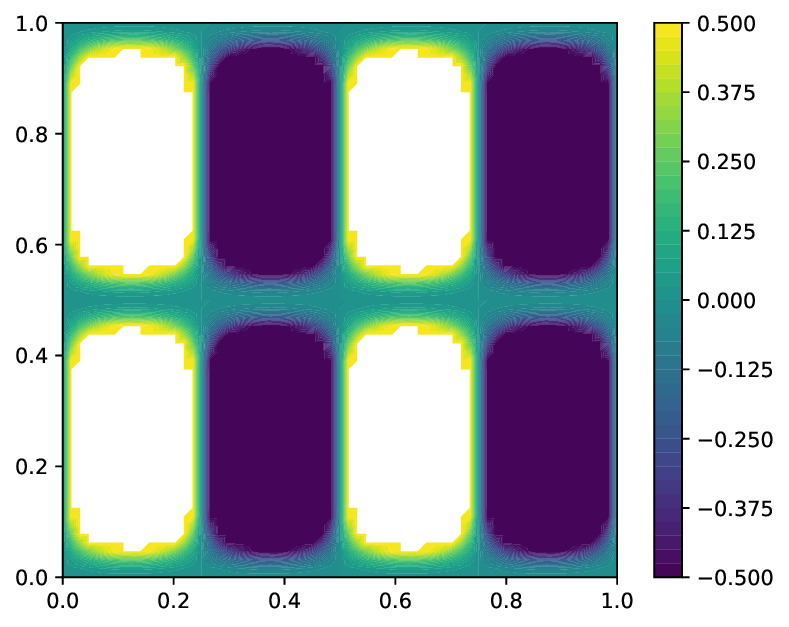}
	\end{center}\vspace{-1mm}
	\caption{Plots of numerical solutions of state velocity $(\y_{h1},\y_{h2})$, co-state velocity $(\w_{h1},\w_{h2})$, state pressure $(p_h)$, co-state pressure $(r_h)$, and control $(\u_{h1},\u_{h2})$, respectively, for Example- \ref{Example 6.1.}.} \label{FIGURE 2}
\end{figure}
\begin{figure}
	\begin{center}
		\includegraphics[scale=0.35]{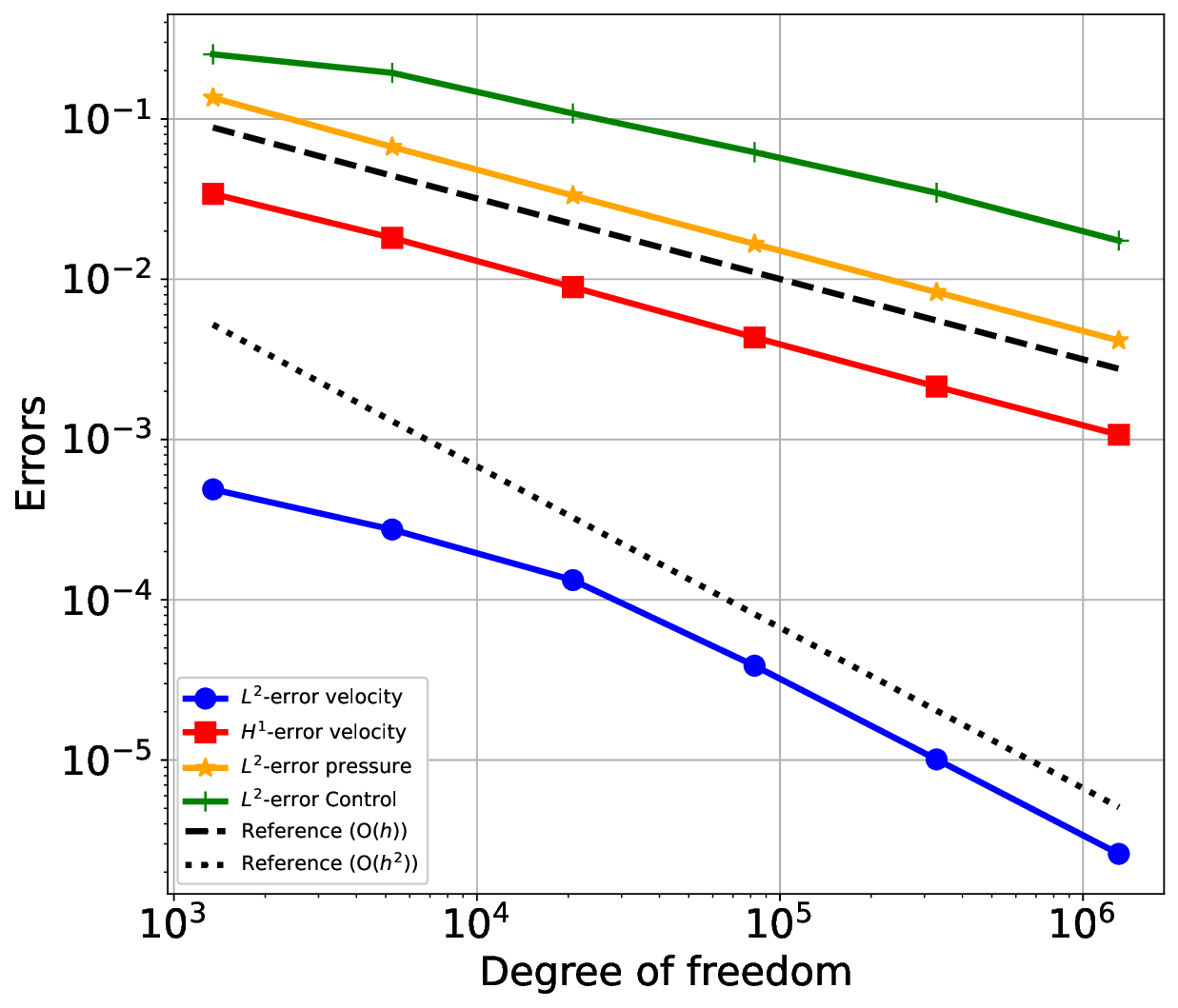}
		\includegraphics[scale=0.35]{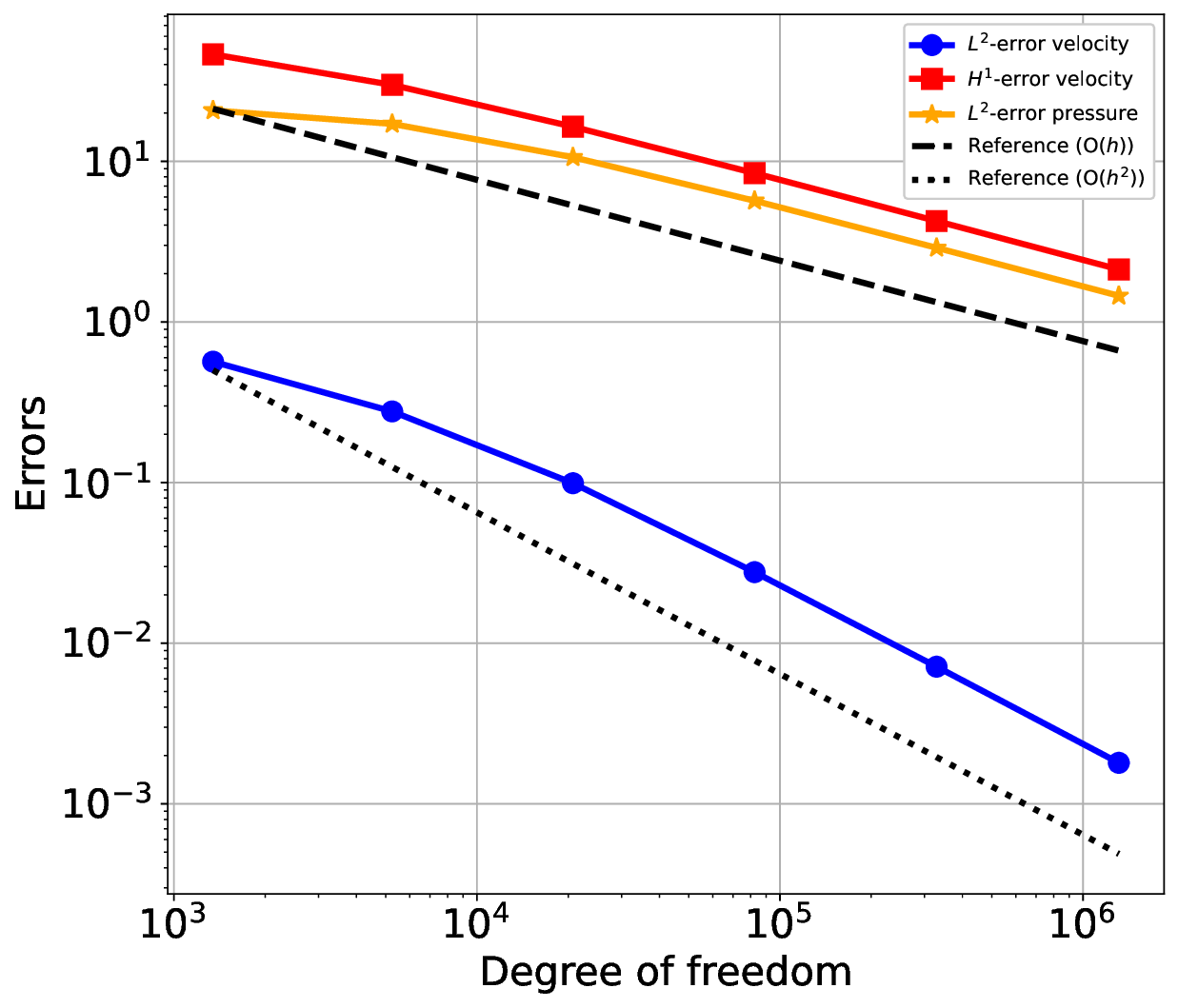} 
	\end{center}
	\caption{Convergence plots for the state and co-state variables for Example- \ref{Example 6.1.}.} \label{FIGURE 3}
\end{figure} 
\begin{figure}
	\begin{center}
 		\includegraphics[scale=0.35]{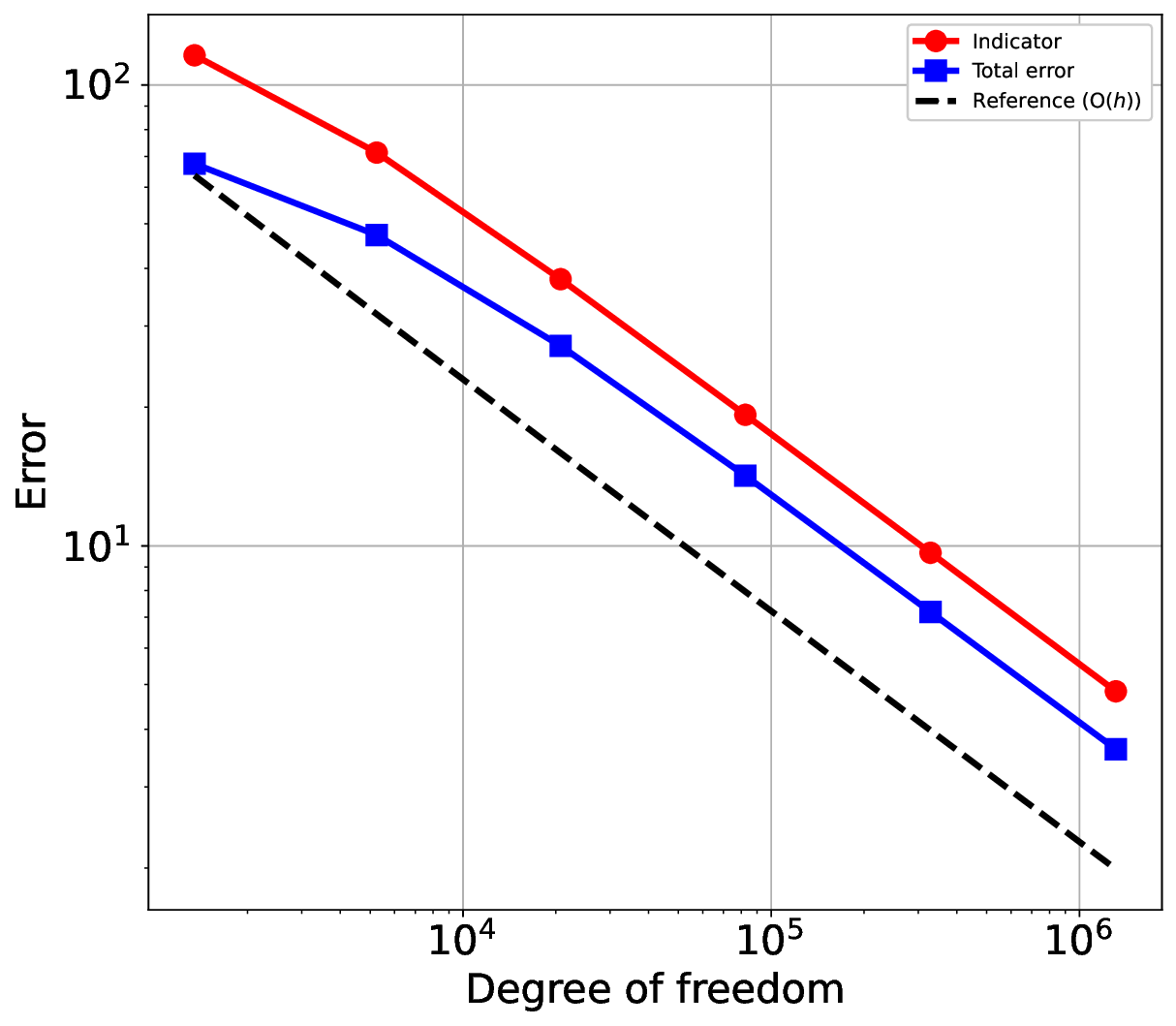}
		\includegraphics[scale=0.35]{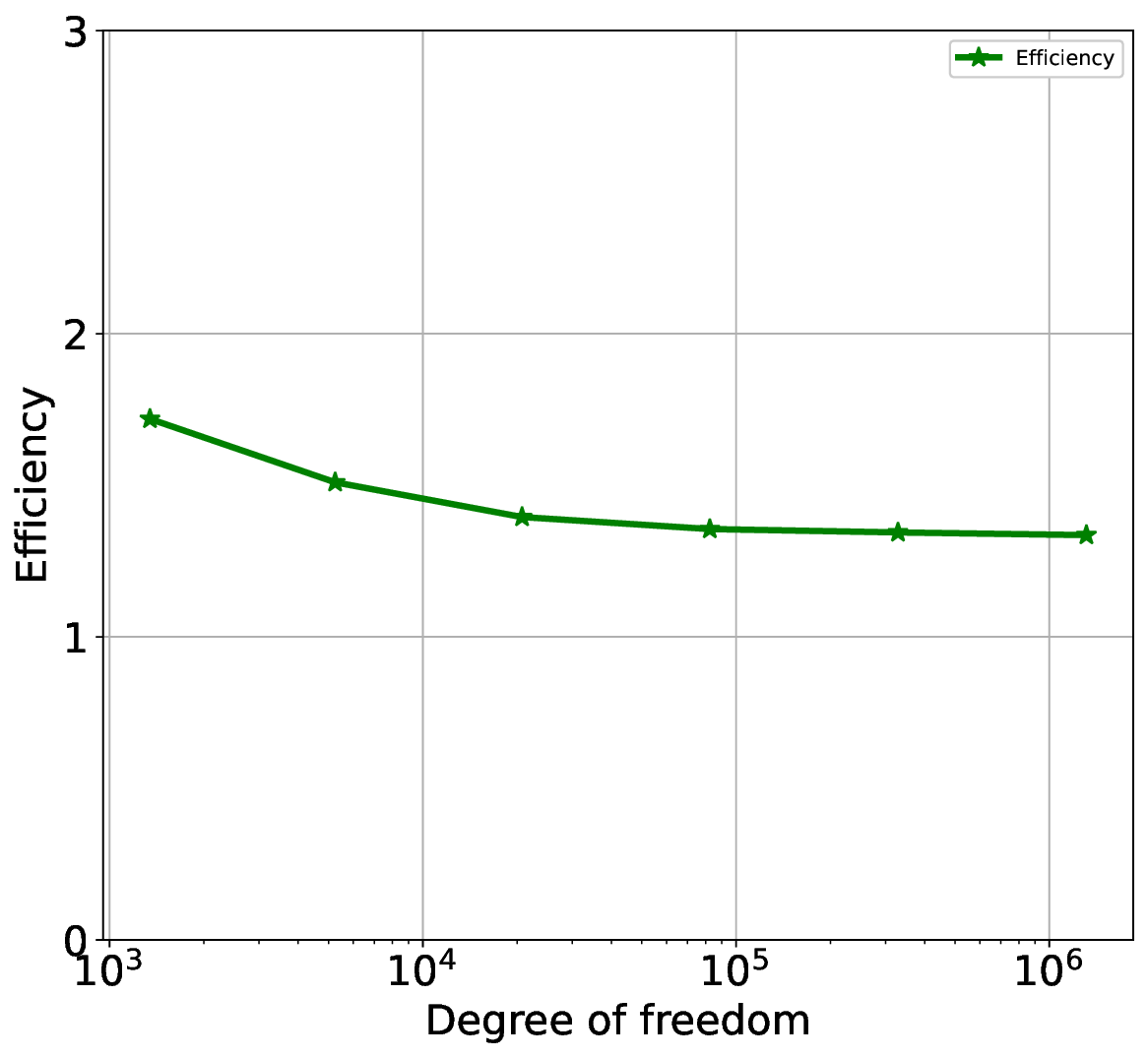}
	\end{center}
	\caption{Convergence plot for the indicator and total error (uniform refinement) for Example- \ref{Example 6.1.}.} \label{FIGURE 4}
\end{figure}
\subsection{Boundary layers problem}\label{Example 6.2.}
	 In this example, we consider the triangular domain $\Omega = \{(\xi_1,\xi_2): \xi_1>0, \ \xi_2>0, \ \xi_1 + \xi_2 <1\}$, with  the coefficients $\nu =0.01 \exp(-x_1^{2}-x_2^{2}),\ \boldsymbol{\beta} = (\xi_1,-\xi_2),\ \sigma = 1$, and control bounds $ \mathbf{u_a}=(0,0),$ $\mathbf{u_b}=(0.1,0.1)$. The source function $\mathbf{f}$ and desired function $\mathbf{y}_d$ are choosen such that
	\begin{align*}
		\y(\xi_1,\xi_2) &= \textbf{curl} \ \bigg(\xi_1 \xi_2^{2}(1-\xi_1-\xi_2)^{2} \bigg(1-\xi_1- \frac{\exp(-100\xi_1)-\exp(-100)}{1-\exp(-100)}\bigg)\bigg), \\
		\w(\xi_1,\xi_2) &=  \textbf{curl} \ \bigg(\xi_1^{2} \xi_2(1-\xi_1-\xi_2)^{2} \bigg(1-\xi_2- \frac{\exp(-100\xi_2)-\exp(-100)}{1-\exp(-100)}\bigg)\bigg),\\
		p(\xi_1,\xi_2) &= \frac{\cos(2\pi \xi_2)}{1024}, \ \ \ \ \ \   r(\xi_1,\xi_2) = \frac{\cos(2\pi \xi_1)}{1024},
	\end{align*}
are the optimal solutions. The presence of boundary layers in the solution hampers the convergence rates of the state and co-state variables under uniform refinement. To overcome this challenge, we implement an adaptive mesh refinement procedure, focusing on refining the region of boundary layers, as depicted in Figure-\ref{FIGURE 5}. Upon achieving sufficient refinement, we observe a commencement of optimal-rate decrease in both the global estimator $\Upsilon$ and the error $ |\!|\!|\mathbf{e}|\!|\!|_{\Omega} $ . With further refinements, the graph of the estimator $\Upsilon$ aligns parallel to the error $ |\!|\!|\mathbf{e}|\!|\!|_{\Omega} $, and the efficiency index exhibits a nearly constant behavior, as illustrated in Figure-\ref{FIGURE 7}. Optimal convergence rates for the state and co-state variables are restored, as shown in Figure-\ref{FIGURE 6}. This substantiates the numerical reliability and efficiency of the proposed error estimator. The numerical solution plots are presented in Figure-\ref{FIGURE 8}.
\begin{figure}
	\begin{center}
		\includegraphics[scale=0.48]{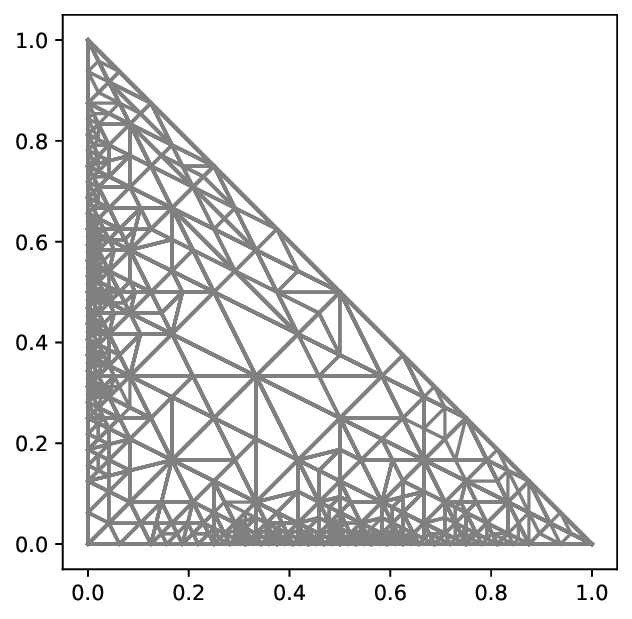}
		\includegraphics[scale=0.48]{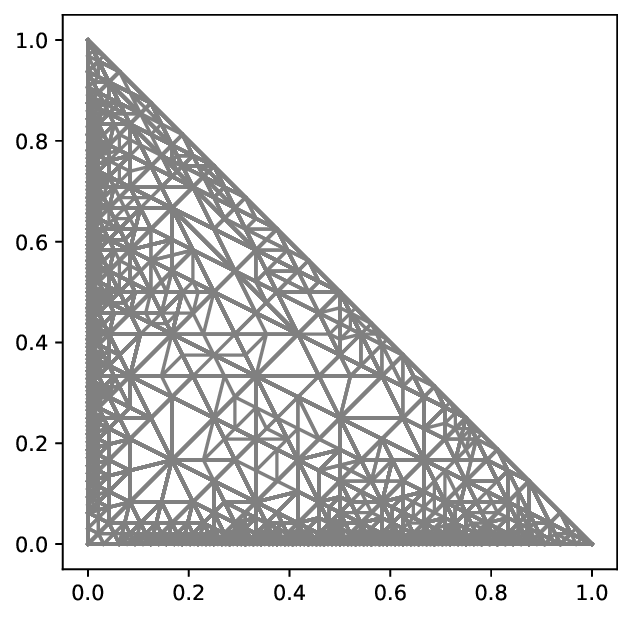}
		\includegraphics[scale=0.48]{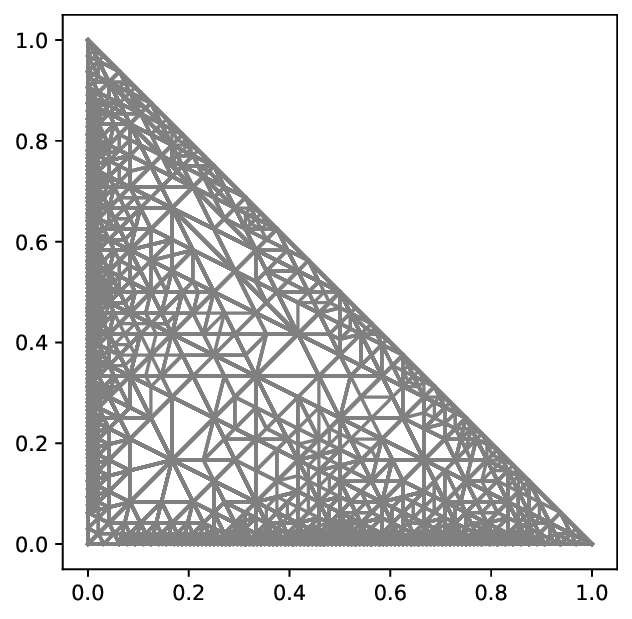}
	\end{center} \vspace{-3mm}
	\caption{Adaptively refined meshes (a) 5784 DOF (b) 17380 DOF (c)  27948 DOF (showing the boundary layers).} \label{FIGURE 5}
\end{figure}
\begin{figure}
	\begin{center}
		\includegraphics[scale=0.35]{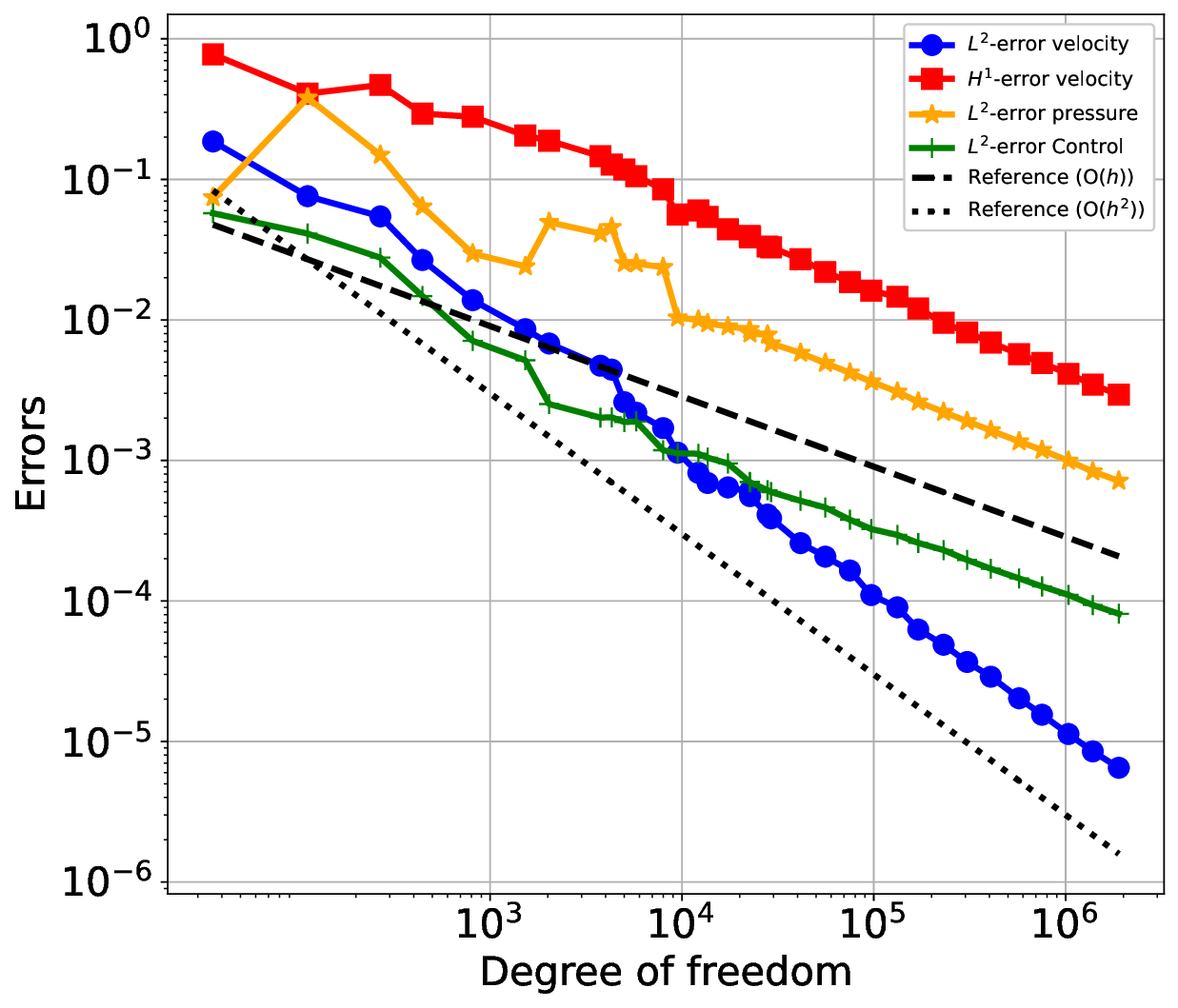}
		\includegraphics[scale=0.35]{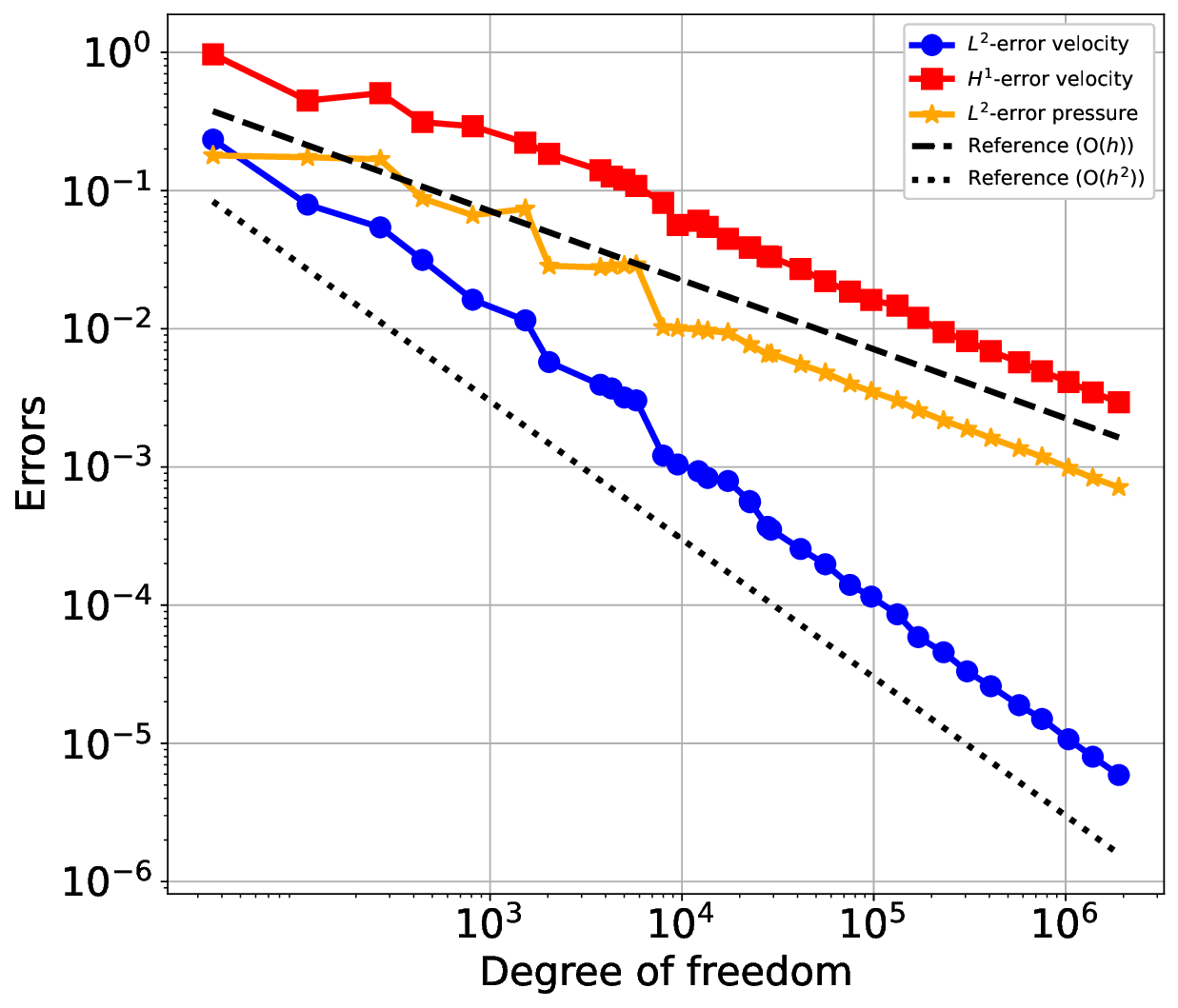} 
	\end{center}
	\caption{Convergence plots for the state and co-state variables (adaptive refinement) for Example- \ref{Example 6.2.}.} \label{FIGURE 6}
\end{figure}
\begin{figure}
	\begin{center}
		\includegraphics[scale=0.35]{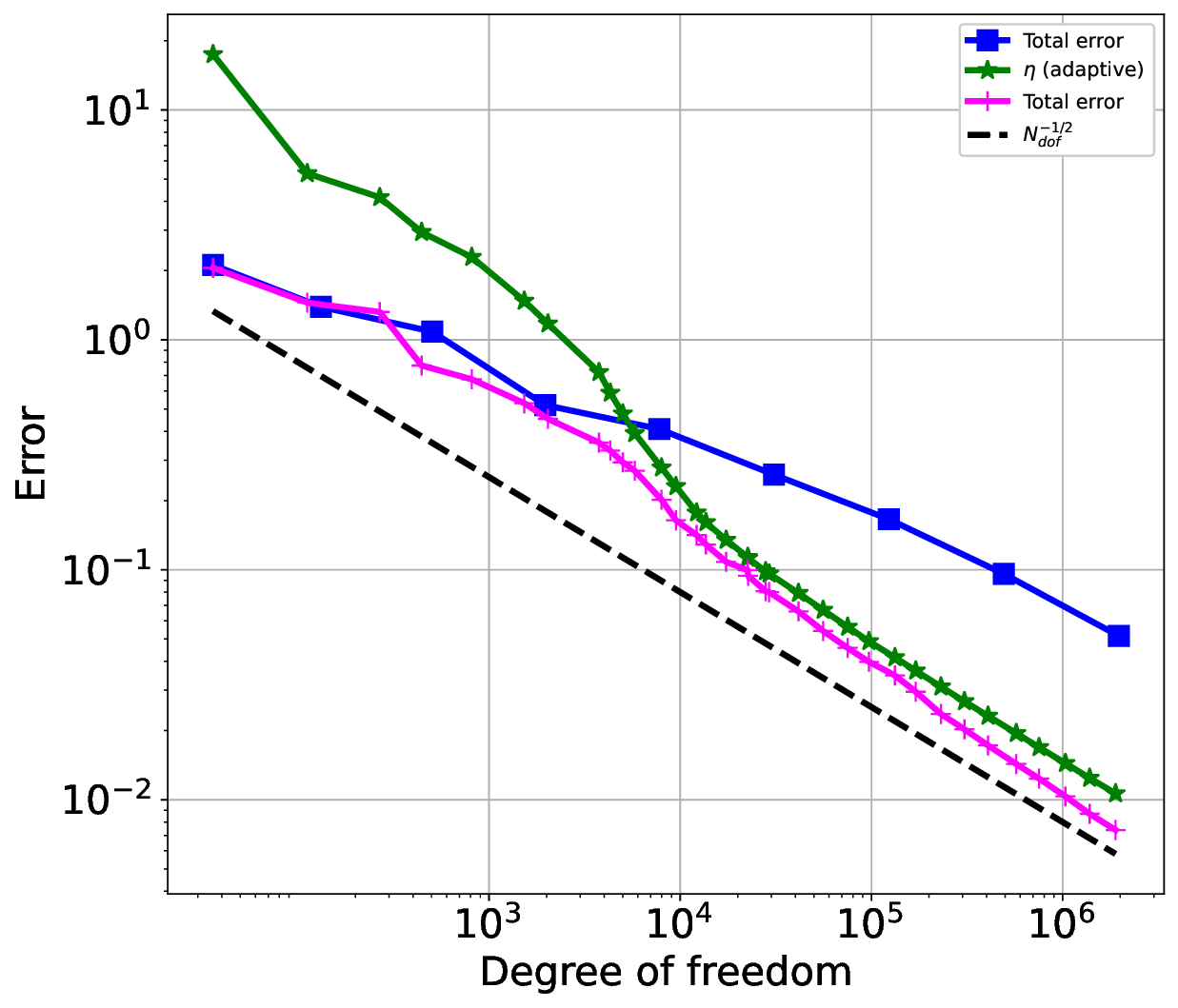}
		\includegraphics[scale=0.35]{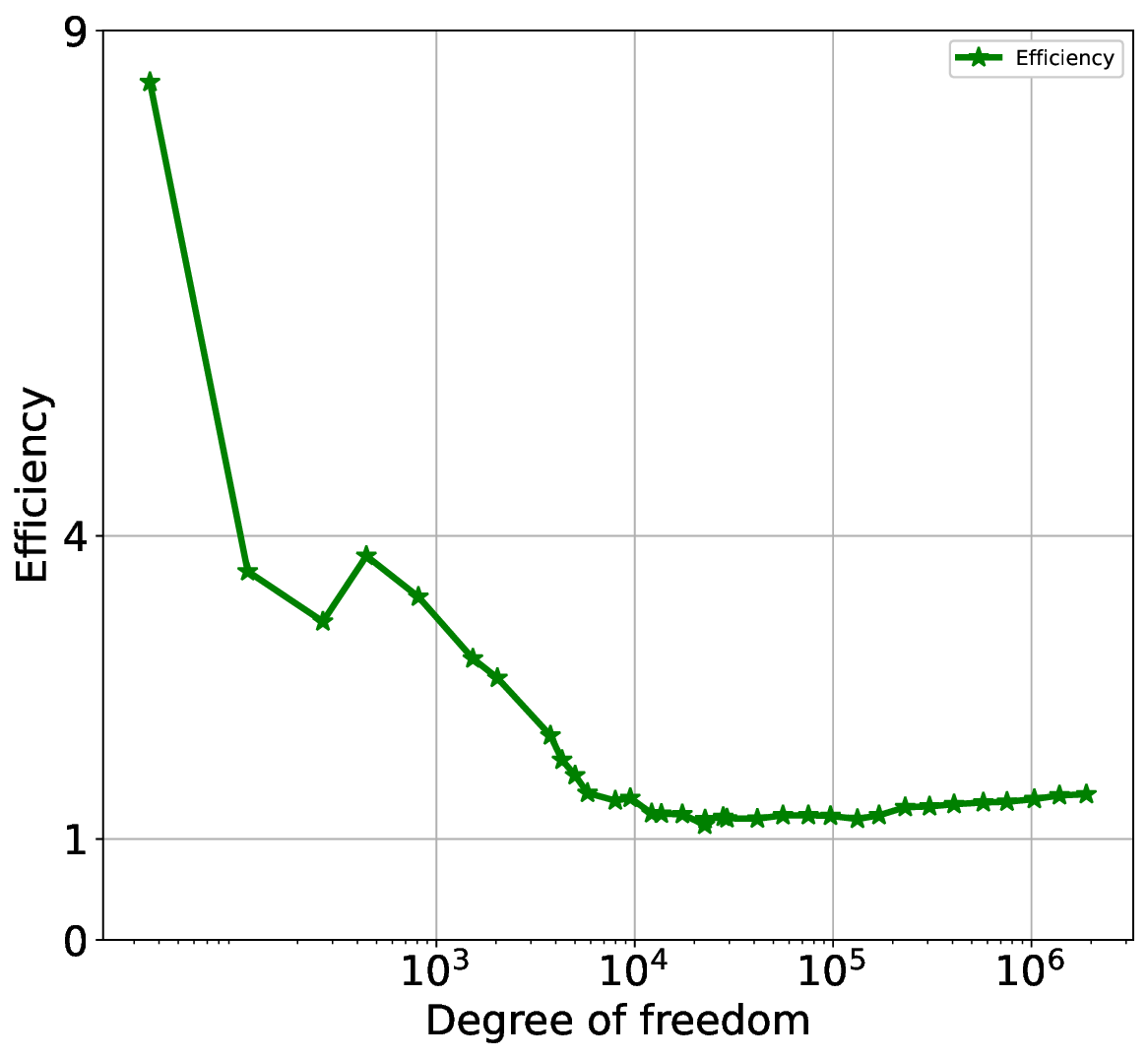}
	\end{center} 
	\caption{Convergence plots of indicator-total error and efficency for Example- \ref{Example 6.2.}.} \label{FIGURE 7}
\end{figure}
\begin{figure}
	\begin{center}
		\includegraphics[scale=0.52]{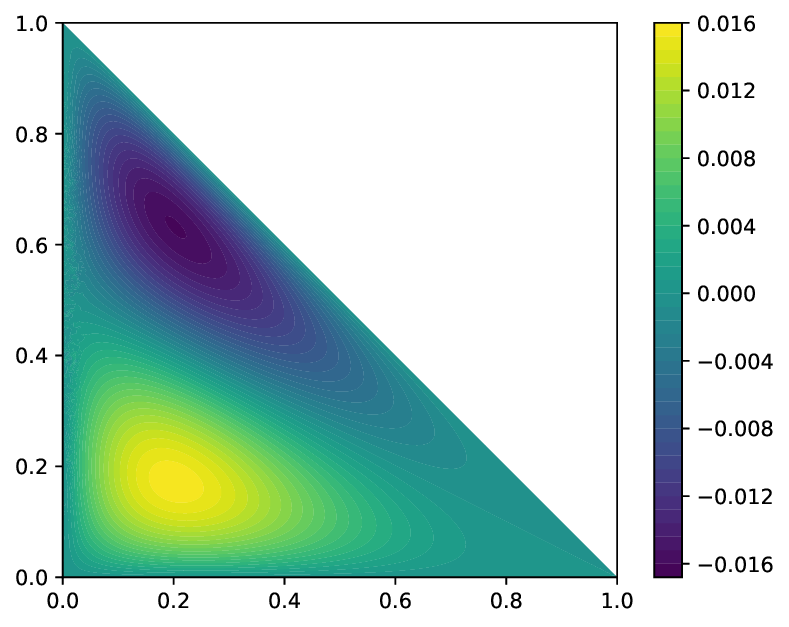}
		\includegraphics[scale=0.52]{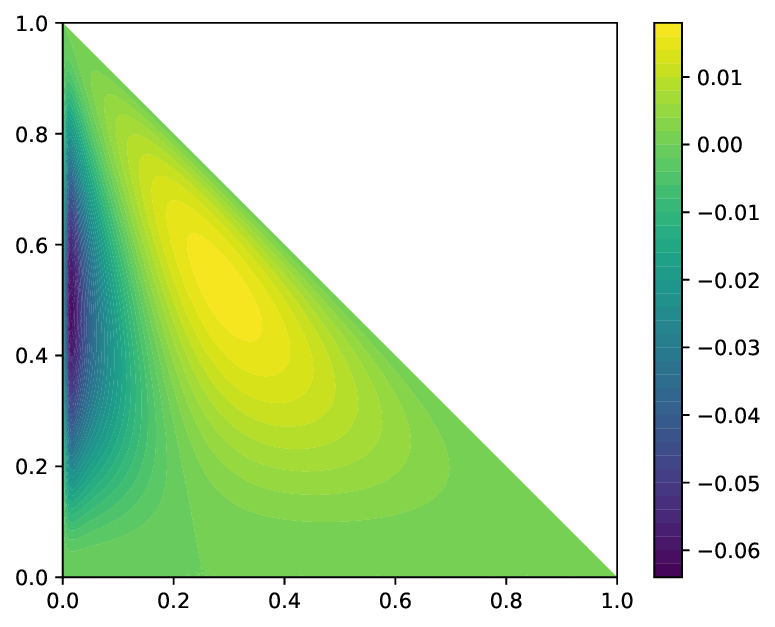}\\ 
		\includegraphics[scale=0.52]{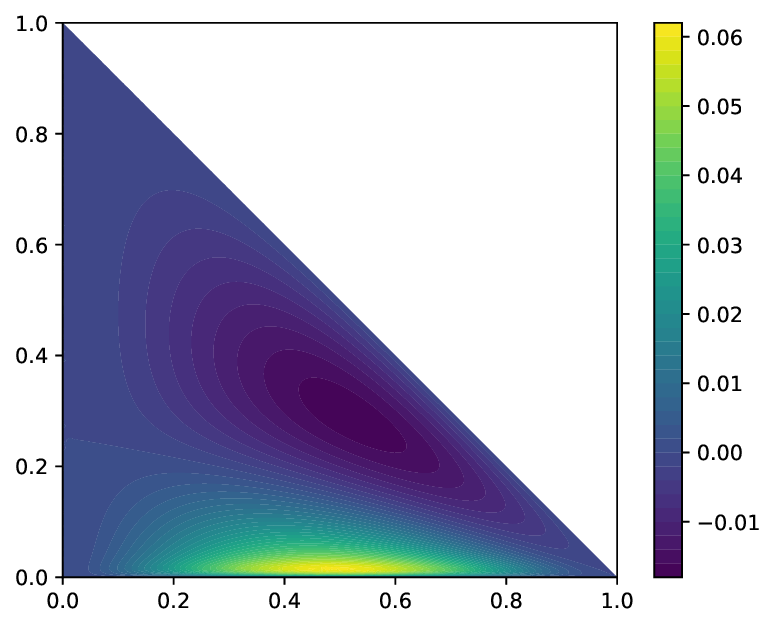}
		\includegraphics[scale=0.52]{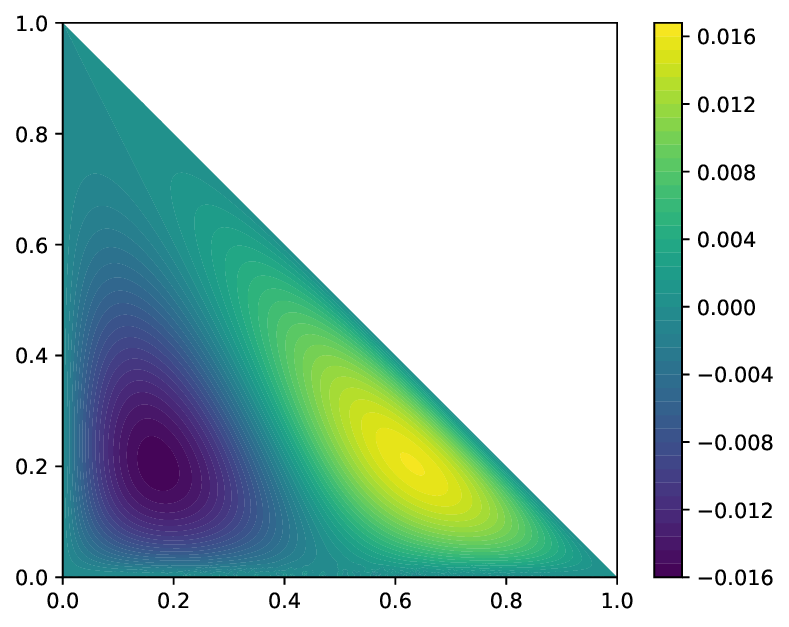}\\
		\includegraphics[scale=0.52]{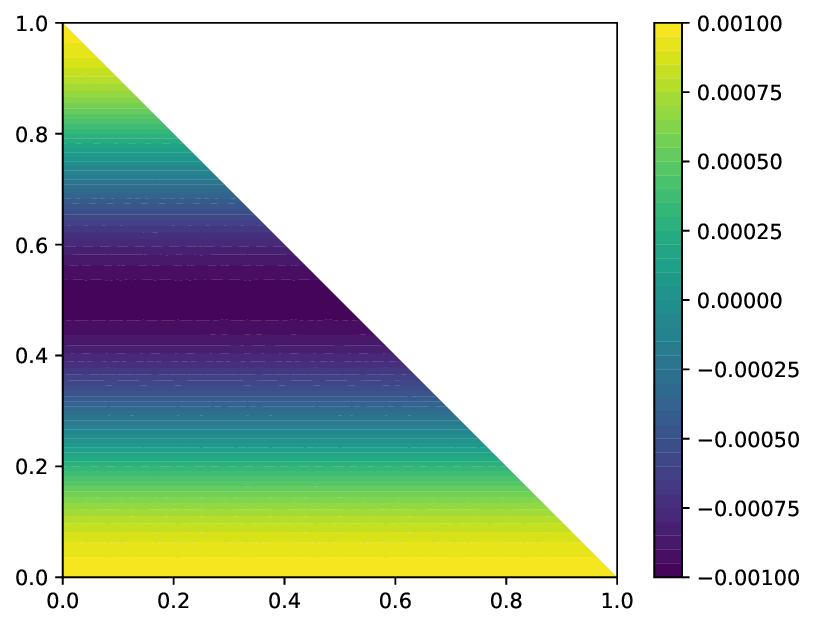}
		\includegraphics[scale=0.52]{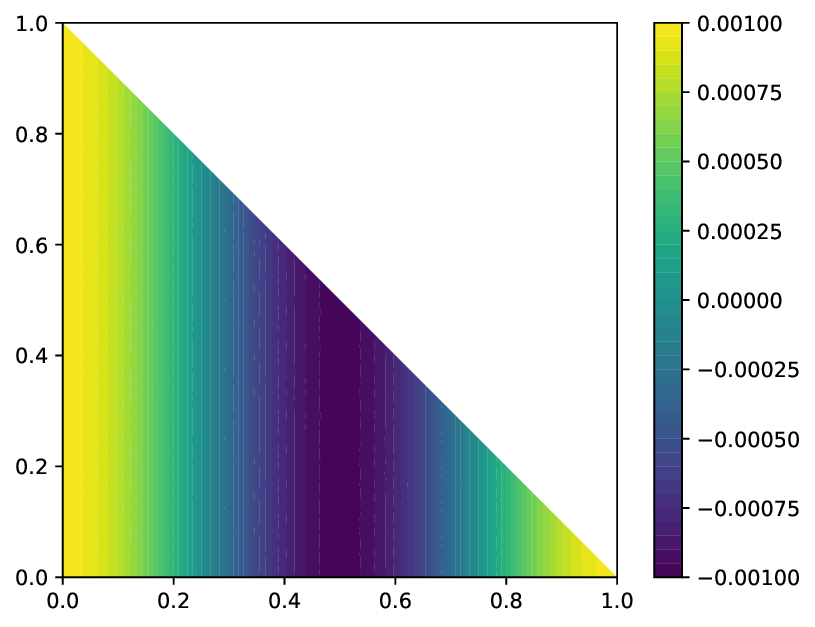} \\
		\includegraphics[scale=0.52]{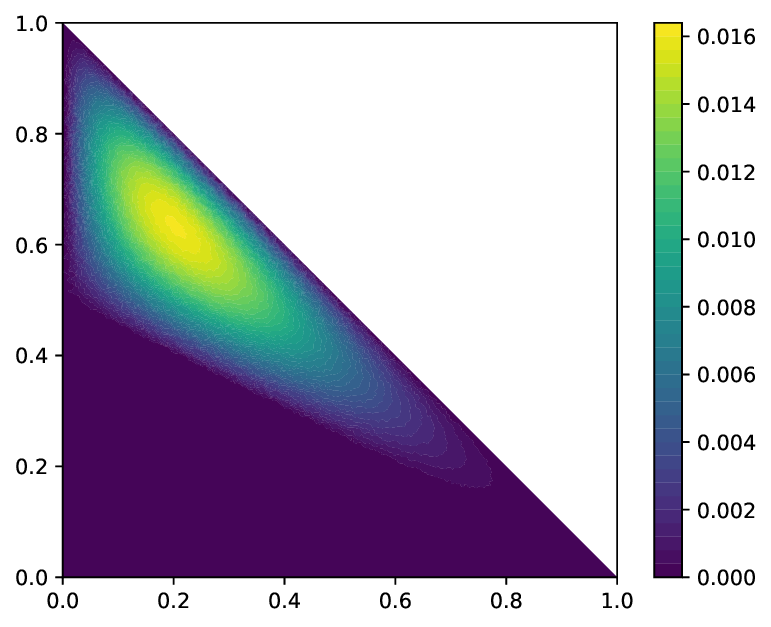}
		\includegraphics[scale=0.52]{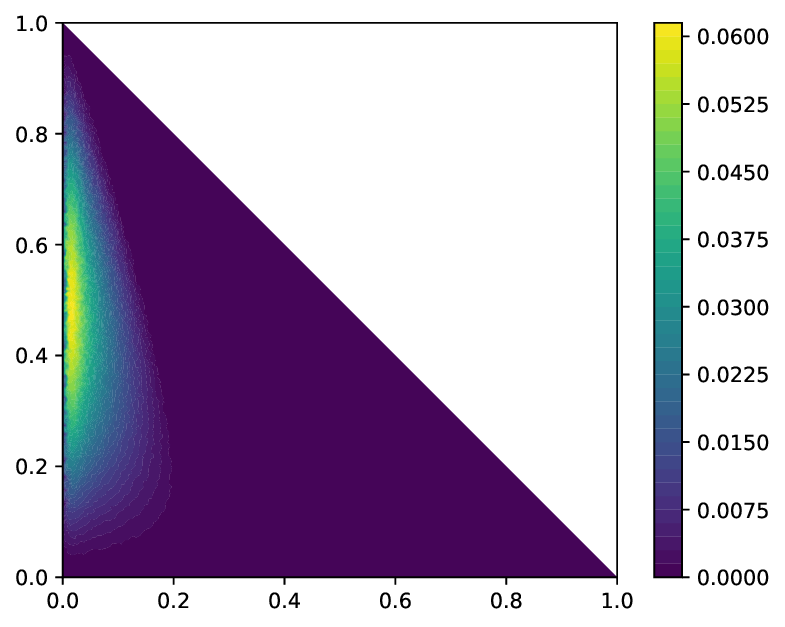}
	\end{center} \vspace{-1.25mm}
	\caption{Plots of numerical solutions of state velocity $(\y_{h1},\y_{h2})$, co-state velocity $(\w_{h1},\w_{h2})$, state pressure $(p_h)$, co-state pressure $(r_h)$, and control $(\u_{h1},\u_{h2})$, respectively, for Example- \ref{Example 6.2.}.} \label{FIGURE 8}
\end{figure}
\subsection{L-shaped and T-shaped domains}\label{Example 6.3.}
	In this example, we consider the non-convex L-shaped domain $\Omega = (-1,1)^{2} \setminus ([0,1) \times (-1,0])$, and the T-shaped domain $\Omega = ((-1.5,1.5) \times (0,1)) \cup ((-0.5,0.5) \times (-2,0])$, with coefficients $\nu = 1,\ \boldsymbol{\beta} = (\xi_1,\xi_2),\ \sigma = 0$, and control bounds $ \mathbf{u_a}=(0,0),$ $\mathbf{u_b}=(0.1,0.1)$. We opt for the source function $\f=(1,1)$ and the desired function $\mathbf{y}_d = (\xi_2,-\xi_2)$. Although the exact solutions for these problems are unknown, through an adaptive refinement process, we observe increased refinements in regions proximate to the reentrant corners  as shown inFigures-\ref{FIGURE 9} and \ref{FIGURE 11}. Following a sufficient number of refinements, the global estimator $\Upsilon$ initiates a decline at the optimal rate, as depicted in Figure-\ref{FIGURE 10}. The numerical solution plots are presented in Figure-\ref{FIGURE LP} and \ref{FIGURE TP}. 
\begin{figure}
	\begin{center}
		\includegraphics[scale=0.45]{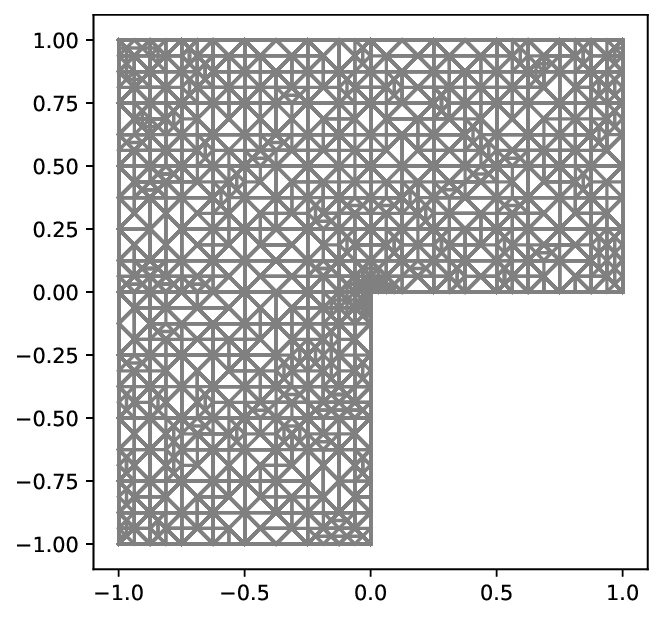}
		\includegraphics[scale=0.45]{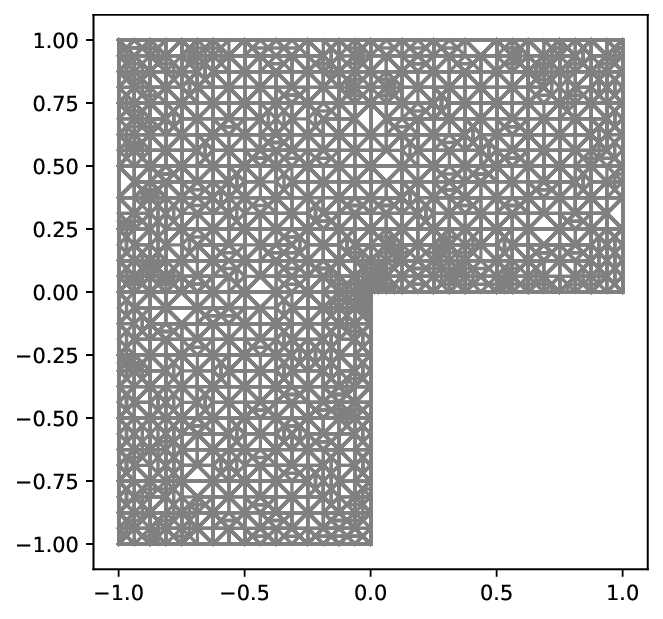}
		\includegraphics[scale=0.45]{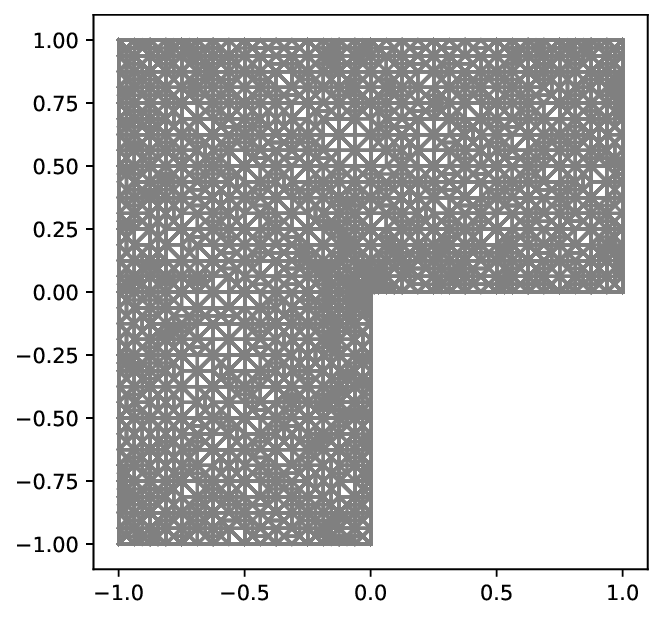}
	\end{center} \vspace{-3mm}
	\caption{Adaptively refined meshes with 5092 DOF, 21972 DOF and 80288 DOF for Example- \ref{Example 6.3.}.} \label{FIGURE 9}
\end{figure}
\begin{figure}
	\begin{center}
		\includegraphics[scale=0.27]{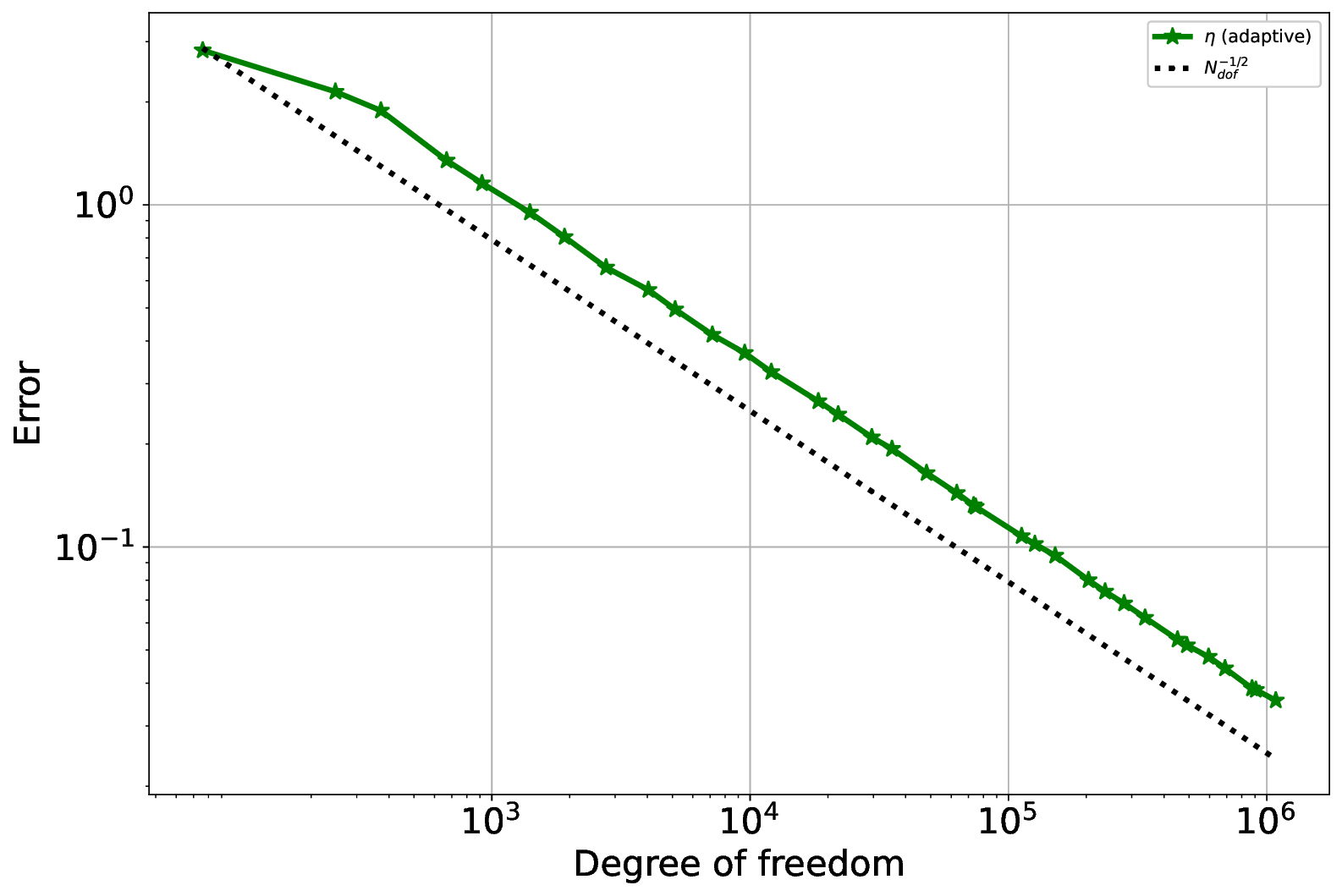}
		\includegraphics[scale=0.27]{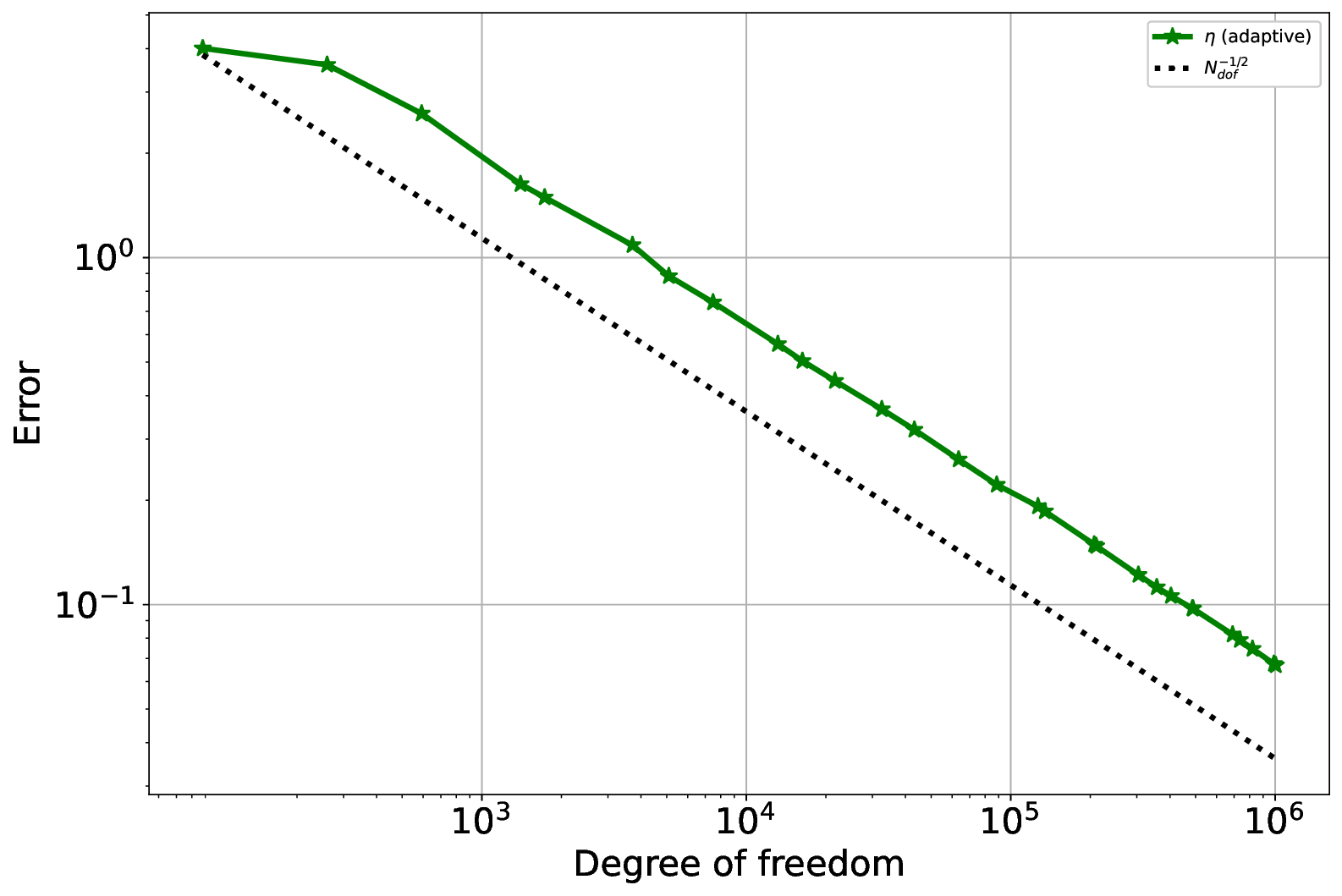} 
	\end{center} \vspace{-3mm}
	\caption{Convergence plots for the indicator under adaptive refinement for L-shape and T-shape Example- \ref{Example 6.3.}.} \label{FIGURE 10}
\end{figure}
\begin{figure}
	\begin{center}
			\includegraphics[scale=0.45]{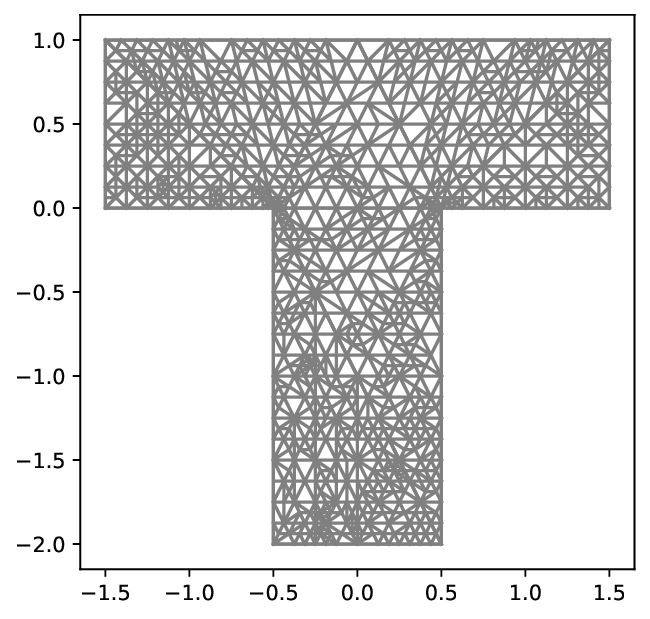} 
			\includegraphics[scale=0.45]{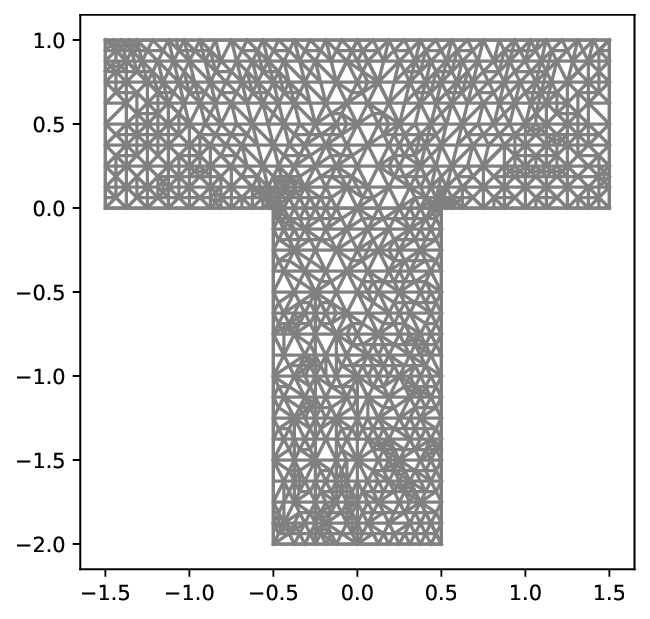}
			\includegraphics[scale=0.45]{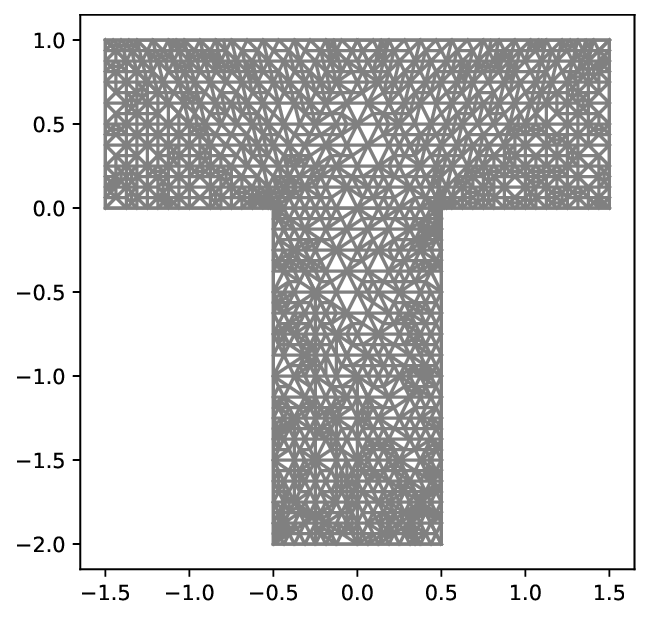} 
		\end{center} \vspace{-3mm}
	\caption{Adaptively refined meshes (a) 16304 DOF (b) 21644 DOF (c) 32552 DOF for Example \ref{Example 6.3.}.} \label{FIGURE 11}
\end{figure}
\begin{figure}
	\begin{center}
		\includegraphics[scale=0.52]{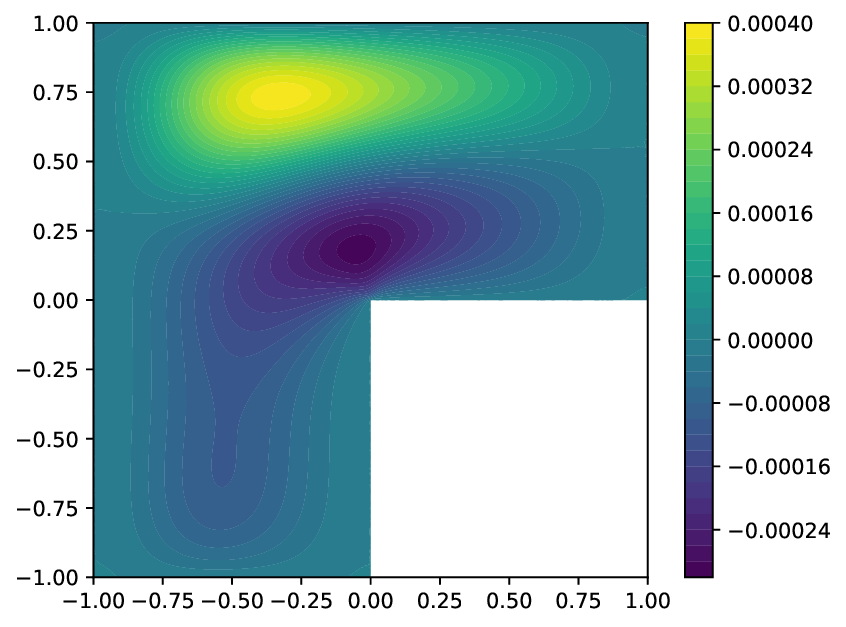}
		\includegraphics[scale=0.52]{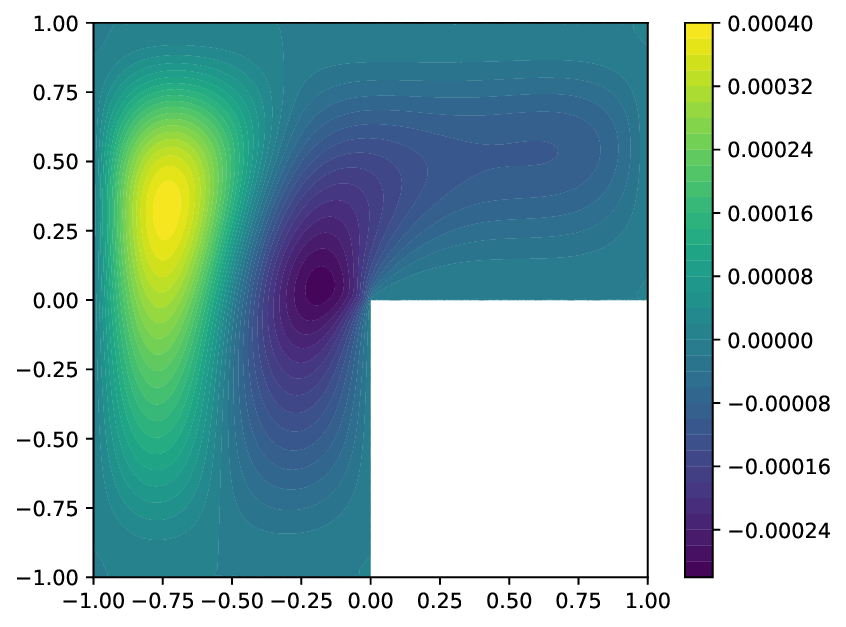} \\ 
		\includegraphics[scale=0.52]{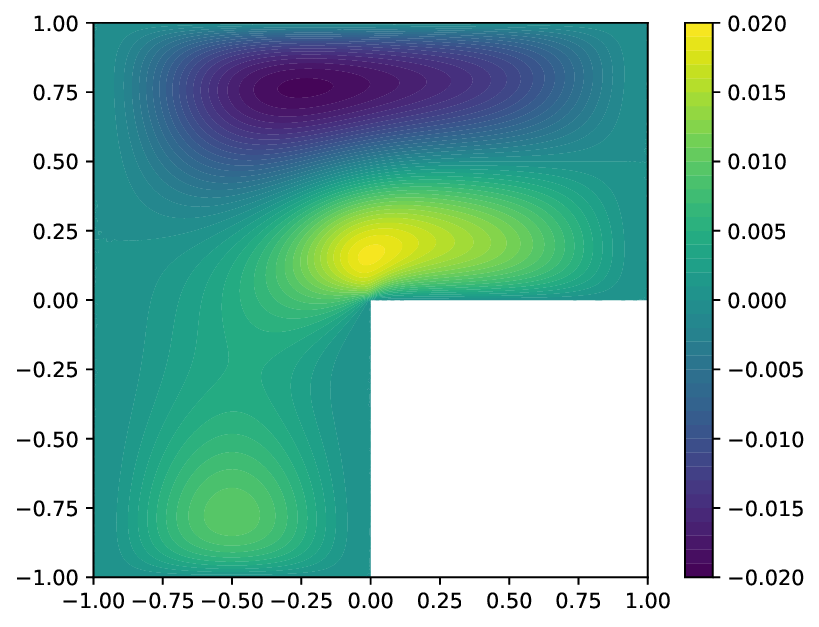}
		\includegraphics[scale=0.52]{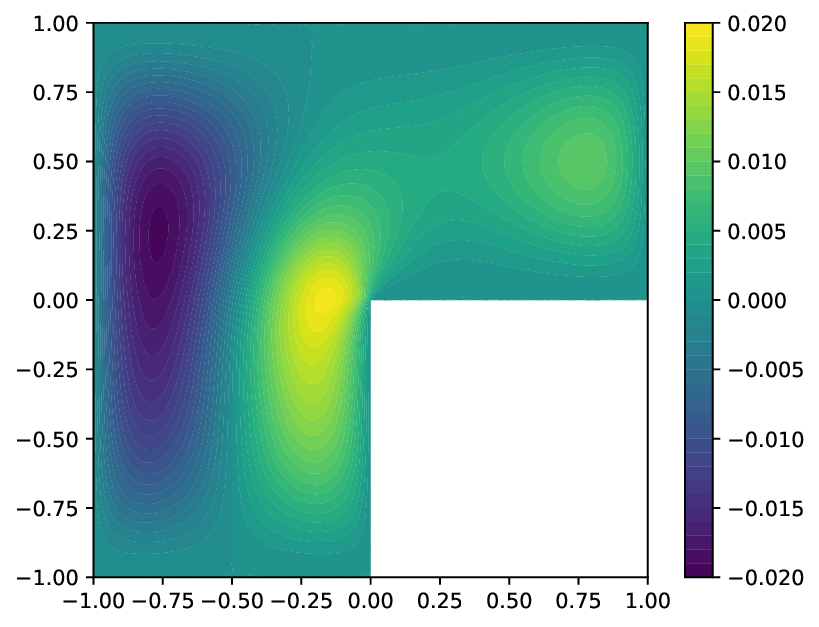}\\
		\includegraphics[scale=0.52]{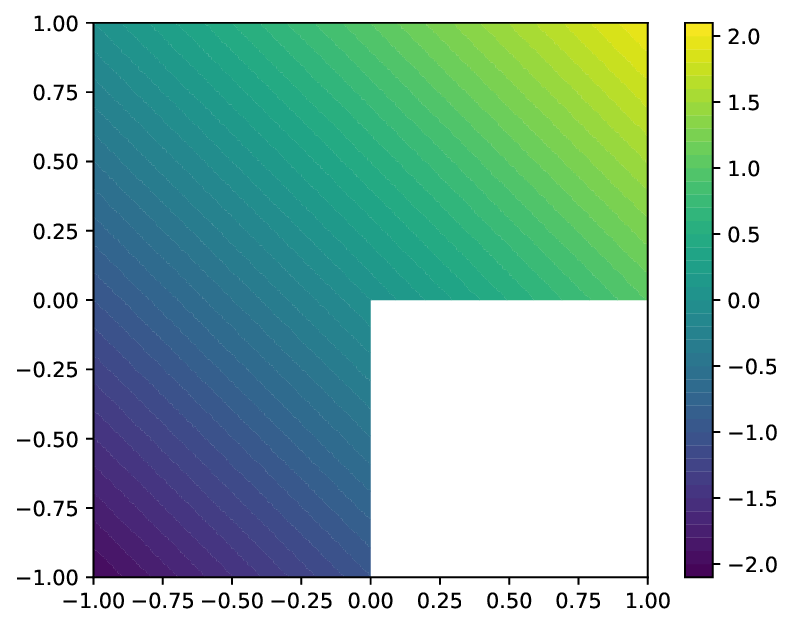}
		\includegraphics[scale=0.52]{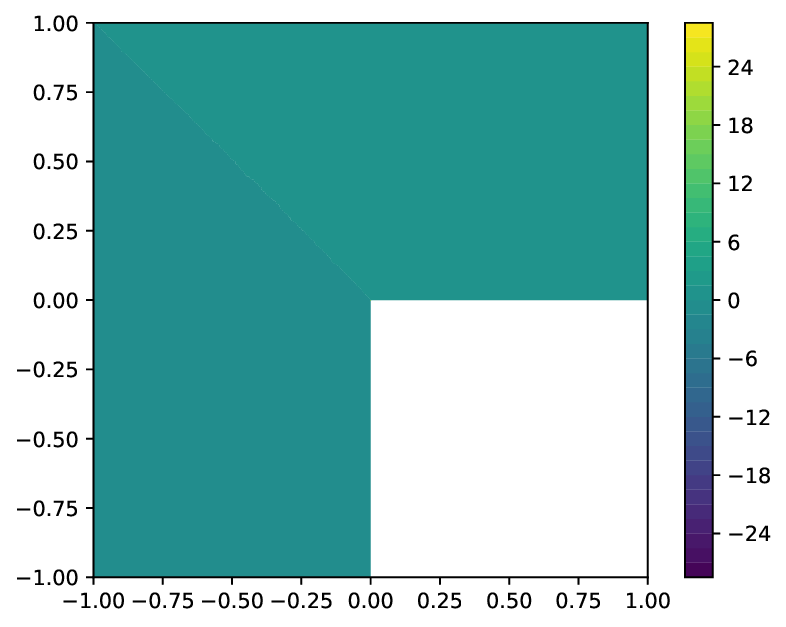} \\ 
		\includegraphics[scale=0.52]{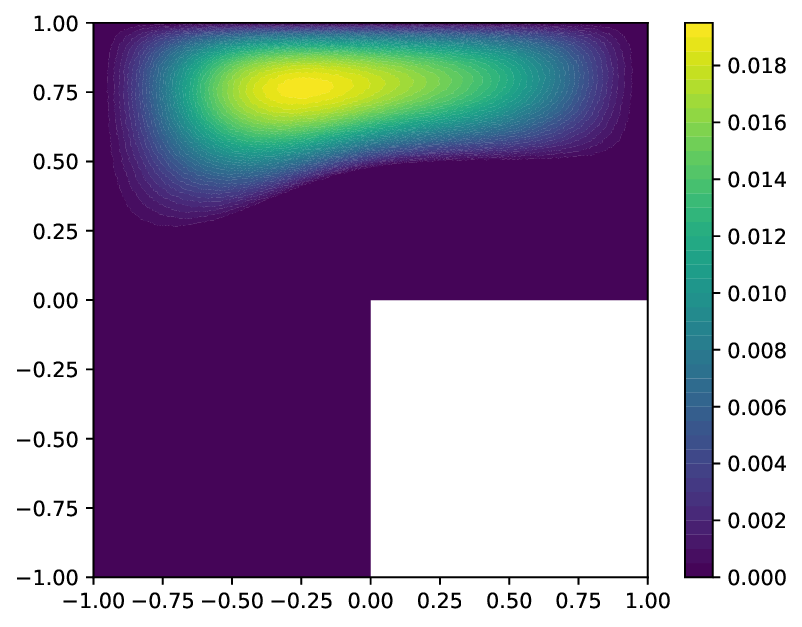}
		\includegraphics[scale=0.52]{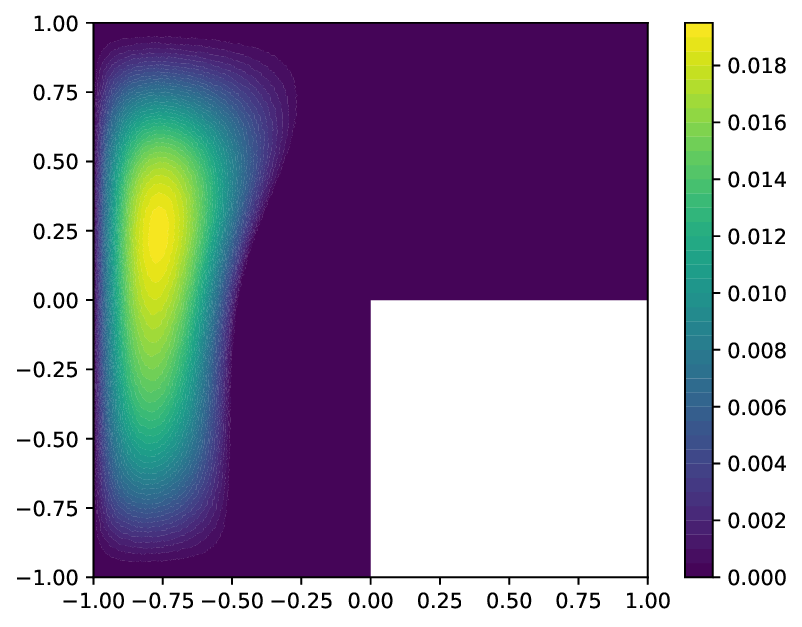}
	\end{center} \vspace{-1.2mm}
	\caption{Plots of numerical solutions of state velocity $(\y_{h1},\y_{h2})$, co-state velocity $(\w_{h1},\w_{h2})$, state pressure $(p_h)$, co-state pressure $(r_h)$, and control $(\u_{h1},\u_{h2})$, respectively, for Example- \ref{Example 6.3.}.} \label{FIGURE LP}
\end{figure}
\begin{figure}
	\begin{center}
		\includegraphics[scale=0.52]{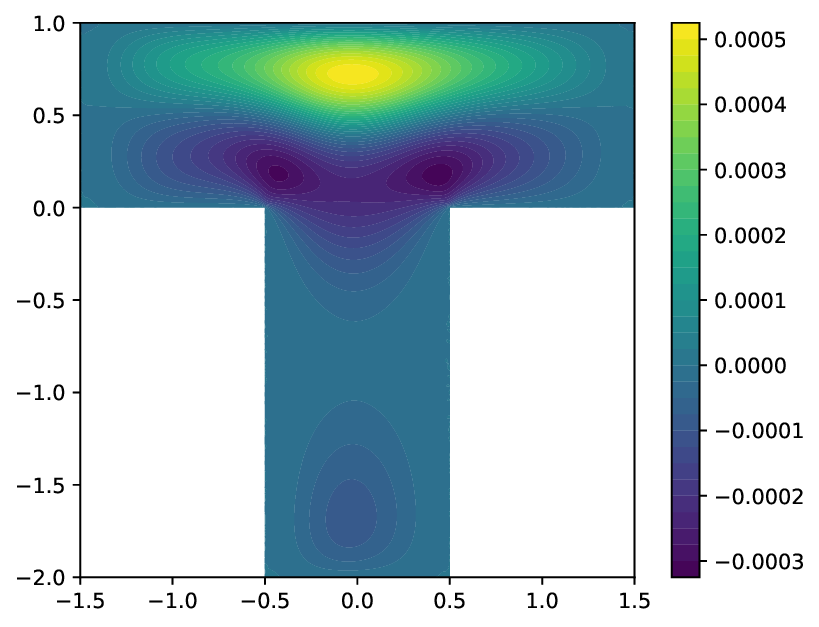}
		\includegraphics[scale=0.52]{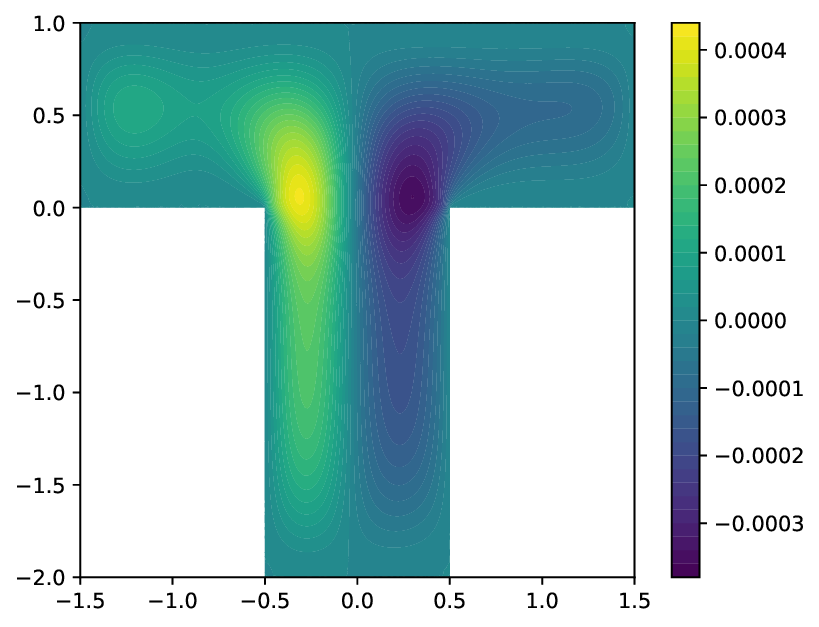} \\ 
		\includegraphics[scale=0.52]{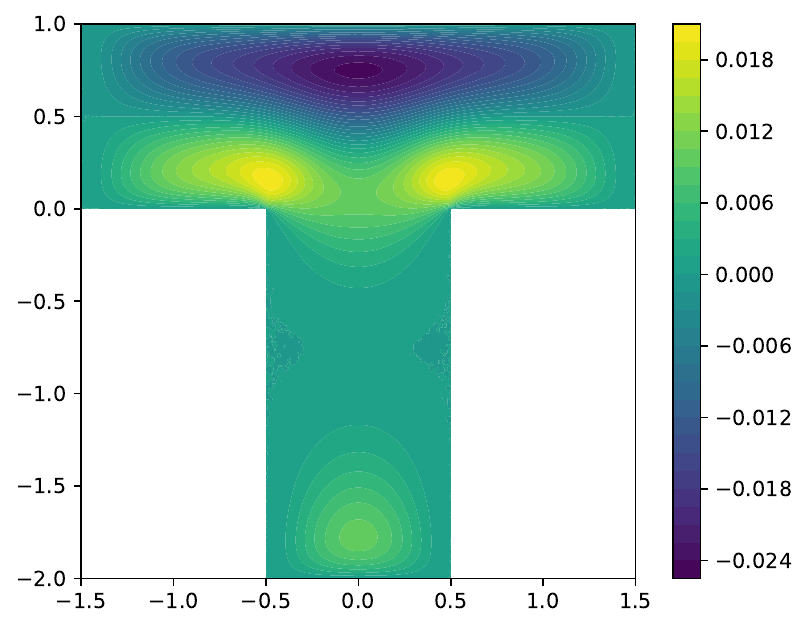}
		\includegraphics[scale=0.52]{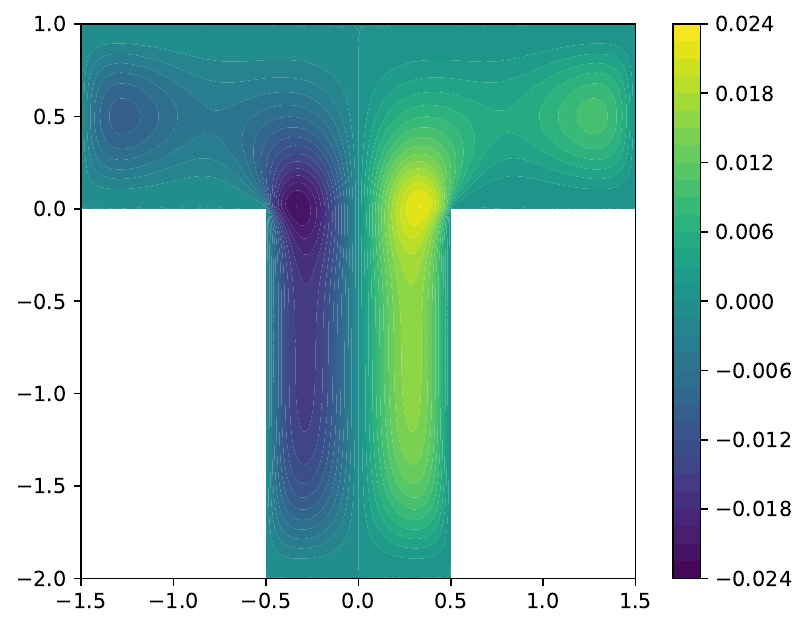}\\
		\includegraphics[scale=0.52]{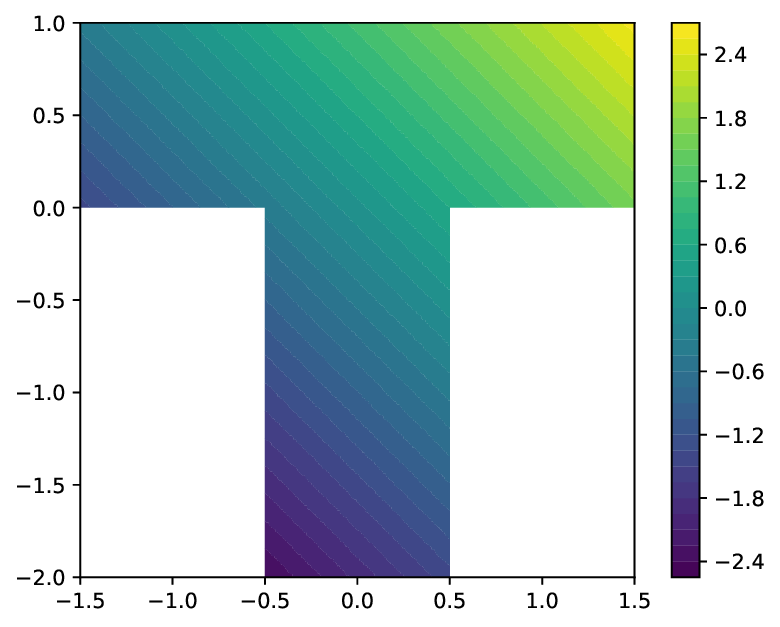}
		\includegraphics[scale=0.52]{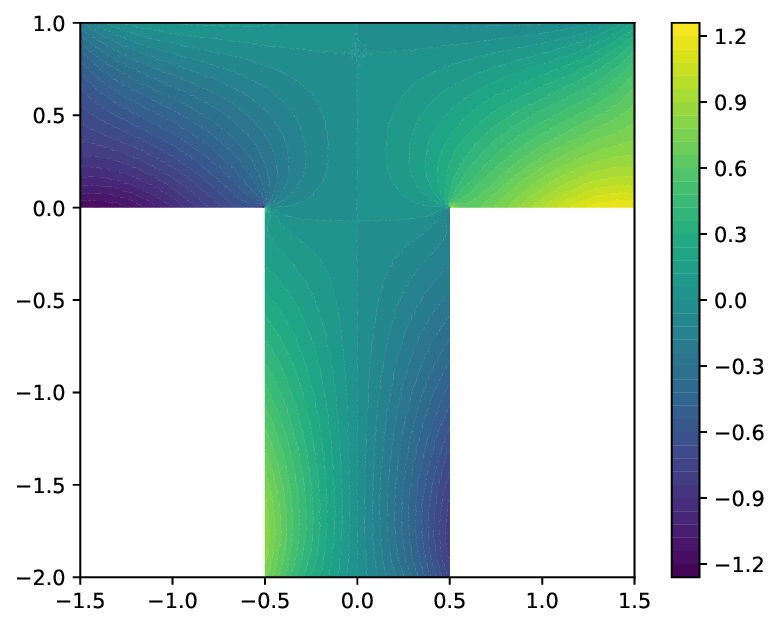} \\ 
		\includegraphics[scale=0.52]{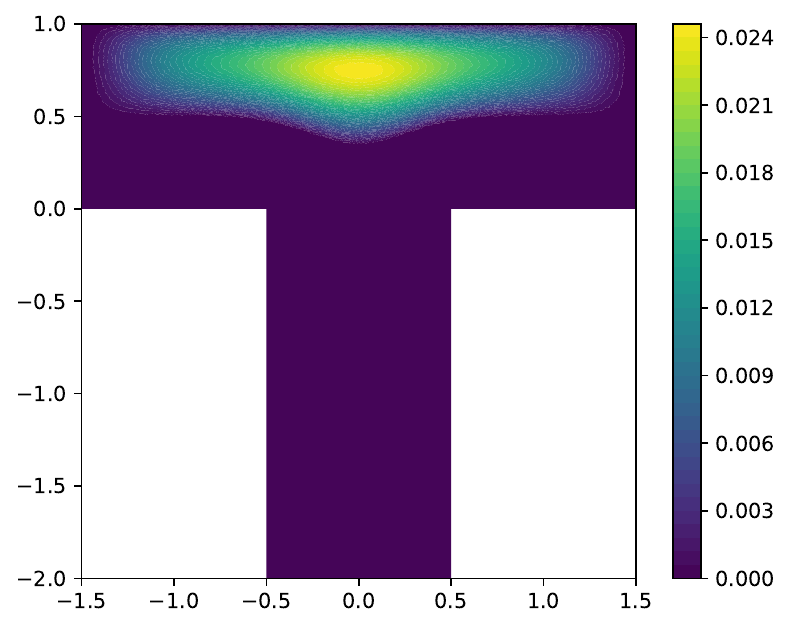}
		\includegraphics[scale=0.52]{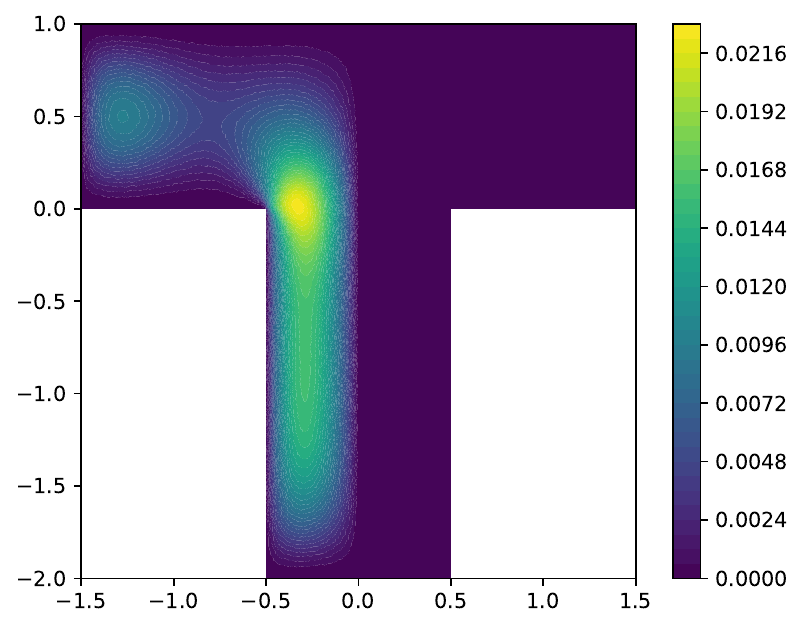}
	\end{center} \vspace{-1.5mm}
	\caption{Plots of numerical solutions of state velocity $(\y_{h1},\y_{h2})$, co-state velocity $(\w_{h1},\w_{h2})$, state pressure $(p_h)$, co-state pressure $(r_h)$, and control $(\u_{h1},\u_{h2})$, respectively, for Example- \ref{Example 6.3.}.} \label{FIGURE TP}
\end{figure}
\subsection{Three-dimensional example}\label{Example 6.5.}
	Consider the cubical domain $\Omega = (0,1)^{3}$ with coefficients $\nu =1+0.01 \xi_3^{3},\ \boldsymbol{\beta} (\xi_1,\xi_2,\xi_3) = (\xi_2-\xi_3,\xi_3-\xi_1,\xi_1-\xi_2),\ \sigma = 1,\ \mathbf{u_a}=(0,0,0)$ and $\mathbf{u_b}=(0.1,0.1,0.1)$. The source function $\mathbf{f}$ and desired function $\mathbf{y}_d$ are choosen such that 
	\begin{align*}
		\y(\xi_1,\xi_2,\xi_3= \w(\xi_1,\xi_2,\xi_3) &= \textbf{curl} \ \big(0.5 \pi (\sin(\pi \xi_1) \sin(\pi \xi_2)\sin(\pi \xi_3))^{2}\big)\\
		 p(\xi_1,\xi_2,\xi_3) = r(\xi_1,\xi_2,\xi_3) &= \sin(2\pi \xi_1)\sin(2\pi \xi_2)\sin(2\pi \xi_3),
	\end{align*}
	are the exact solutions. We attain optimal convergence rates for state, co-state, and control variables through uniform refinement, as presented in Figure-\ref{FIGURE 15}. Both the global estimator $\Upsilon$ and the total error $|\!|\!|\mathbf{e}|\!|\!|_{\Omega}$ exhibit a decrease at the optimal rate, as depicted in Figure-\ref{FIGURE 16}. The graph of the estimator $\Upsilon$ aligns parallel to the error $|\!|\!|\mathbf{e}|\!|\!|_{\Omega}$, and the efficiency index demonstrates consistent behavior. This observation validates the numerical reliability and efficiency of the proposed error estimator. Plots showcasing the numerical solutions of the state and co-state variables are provided in Figure-\ref{FIGURE 14}.
\begin{figure}
	\begin{center}
		\includegraphics[scale=0.09]{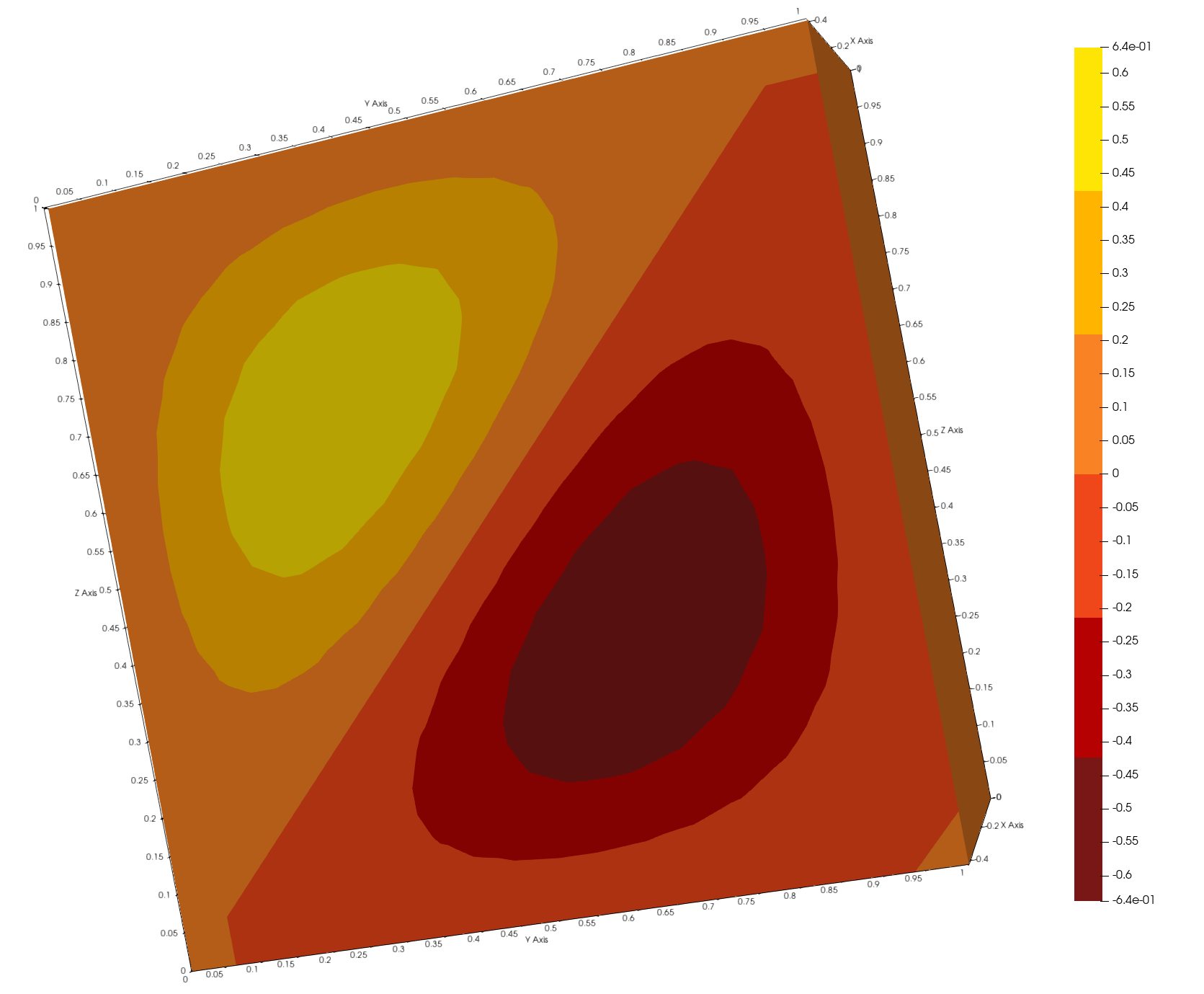}
		\includegraphics[scale=0.095]{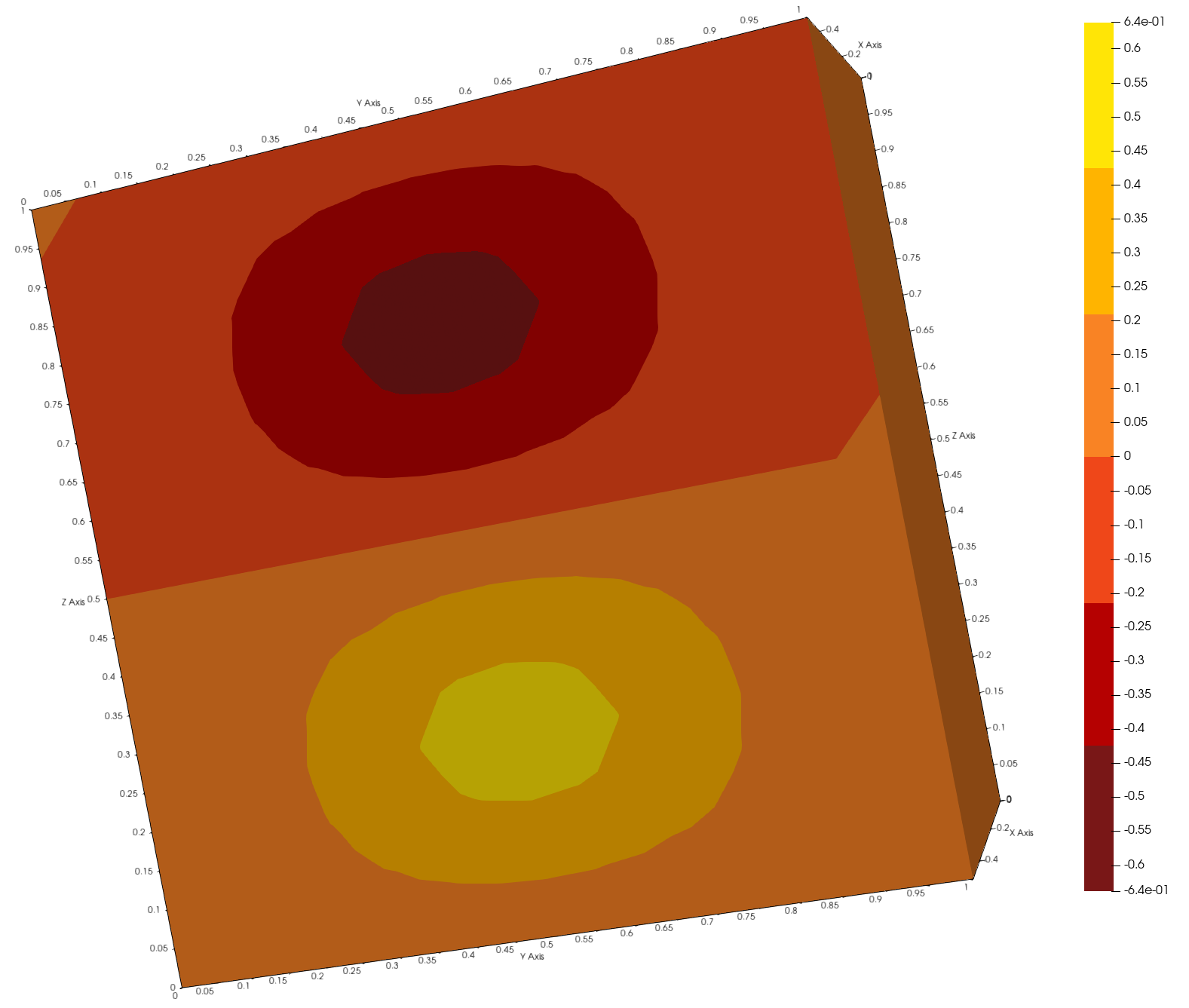}
		\includegraphics[scale=0.09]{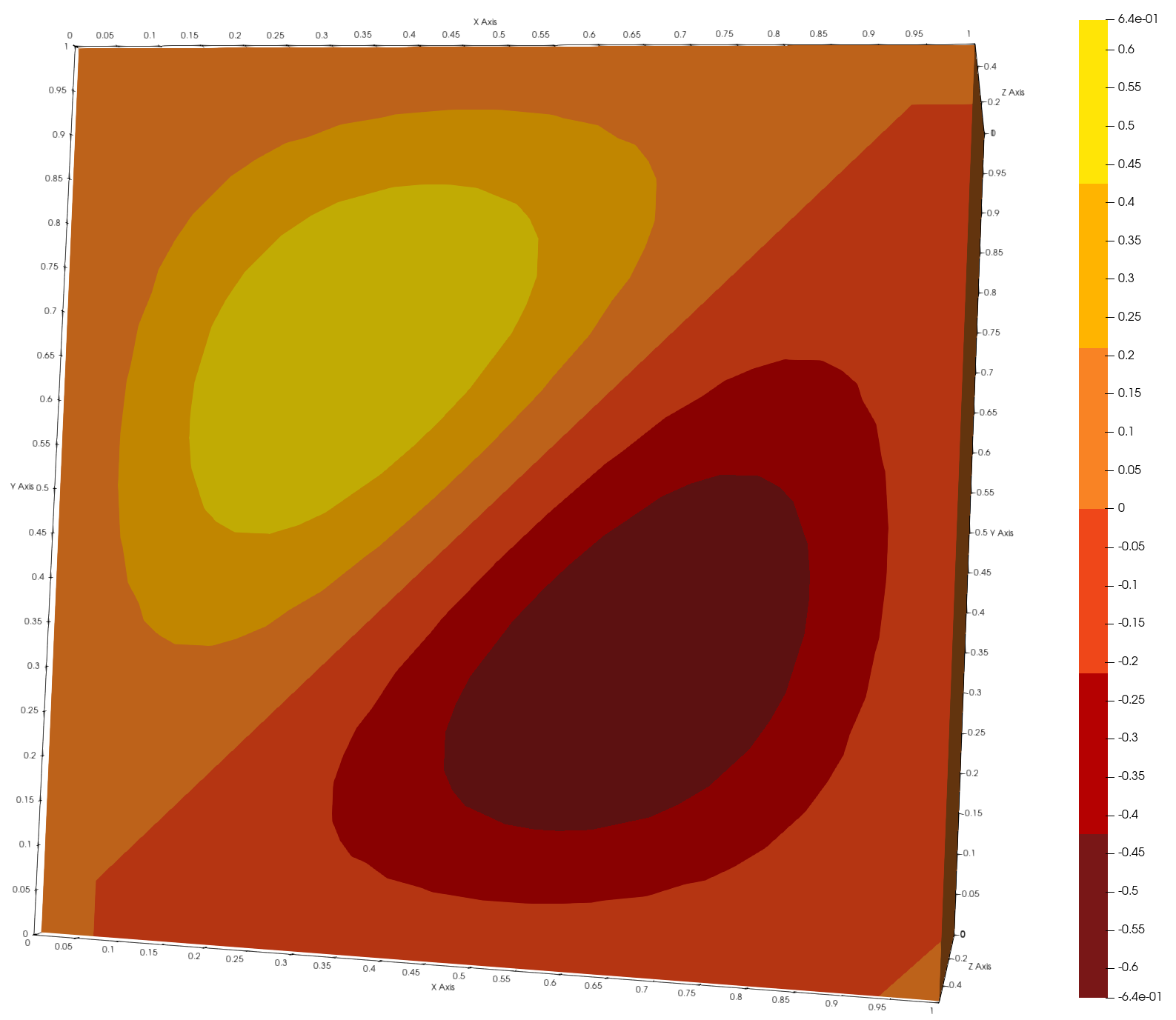}\\ 
		\includegraphics[scale=0.09]{yh_41.png}
		\includegraphics[scale=0.09]{yh_42.png}
		\includegraphics[scale=0.087]{yh_43.png}\\  
		\includegraphics[scale=0.1]{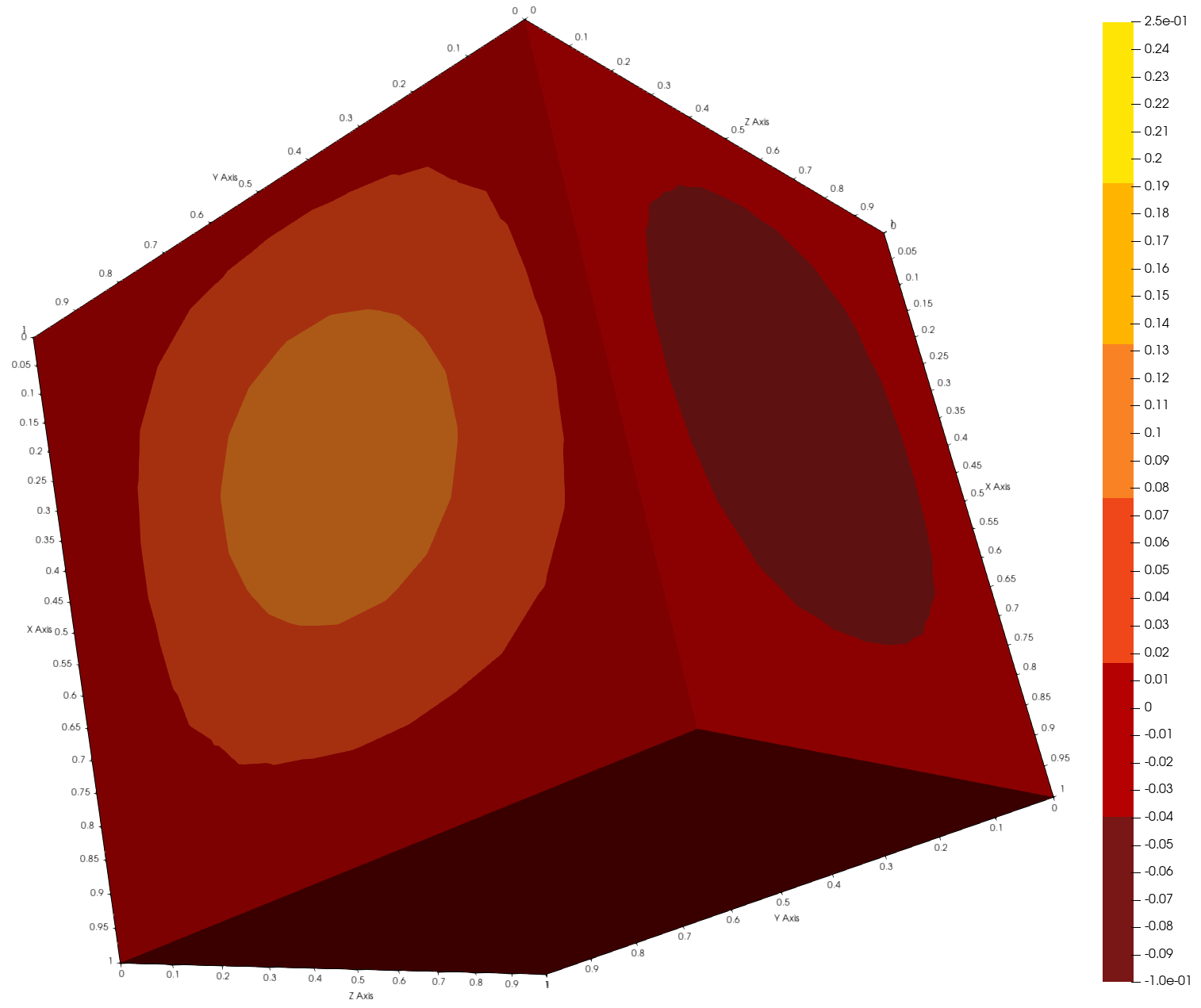}
		\includegraphics[scale=0.09]{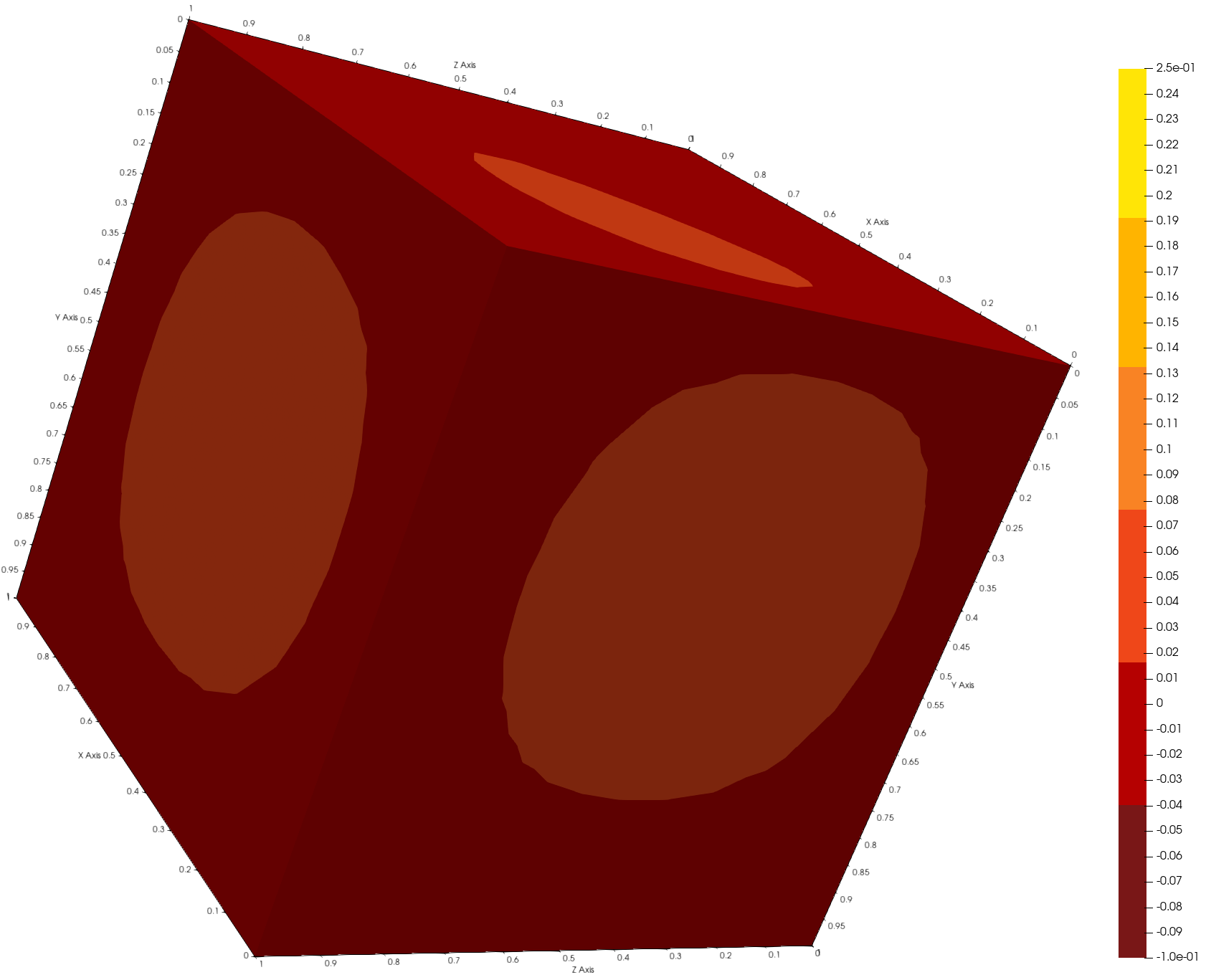}
		\includegraphics[scale=0.09]{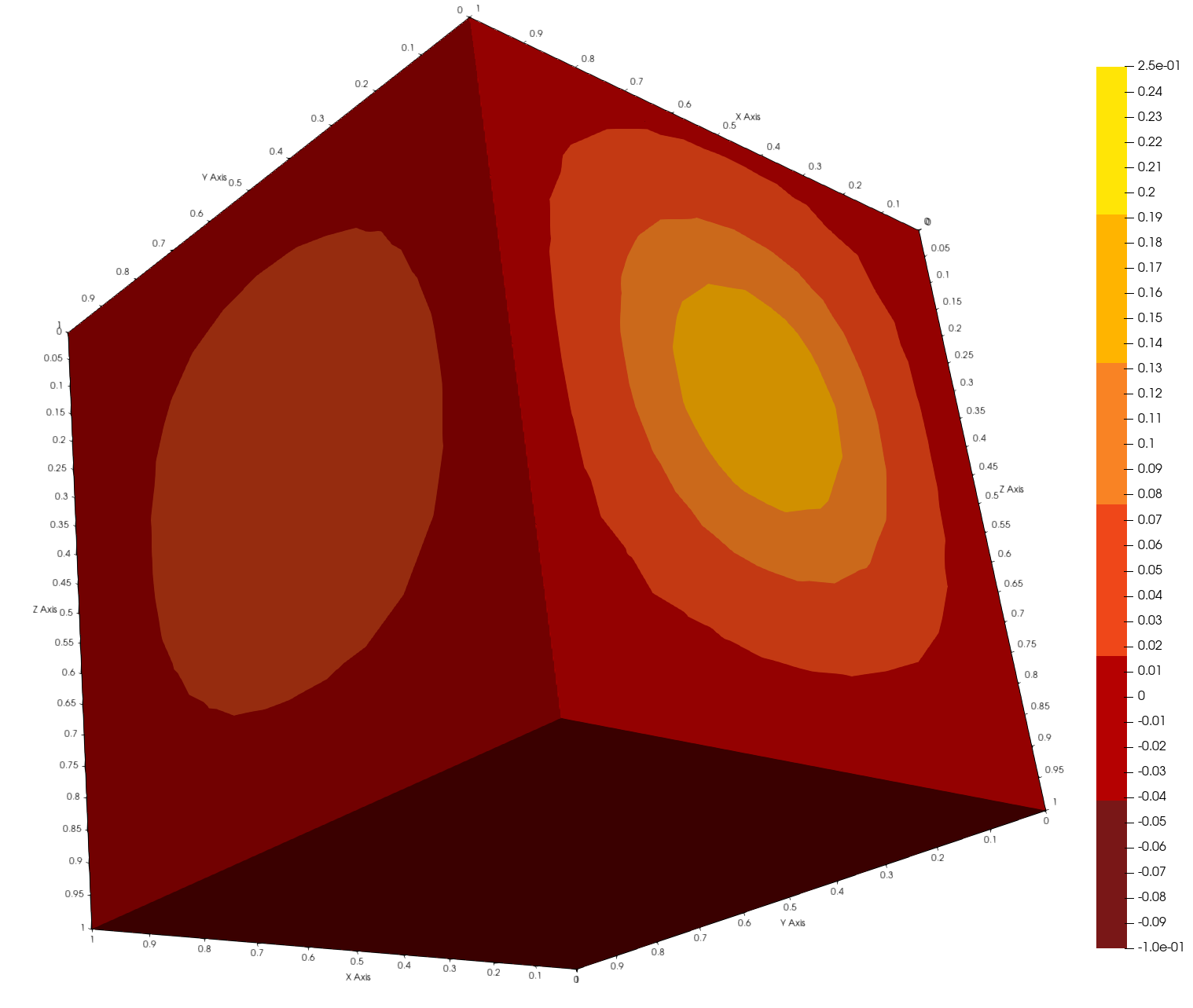} \\
		\includegraphics[scale=0.12]{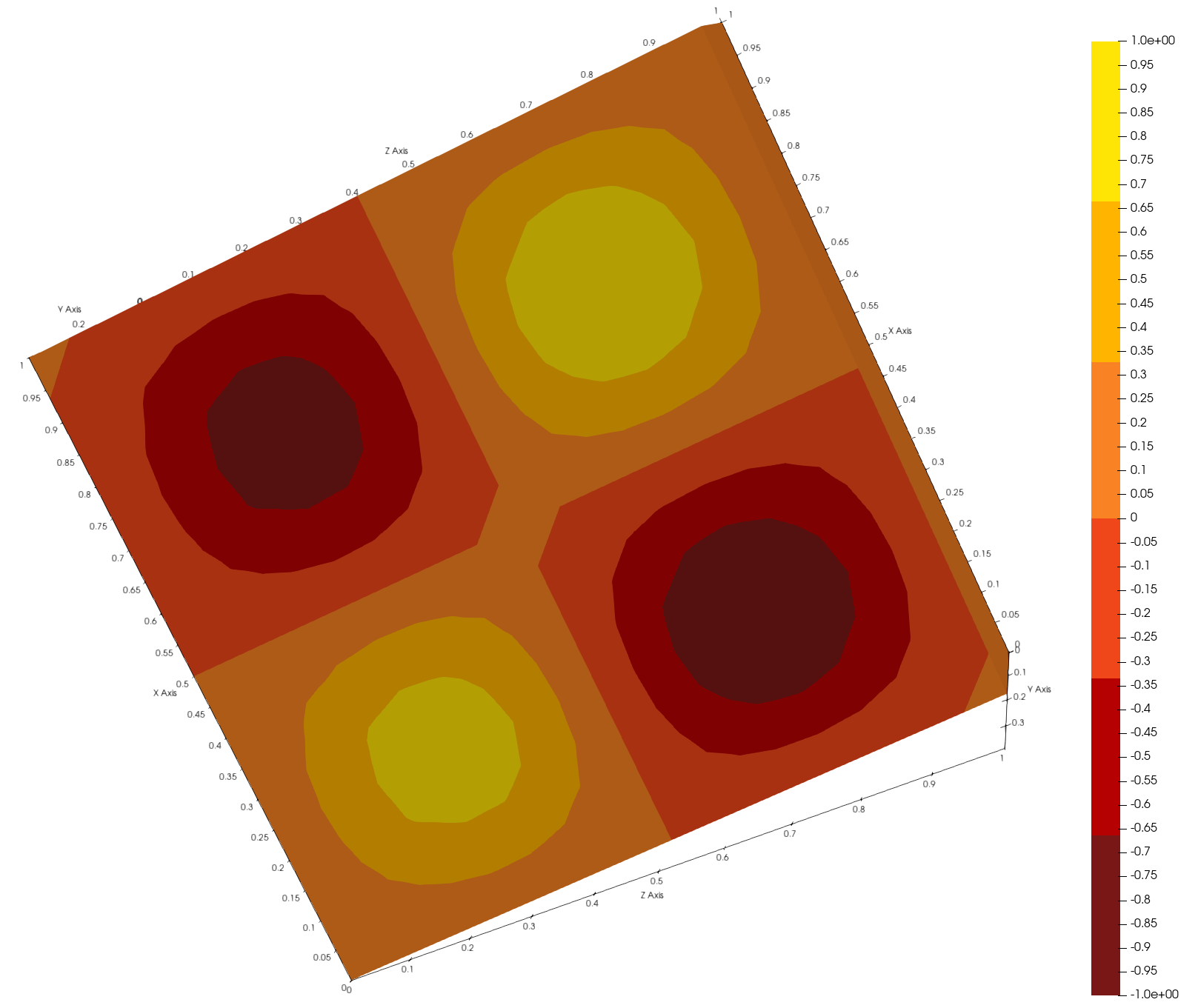}
		\includegraphics[scale=0.12]{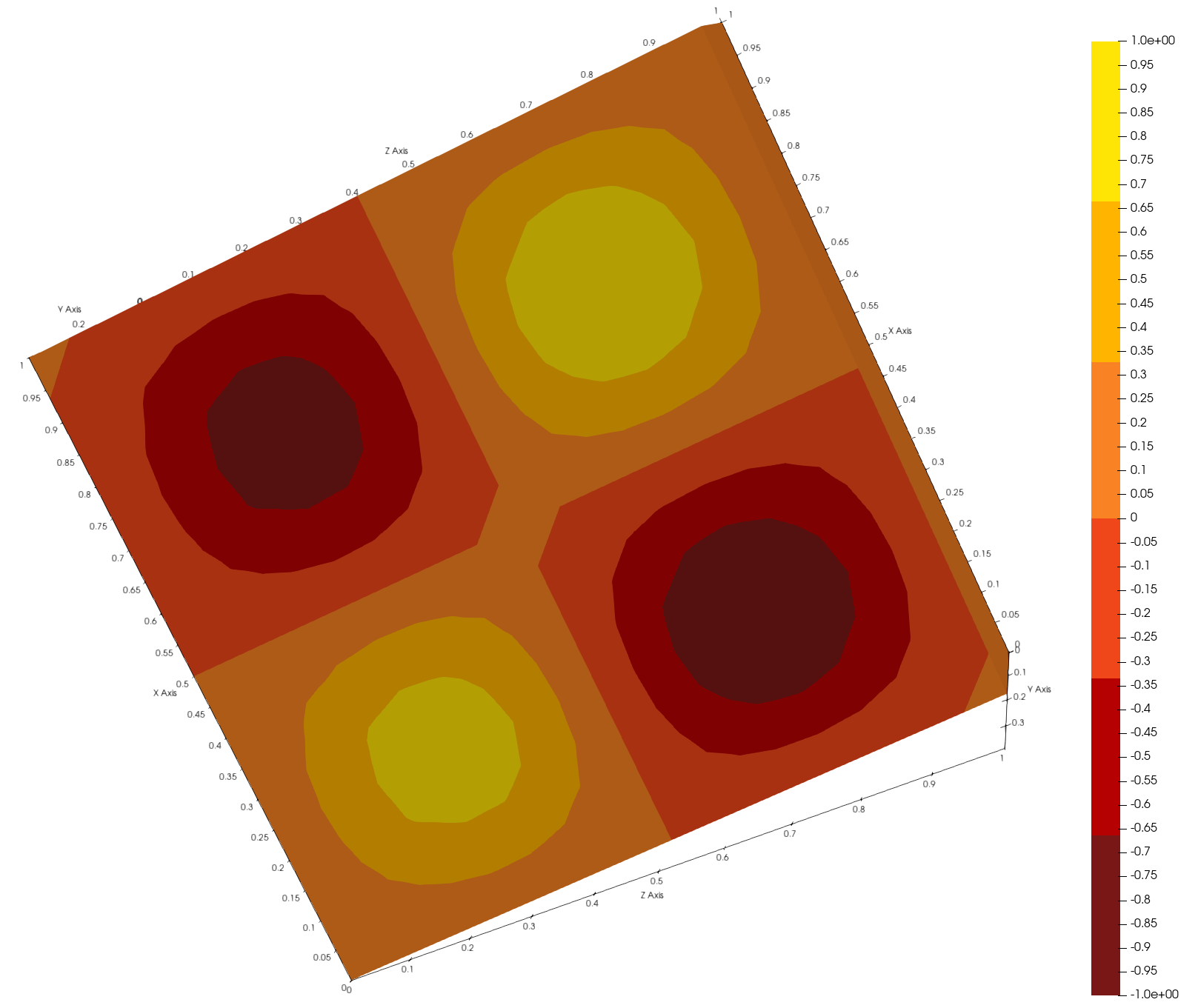}
	\end{center}
	\caption{ 3-D plots of numerical solutions of state velocity $(\y_{h1},\y_{h2},\y_{h3})$, co-state velocity $(\w_{h1},\w_{h2},\w_{h3})$, control $(\u_{h1},\u_{h2},\u_{h3})$, state pressure $(p_h)$ and co-state pressure $(r_h)$, respectively, for Example \ref{Example 6.5.}.} \label{FIGURE 14}
\end{figure}
\begin{figure}
	\begin{center}
		\includegraphics[scale=0.35]{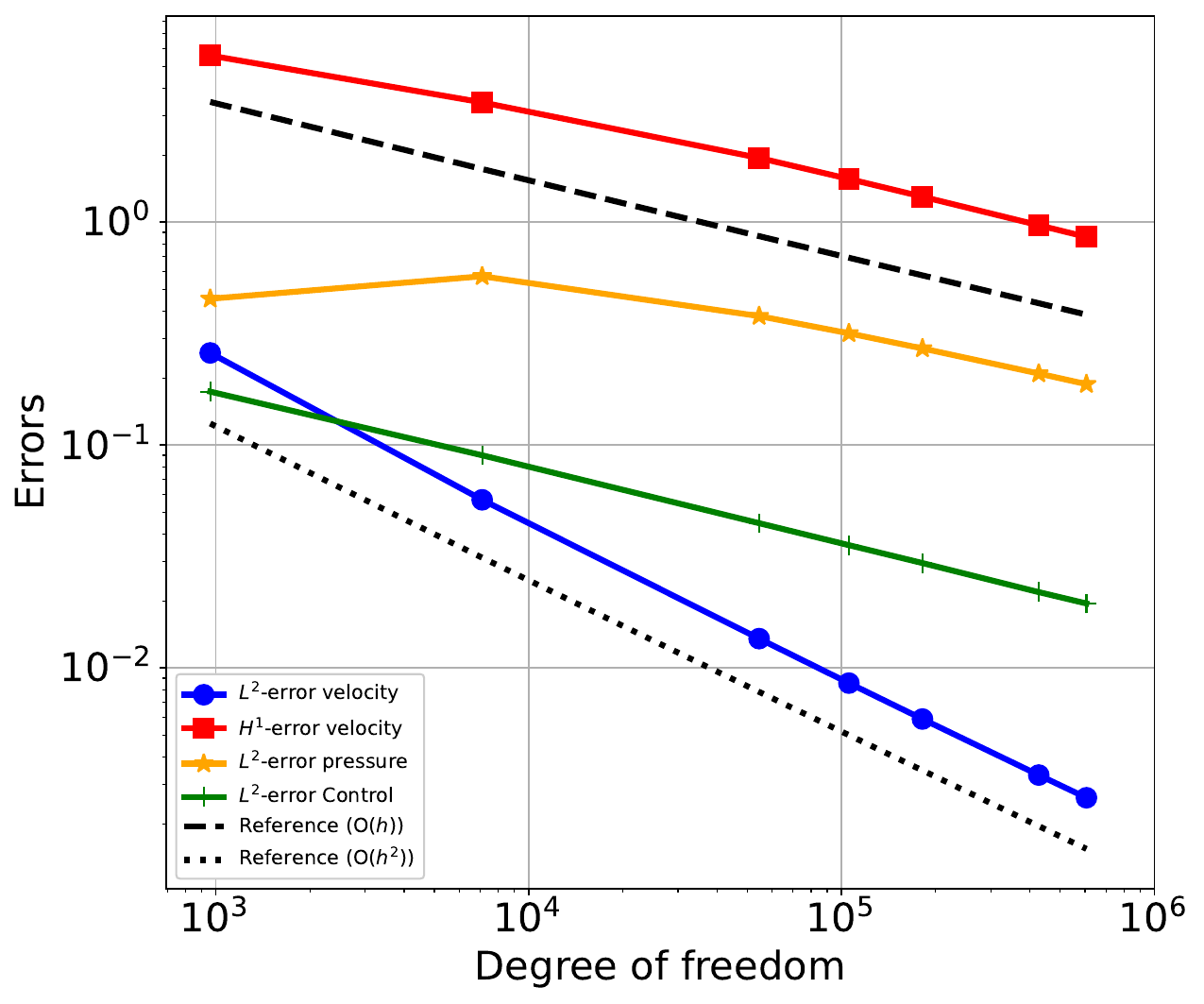}
		\includegraphics[scale=0.35]{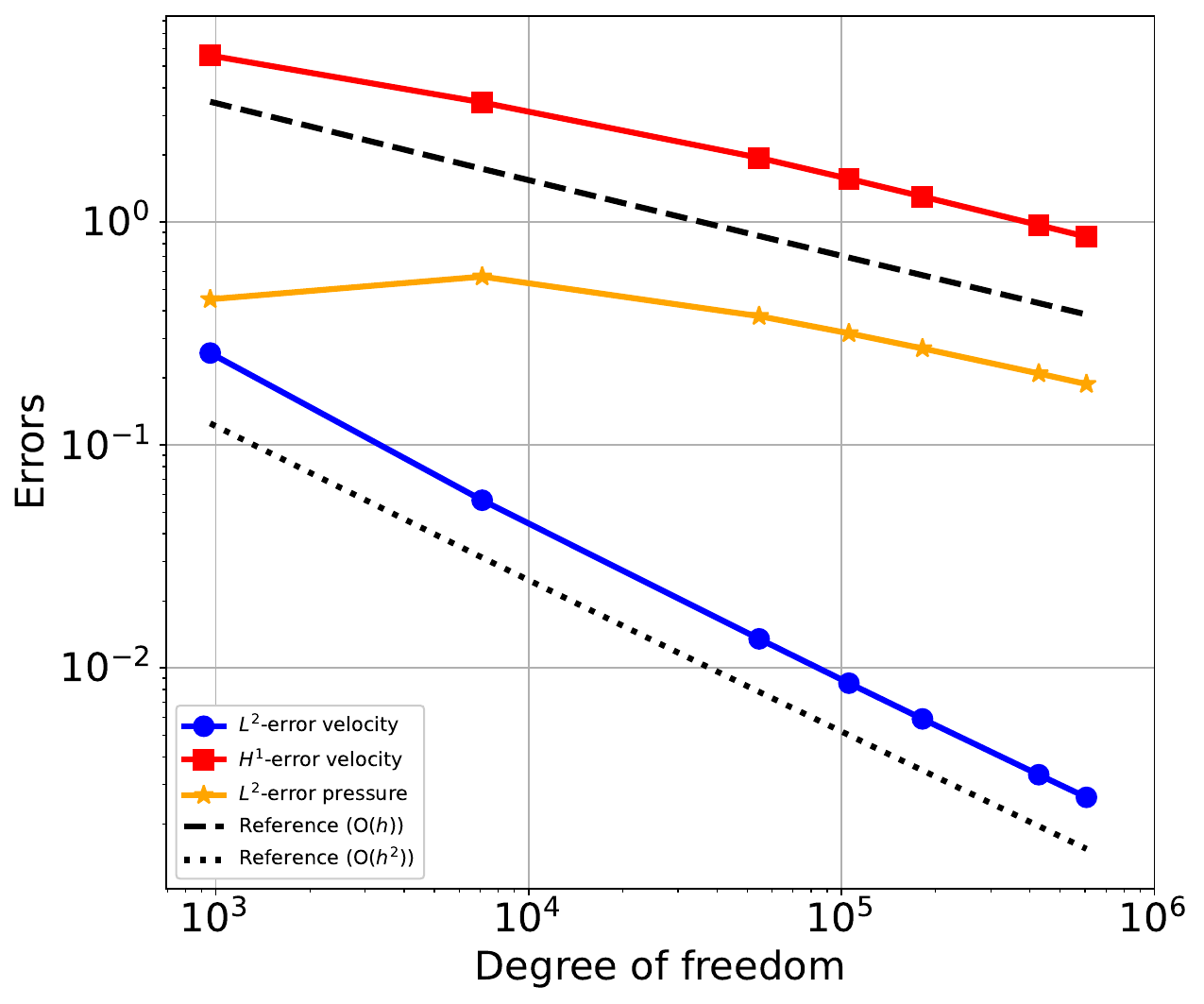} 
	\end{center} \vspace{-3mm}
	\caption{Convergence plots for the state and co-state variables (uniform refinement) for Example- \ref{Example 6.5.}.} \label{FIGURE 15}
\end{figure} 
\begin{figure}\vspace{-3mm}
	\begin{center}
		\includegraphics[scale=0.35]{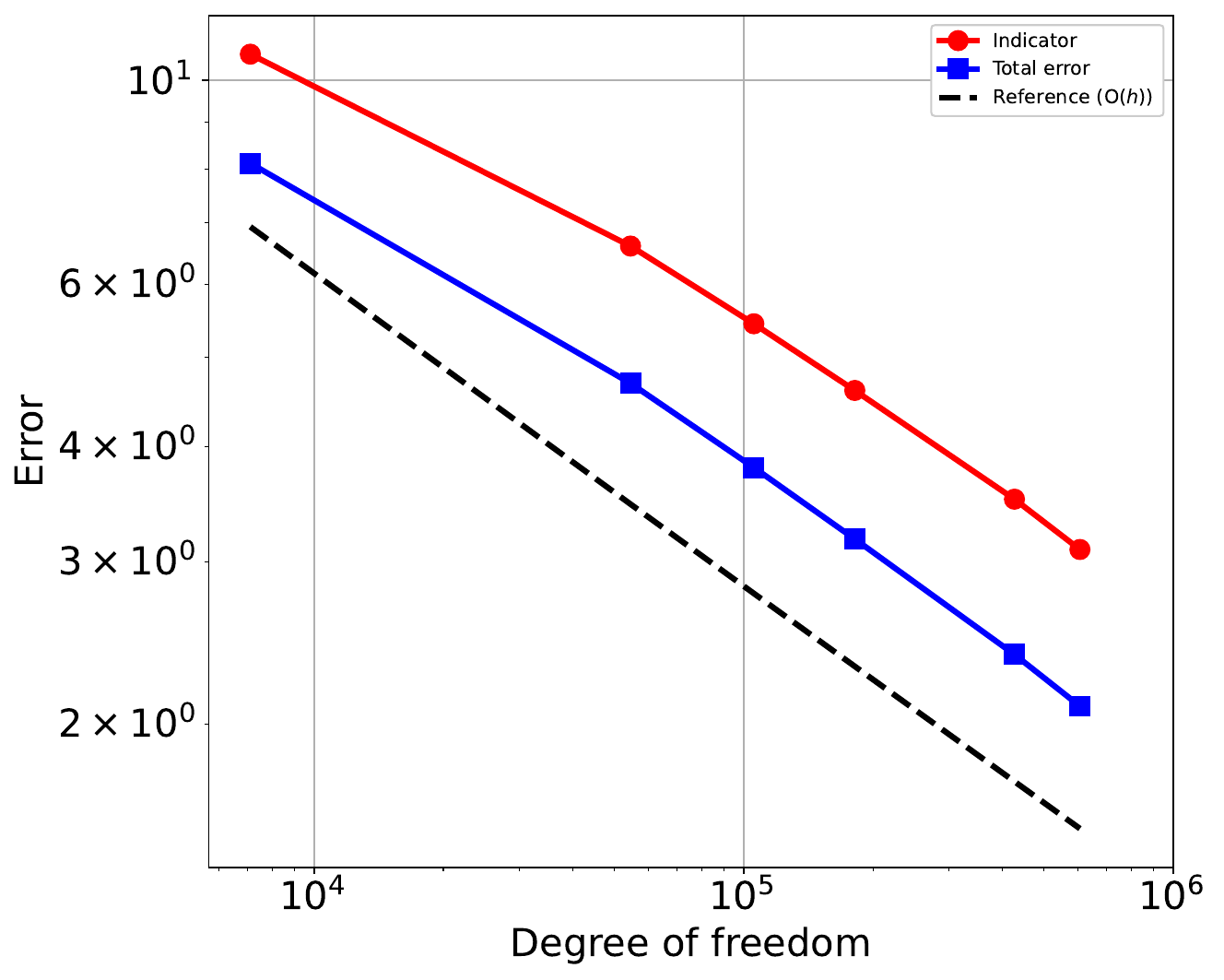}
		\includegraphics[scale=0.35]{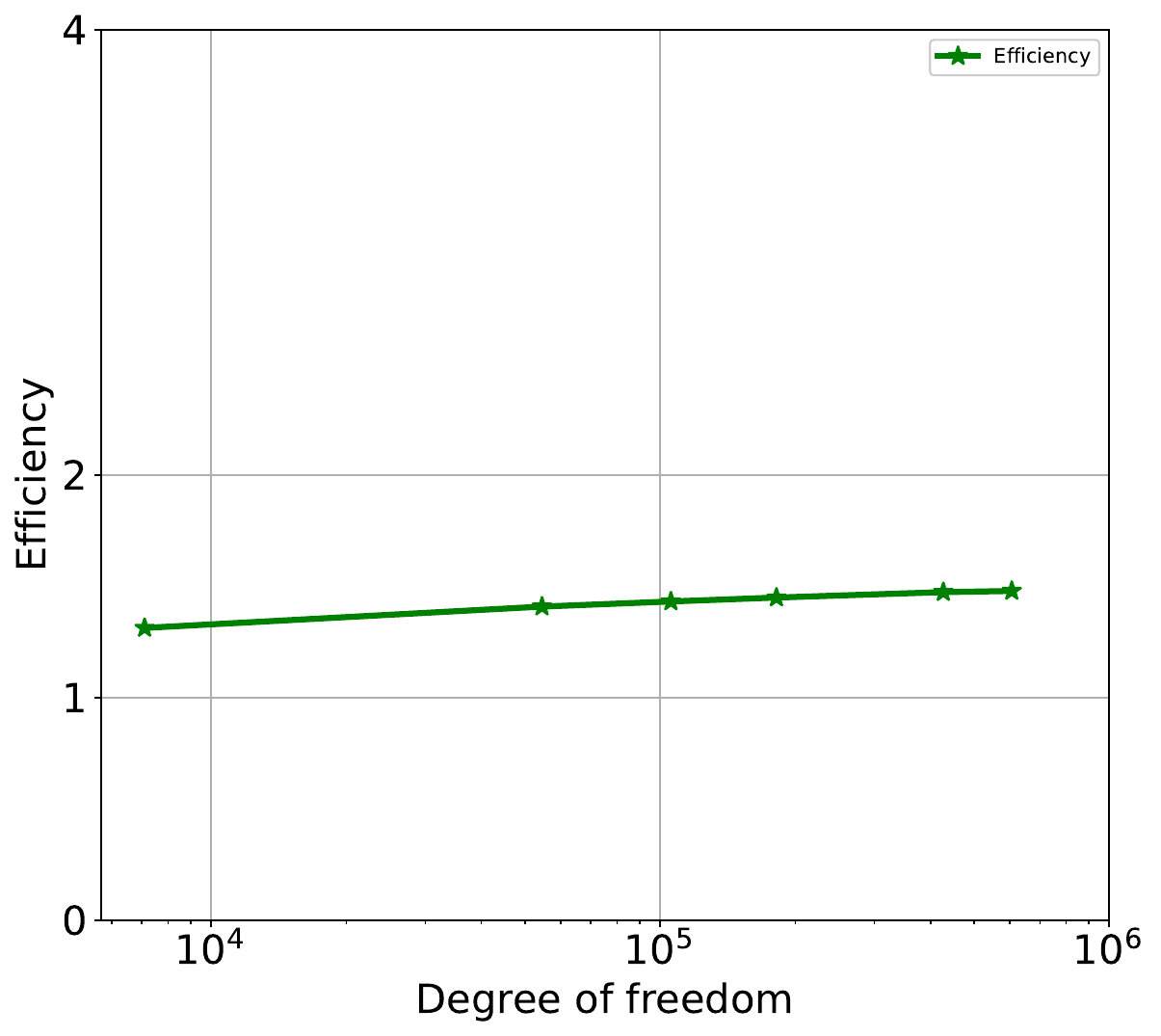}
	\end{center}\vspace{-3mm}
	\caption{Convergence plot for the indicator and total error (uniform refinement) for Example- \ref{Example 6.5.}.} \label{FIGURE 16}
\end{figure}
\section{Conclusion}
In this paper, we have derived a priori error estimates in $L^{2}$ and energy norms using the divergence-conforming DG finite element methods for distributed optimal control problems governed by generalized Oseen equations. Additionally, we established reliable and efficient a posteriori error estimators. Numerical experiments showcase the efficacy of the a posteriori error estimator, validating its performance for both convex and non-convex domains. This work contributes valuable insights into the accuracy and applicability of DG methods in optimizing systems described by the Oseen equations, offering a comprehensive analysis of error estimation techniques for practitioners in numerical optimization and CFD.\\
\textbf{Data availability:}
Upon reasonable request, datasets generated during the research discussed in the paper will be made available by contacting the corresponding author.

\textbf{Declarations:}
The authors affirm that they do not have any conflicts of interest.
	
\bibliographystyle{siam}
\bibliography{reference_arxiv}
\end{document}